\documentclass[preprint,reqno]{imsart}

\RequirePackage{amsthm,amsmath,amsfonts,amssymb}
\RequirePackage[numbers,sort]{natbib}
\RequirePackage{graphicx}%
\usepackage[shortlabels]{enumitem}

\usepackage{bbm}
\usepackage{amsmath}
\usepackage{comment}
\usepackage{hyperref}
\usepackage{subcaption}
\usepackage[noabbrev,capitalize]{cleveref}
\usepackage{booktabs,multirow}

\usepackage[margin=1.2in]{geometry}
\allowdisplaybreaks

\startlocaldefs
\newtheorem{thm}{Theorem}[section]
\newtheorem{cor}[thm]{Corollary}
\newtheorem{defi}{Definition}[section]
\newtheorem{assume}{Assumption}[section]

\newtheorem{lemma}[thm]{Lemma}
\newtheorem{ex}{Example}[section]
\newtheorem{remark}{Remark}[section]

\newcommand{\ttheta}{\boldsymbol{\theta}}
\newcommand{\I}{\mathbf{I}_{\mathnormal{d}}}

\DeclareMathOperator{\bu}{\textbf{u}}
\DeclareMathOperator{\bw}{\textbf{w}}
\DeclareMathOperator{\bh}{\textbf{h}}
\DeclareMathOperator{\bx}{\textbf{x}}
\DeclareMathOperator{\bz}{\textbf{z}}
\DeclareMathOperator{\w}{{\mathnormal \mathbf{w}}}
\DeclareMathOperator{\ft}{\tilde{\mathnormal f}}
\DeclareMathOperator{\QQCW}{\QQ_{\ttheta}^{\text{CW}}}

\newcommand{\thetahat}{\hat{\ttheta}_{\mathnormal n}}

\DeclareMathOperator{\xp}{\xrightarrow{\mathnormal p}}
\DeclareMathOperator{\xd}{\xrightarrow{\mathnormal d}}

\newcommand{\an}{\alpha_{\mathnormal n}}

\DeclareMathOperator{\ggamma}{\boldsymbol{\gamma}}
\DeclareMathOperator{\ssigma}{\boldsymbol{\sigma}}
\DeclareMathOperator{\ddelta}{\boldsymbol{\delta}}

\DeclareMathOperator{\sech}{sech}
\DeclareMathOperator{\EE}{\mathbb{E}}
\DeclareMathOperator{\Var}{Var}
\DeclareMathOperator{\Cov}{Cov}
\DeclareMathOperator{\PP}{\mathbb{P}}
\DeclareMathOperator{\PPB}{\bar{\mathbb{P}}}
\DeclareMathOperator{\QQ}{\mathbb{Q}}

\newcommand{\A}{\mathbf{A}}
\newcommand{\W}{\mathbf{W}}

\newcommand{\x}{\mathbf{x}}
\newcommand{\Z}{\mathbf{Z}}
\newcommand{\X}{\mathbf{X}}

\newcommand{\xxi}{\boldsymbol \xi}

\DeclareMathOperator{\sumin}{\sum_{\mathnormal i=1}^{\mathnormal n}}
\DeclareMathOperator{\sumjn}{\sum_{\mathnormal j=1}^{\mathnormal n}}

\DeclareMathOperator{\sumij}{\sum_{\mathnormal{ i,j}=1}^{\mathnormal n}}
\newcommand{\thetaiid}{\hat{\boldsymbol \theta}^{\text{iid}}_{\mathnormal n}}
\newcommand{\thetamle}{\hat{\boldsymbol \theta}^{\text{MLE}}_{\mathnormal n}}
\newcommand{\thetamf}{\hat{\boldsymbol \theta}^{\text{MF}}_{\mathnormal n}}
\newcommand{\thetaamle}{\hat{\boldsymbol \theta}^{\text{aMLE}}_{\mathnormal n}}

\DeclareMathOperator*{\argmax}{arg\,max}
\DeclareMathOperator*{\argmin}{arg\,min}
\endlocaldefs

\begin{document}

\begin{frontmatter}

\title{Inference on Gaussian mixture models with dependent labels}
\runtitle{Dependent Gaussian mixture models}
\runauthor{Lee, Mukherjee, Mukherjee}

\begin{aug}

\author[A]{\fnms{Seunghyun}~\snm{Lee}\ead[label=e1]{sl4963@columbia.edu}},
\author[B]{\fnms{Rajarshi}~\snm{Mukherjee}\ead[label=e2]{ram521@mail.harvard.edu}}
\and
\author[A]{\fnms{Sumit}~\snm{Mukherjee}\ead[label=e3]{sm3949@columbia.edu}}

\address[A]{Department of Statistics, Columbia University \printead[presep={,\ }]{e1,e3}}

\address[B]{Department of Biostatistics, Harvard University\printead[presep={,\ }]{e2}}
\end{aug}

\begin{abstract}
    Gaussian mixture models are widely used to model data generated from multiple latent sources. Despite its popularity, most theoretical research assumes that the labels are either independent and identically distributed, or follows a Markov chain.  It remains unclear how the fundamental limits of estimation change under more complex dependence. In this paper, we address this question for the spherical two-component Gaussian mixture model.
    We first show that for labels with an arbitrary dependence, a naive estimator based on the misspecified likelihood is $\sqrt{n}$-consistent. Additionally, under labels that follow an Ising model, we establish the information theoretic limitations for estimation, and discover an interesting phase transition as dependence becomes stronger. %
    When the dependence is smaller than a threshold, the optimal estimator and its limiting variance exactly matches the independent case, for a wide class of Ising models. %
    On the other hand, under stronger dependence, estimation becomes easier and the naive estimator is no longer optimal. Hence, we propose an alternative estimator based on the variational approximation of the likelihood, and argue its optimality under a specific Ising model. %
\end{abstract}



\end{frontmatter}

\section{Introduction}
Inference under the presence of latent mixing variables is a classical research area that remains highly relevant in modern statistical paradigms. %
In the most general setting, an investigator observes %
some variables of primary interest  -- where the observations are \textit{conditionally independent on some unobserved latent variables.}  %
Owing to both the theoretical and computational challenges that arise due to the hidden nature of the latent variables, significant research has been devoted to addressing how to learn the conditional %
distribution of the observed data, among other things.  The subtlety of the problem deepens when the hidden variables display dependence.
A growing body of research has made substantive progress in this regard by developing scalable methods under dependent models such as Hidden Markov Models (HMM) and Hidden Markov Random Fields (HMRF). For both HMMs and HMRFs and other related models of study, the focus mostly has been distributed across both statistical and computational efficiency considerations. However, unlike classical mixture models for independent hidden mixing variables, theoretical explorations for dependent latent variables is somewhat limited %
to HMMs. In this paper, we take the first steps to fill this gap by initiating a study of the two-class symmetric Gaussian mixture model with dependent mixing labels, and developing a theory of optimal inference therein.

\subsection{Problem formulation and challenges} 
We consider observing $d$-dimensional random vectors $\X_1,\ldots,\X_n$ generated from latent labels $ \Z^n := (Z_1, \ldots, Z_n)\in \{-1,+1\}^n$, as follows:
\begin{equation}\label{eq:model}
\begin{aligned}
    \Z^n := (Z_1, \ldots, Z_n) \sim \QQ_{0}, \quad \X_i \mid \Z^n \equiv \X_i \mid Z_i &\overset{\text{ind}}{\sim} N_d(\ttheta Z_i, \I), \quad i = 1, \ldots, n.
\end{aligned}
\end{equation}
This paper's primary goal is optimal estimation of the %
mean parameter $\ttheta \in \Theta := \mathbb{R}^d \setminus \{\mathbf{0}\}$, and how this is affected by $\mathbb{Q}_0$. 
To begin, note that the distribution of $\X$ under $\ttheta$ and $-\ttheta$ are the same. %
To ensure identifiability, we assume that the true parameter $\ttheta_0$ lives in the half-space $$\Theta_1 := \{\ttheta: \theta_1 > 0\} \cup \{\theta_1 = 0, \theta_2 > 0 \} \cup \cdots \cup \{\theta_1 = 0, \ldots, \theta_{d-1} = 0, \theta_d > 0 \}.$$ We also define $\Theta_2 = - \Theta_1$, so that $\Theta_1 \cup \Theta_2 = \Theta$.
In particular when $d = 1$, $\Theta_1$ is simply $\{\theta: \theta> 0 \}$.
If $\QQ_{0}$ represents a $n$-fold product measure on $\{-1,1\}^n$, the model reduces to the classical symmetric two-class isotropic Gaussian mixtures problem. Even this simple model has served as the basis for understanding several statistical challenges in unsupervised learning \citep[][and the references therein]{daskalakis2017ten, balakrishnan2017statistical, wu2020optimal, wu2021randomly, ndaoud2022sharp}. Interestingly, as these literature suggests, 
a complete understanding of even this model can be subtle from both theoretical and algorithmic perspectives, and has therefore attracted the keen attention of researchers across several quantitative domains. %
However, a parallel theory for more general $\QQ_0$ remains lacking. 

A natural class of problems that have evolved to extend this domain pertains to a specific class of $\QQ_0$ arising in the context of Markov Random Fields (MRF) \citep{clifford1971markov, besag1974spatial}. When $\QQ_0$ corresponds to a MRF on a given network, model \eqref{eq:model} is known in the literature as the Hidden Markov Random Field (HMRF) \citep{kunsch1995hidden, besag1986statistical} and a parallel literature have enriched the methodological arsenal for inference in HMRFs. However, to the best of our knowledge, rigorous theoretical guarantees or issues of statistical efficiency are yet to be thoroughly explored. In this paper, we take one of the first rigorous steps to quantify efficient statistical estimation of $\ttheta_0$ under some mean-field type HMRFs. %
As we will see below, the rate of estimation of $\ttheta_0$ is not affected by the choice of $\QQ_0$, whereas the efficient information bound for estimating $\ttheta_0$ is. To illustrate this, we focus on the case where $\QQ_0$ is an Ising model on a dense graph, and 
establish efficiency theory under various regimes of dependence. %
We provide a brief summary of these results below.

\subsection{Summary of results}
We develop a statistical theory for efficient estimation in model \eqref{eq:model}, under various types of label dependence.
We present our main contributions in three subsections: Sections \ref{sec:iid estimator}, \ref{subsec:high temperature cw}, and \ref{subsec:low temperature cw}. In the following, we summarize our main results.

\begin{figure}[h!]
    \centering
    \includegraphics[width = 0.9\textwidth]{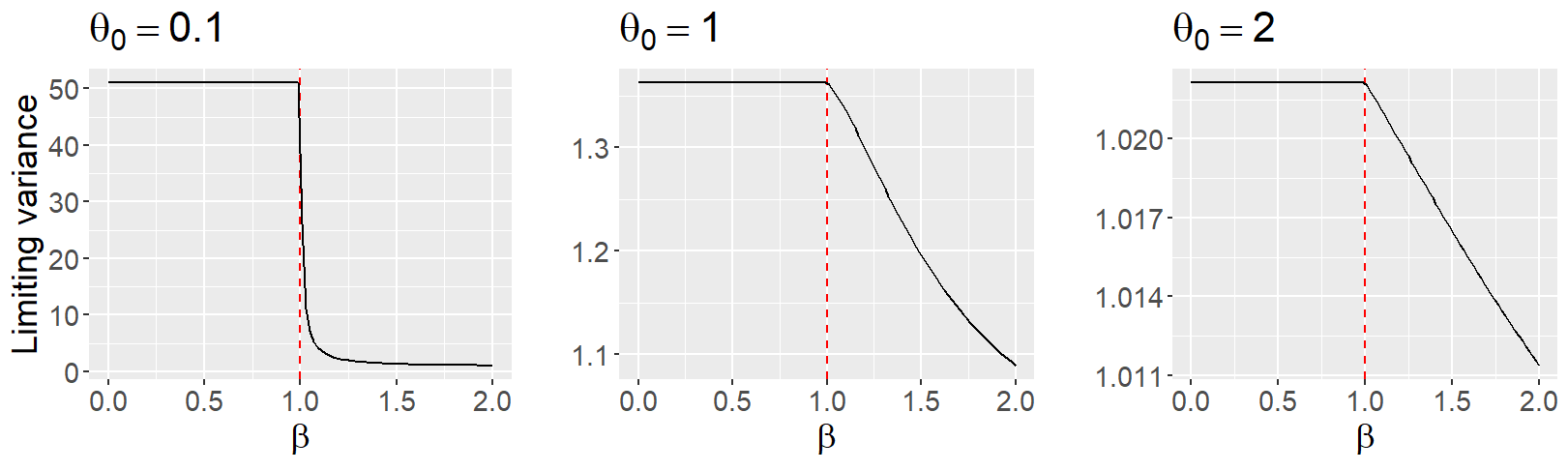}
    \caption{Plot of the (scaled) optimal limiting variance with respect to the dependence parameter $\beta \in [0,2]$, under Curie-Weiss labels. The hardness of estimation changes at $\beta = 1$, regardless of the true parameter $\ttheta_0$. 
    note that the scale of the $y$-axis is different for each panel.}
    \label{fig:info_wrt_beta}
\end{figure}

\begin{itemize}[itemsep = 0.5em]
    \item In \cref{sec:iid estimator}, we show that there exists an estimator that is $\sqrt{n}$-consistent with the same limiting distribution for \emph{any} label distribution $\QQ_0$. Surprisingly, the estimator we consider is the MLE computed under iid labels, which we denote as $\thetaiid$. In other words, the estimator under the misspecified likelihood attains the usual parametric (and optimal) rate for estimation. Additionally, we argue that $\thetaiid$ can be easily computed by an EM algorithm resulting from the misspecified likelihood.
    
    \item In \cref{sec:Curie-Weiss}, we assume a specific dependent parametrization for the labels and analyze the information-theoretic optimal limiting variance. We consider the Ising model to model the dependent labels. The Ising model is a popular Markov random field that flexibly handles network-type dependencies. This model has a parameter $\beta \ge 0$ that reflects the strength of dependence; $\beta = 0$ corresponds to the iid distribution, and a larger $\beta$ leads to stronger dependence. In \cref{fig:info_wrt_beta}, we plot the optimal limiting variance with respect to $\beta$ under the Curie-Weiss version of the Ising model (formally defined in eq. \eqref{eq:CW pmf}). Compared to iid labels, estimation becomes easier under \emph{strong dependence} ($\beta > 1$, see \cref{subsec:low temperature cw}), but there is no improvement under \emph{weak dependence} ($\beta \le 1$, see \cref{subsec:high temperature cw}). %
    In the following bullet points, we separate the two regimes and elaborate on tractable alternatives to the MLE that still attain the information-theoretic variance.
    
    \item Under weak dependence, we show that the misspecified MLE $\thetaiid$ is optimal. 
    This claim holds for a large class of Ising models on ``mean-field'' graphs with the maximum degree larger than $\sqrt{n \log n}$ (see \cref{defi:mean-field} for the precise condition). Thus, for dependent labels under weak dependence, the fundamental limit of estimation remains the same as that under iid labels, and one can even perform inference without any cost by blindly assuming iid labels.
    
    \item Under strong dependence, $\thetaiid$ is no longer optimal, and we propose a more efficient estimator $\thetamf$ based on the variational approximation of the marginal likelihood.
    When the underlying Ising model is mean-field and satisfies some additional conditions such as regularity (see \cref{defi:regularity}), $\thetamf$ is asymptotically normal with a strictly less variance compared to $\thetaiid$.
    However, due to technical reasons, we prove the optimality of $\thetamf$ only for Curie-Weiss labels.

    \item We also summarize properties of the estimators $\thetaiid$ and $\thetamf$ in \cref{tab:estimators summary}.
\end{itemize}

\begin{table}[h!]
\centering
\caption{Summary of the properties of estimators under various label dependencies. MF denotes ``mean-field'' Ising models (see \cref{defi:mean-field}) and CW denotes the Curie-Weiss model (see eq. \eqref{eq:CW pmf}).}
\begin{tabular}{c|cccc}
\toprule
\multirow{2}{*}{Estimator $\setminus$ Distribution $\QQ_0$} & \multirow{2}{*}{arbitrary} & \multicolumn{3}{c}{Ising model}
\\ \cmidrule{3-5}
                           &                                          & MF, $\beta <1$ or CW, $\beta \le 1$ & MF+regular, $\beta > 1$ &  CW, $\beta > 1$ \\
    \midrule
        $\thetaiid$ & $\sqrt{n}$-consistent & optimal & not optimal & not optimal \\
        $\thetamf$ & not defined & optimal & better than $\thetaiid$ & optimal \\
        \bottomrule
\end{tabular}
    \label{tab:estimators summary}
\end{table}

\subsection{Notations}
We use the following notations in the remainder of the paper. First, we use bold capital letters (e.g. $\X$) to denote matrices and random vectors, bold lower-case letters (e.g. $\mathbf{x}$) to denote deterministic vectors, and non-bold letters to denote scalars (e.g. $X, x$). The symbols 
$\|\cdot \|$ and $ \|\cdot\|_\infty$ denotes the $L^2$ and $L^\infty$ norm for a vector/matrix, respectively. For two symmetric $d\times d$ matrices $\mathbf{C}_d$ and $\mathbf{D}_d$, we write $\mathbf{C}_d \succ \mathbf{D}_d$ and $\mathbf{C}_d \succeq \mathbf{D}_d$ when $\mathbf{C}_d - \mathbf{D}_d$ is positive definite and positive semi-definite, respectively. Let $\mathbf{0}_d, \I$ denote the $d$-dimensional zero-vector and identity matrix, respectively. Let $\text{Rad}(p)$ denote the Radamacher distribution on $\{-1,1\}$ with probability of $1$ equal to $p$. For two probability measures $\PP,\QQ$, $\text{KL}(\PP \| \QQ)$ denotes the KL divergence of $\PP$ from $\QQ$. Also for any vector $\pmb{\upsilon}=(\upsilon_1,\ldots,\upsilon_k)^\top\in \mathbb{R}^k$ we will denote by $\bar{\pmb{\upsilon}}=\frac{1}{k}\sum_{j=1}^k\upsilon_j$.

As most results in this paper are asymptotic in $n$, we also introduce asymptotic notations. We use the standard Bachmann-Landau notations $o(\cdot), O(\cdot)$ for deterministic sequences. The symbols $\xp$ and $\xd$ denote convergence in probability and in distribution, respectively.
For a sequence of random variables $\{Y_n\}_{n \ge 1}$ and a deterministic positive sequence $\{a_n\}_{n \ge 1}$, we write $Y_n = o_p(a_n)$ when $\frac{Y_n}{a_n} \xp 0$, and $Y_n = O_p(a_n)$ when $\lim_{K \to \infty} \lim_{n \to \infty} \PP(\frac{|Y_n|}{a_n} \le K) = 1$, respectively. We also use the same asymptotic notations for finite-dimensional random vectors $\{\mathbf{Y}_n \}_{n \ge 1}$, by writing $\mathbf{Y}_n = o_p(a_n)$ and $\mathbf{Y}_n = O_p(a_n)$ when $\|\mathbf{Y}_n\| = o_p(a_n)$ and $\|\mathbf{Y}_n\| = O_p(a_n)$, respectively.

\section{Main results}
\cref{sec:iid estimator} shows that parametric rate-optimal estimation is possible for any dependence $\QQ_0$. Next, \cref{sec:Curie-Weiss} considers Ising model labels and propose information-theoretic limits and optimal estimators. Throughout the paper, let $P_{\ttheta_0, \QQ_{0}} = P_{\ttheta_0, \QQ_0}^{(n)}$ be the distribution of $\X^n$ defined in \eqref{eq:model}, under the true parameter $\ttheta_0 \in \Theta_1$ and label distribution $\QQ_{0}$. %

\subsection{Universal $\sqrt{n}$-consistent estimation}\label{sec:iid estimator}
We first gather some intuition of the problem from studying the i.i.d. label version of the problem, i.e., when $\QQ_0$ is a product measure. %
Indeed then, \eqref{eq:model} reduces to the classical symmetric isotropic Gaussian mixture problem -- a research area that has continued to witness repeated interest from the quantitative research community as a fundamental object of study in statistics. Specifically with iid labels $Z_i \overset{\text{iid}}{\sim} \text{Rad}(0.5)$, after marginalizing out the label $Z_i$'s, traditional asymptotic theory shows that the maximum likelihood estimator 
\begin{equation}\label{eq:iid estimator definition}
    \thetaiid := \argmin_{\ttheta \in \Theta_1} \left[\frac{\ttheta^\top \ttheta}{2} - \frac{1}{n} \sumin \log \cosh(\ttheta^\top \X_i) \right]
\end{equation}
is $\sqrt{n}$-consistent and asymptotically optimal in the sense of attaining the information theoretic lower bound. In terms of computation, %
it is well-known that the EM algorithm with a random initialization is guaranteed to converge to the MLE at a geometric rate \citep{xu2016global, daskalakis2017ten, klusowski2016statistical}.
However, the problem changes drastically when the labels are dependent. A faithful statistician would expect that the MLE 
\begin{equation}\label{eq:MLE definition}
    \hat{\ttheta}_n^{\text{MLE}} := \argmin_{\ttheta \in \Theta_1} \left[\frac{\ttheta^\top \ttheta}{2} - \frac{1}{n} \log \left( \sum_{\mathbf{z} \in \{-1,1\}^n} \QQ_0(\mathbf{z}) e^{\ttheta^\top \sumin \X_i z_i}\right) \right]
\end{equation}
will still be optimal. A further simplification of the summation inside the log in \eqref{eq:MLE definition} is impossible due to the arbitrary dependence within $\QQ_0$. The data $\X^n$ also becomes dependent, breaking down the classical theory. Consequently, analyzing the MLE and understanding the informational theoretic lower bound becomes nontrivial. In terms of computation, the EM algorithm slows down significantly as each E-step involves summing over $2^n$ terms, and global convergence is yet to be studied. 

We tackle these issues below by considering the naive estimator $\thetaiid$ and show that it has a limiting distribution that does not depend on $\QQ_0$.  Suppose that $\Z^n \sim \QQ_0$ is \emph{arbitrarily} distributed on $\{-1, 1\}^n$, and observe $\X^n \sim P_{\ttheta_0, \QQ_0}$ for some true parameter $\ttheta_0 \in \Theta_1$. %

To simplify notations, let 
$$N_n(\ttheta) := \frac{\ttheta^\top \ttheta}{2} - \frac{1}{n} \sumin \log \cosh( \ttheta^\top \X_i)$$
be the re-scaled negative log-likelihood under i.i.d labels $\Z^n$. Then, \eqref{eq:iid estimator definition} becomes $\thetaiid = \argmin_{\ttheta\in \Theta_1} N_n(\ttheta)$.
Also define
\begin{align*}
    N_{\infty}(\ttheta) &:= \frac{\ttheta^\top \ttheta}{2} - \EE_{\X\sim N_d(\ttheta_0, \I)} \log \cosh( \ttheta^\top \X),
\end{align*}
which is the weak limit of $N_{n}(\ttheta)$.
To see this, note that $\log \cosh$ is an even function, and consequently the distribution of $\log \cosh (\ttheta^\top \X_1) \mid Z_1$
is the same for $Z_1 = \pm 1$. Thus, by the conditional law of large numbers for independent random variables, 
\begin{align*}
    &\frac{1}{n} \sumin \left(\log \cosh(\ttheta^\top \X_i) - \EE\left[\log \cosh(\ttheta^\top \X_i) \mid Z_i \right] \right) \mid \Z^n \\
    =& \frac{1}{n} \sumin \log \cosh(\ttheta^\top \X_i) - \EE_{\X\sim N_d(\ttheta_0, \I)} \log \cosh(\ttheta^\top \X) \mid \Z^n \xp 0,
\end{align*}
and $N_n(\ttheta)$ converges to $N_{\infty}(\ttheta)$ in probability, regardless of the distribution of $\Z^n$. Note that the function $N_{\infty}$ also depend on the true parameter $\ttheta_0$, but we do not display this explicitly as $\ttheta_0$ is \emph{fixed} throughout. To understand why the minimizer of $N_n(\ttheta)$ is close to $\ttheta_0$, we present  the following Lemma to show that the limiting objective function $N_{\infty}$ is uniquely minimized at $\ttheta= \ttheta_0$. 

\begin{lemma}\label{lem:uniqueness iid}
    $N_{\infty}:\Theta_1 \to \mathbb R$ is differentiable in $\textnormal{int}(\Theta_1)$ and uniquely minimized at $\ttheta= \ttheta_0$. Furthermore, $\ttheta_0$ is the unique solution of $(\nabla N_{\infty})(\ttheta) = \mathbf{0}_d$ in $\textnormal{int}(\Theta_1)$.
\end{lemma}

Based on this insight, Theorem \ref{thm:upper bound iid} shows that $\thetaiid$ is $\sqrt{n}$-consistent with a label-independent Normal limit. Thus, $\thetaiid$, the naive estimator that arises from the misspecified likelihood with independent labels, is always rate-optimal\footnote{The rate-optimality follows by noting that the MLE converges at the same $\sqrt{n}$-rate when all labels $\Z^n$ are known, and it is impossible to do better with an unknown $\Z^n$.}.

\begin{thm}\label{thm:upper bound iid}
Let $\QQ_0$ be an arbitrary measure on $\{-1, 1\}^n$ and $\X^n \sim P_{\ttheta_0,\QQ_{0}}$. Then, for $I_0(\ttheta_0) := \I - \EE_{\X\sim N_d(\ttheta_0, \I)} \X\X^{\top} \sech^2(\ttheta_0^\top \X)$, we have
\begin{equation}\label{eq:high tmp iid estimator expansion}
    \sqrt{n}(\thetaiid - \ttheta_0) = I_0(\ttheta_0)^{-1} \frac{1}{\sqrt{n}} \sumin \left(\X_i \tanh(\ttheta_0^\top \X_i) - \ttheta_0 \right) + o_p(1)
\end{equation}
and
    \begin{align*}%
    \sqrt{n} (\thetaiid - \ttheta_0) \xrightarrow{d} N_d\left(0, I_0(\ttheta_0)^{-1} \right) .
    \end{align*}
\end{thm}

The proofs of Lemma \ref{lem:uniqueness iid} and Theorem \ref{thm:upper bound iid} are deferred to \cref{sec:proof iid}. 

\begin{remark}[Computing the estimator]
    In the proof of Lemma \ref{lem:uniqueness iid}, we use the fact from \cite{daskalakis2017ten} that the mapping $T(\ttheta) := \EE_{\X\sim N_d(\ttheta_0, \I)} \X \tanh(\ttheta^\top \X)$ satisfies $T(\ttheta_0) = \ttheta_0$ and 
    $$\|T^{(t)}(\ttheta) - T^{(t)}(\ttheta_0)\| \le \kappa(\ttheta)^t \|\ttheta- \ttheta_0\|, \quad \forall t \ge 1,$$
    with $\kappa(\ttheta) := \exp\left[ - \frac{\min(\ttheta^\top \ttheta, \ttheta_0^\top \ttheta)^2}{2 \ttheta^\top \ttheta} \right] \le 1.$
    Thus, taking an arbitrary initial value $\ttheta^{(0)}\in \Theta_1$ and iteratively applying $T$ would converge to $\ttheta_0$ at an geometric rate, as long as $(\ttheta^{(0)})^\top \ttheta_0 \neq 0$. Note that $T$ can also be viewed as one iteration of the population EM algorithm for the usual symmetric GMMs with independent labels (e.g. see eq (2) in \cite{daskalakis2017ten}). Based on this global convergence guarantee, one can compute $\thetaiid$ using the sample-based EM algorithm with a random initialization $\ttheta^{(0)}$, which iteratively computes
    $$\ttheta^{(t+1)} := \frac{1}{n} \sumin \X_i \tanh(\ttheta^{(t) \top} \X_i).$$
\end{remark}

\subsection{Efficient estimation under Ising model dependence}\label{sec:Curie-Weiss}
Given the $\sqrt{n}$-consistency of $\thetaiid$, we further assess its optimality in terms of its limiting variance. It turns out that such an efficiency theory depends on models assumed on the labels. We demonstrate such a theory under Ising models for the hidden labels. To that end, 
we first formally introducing Ising models for the joint distribution of $\mathbf{Z}^n$ (\cref{subsec:hidden_ising_model}), and discuss related challenges and estimation strategies. Subsequently, we separate the argument by considering two regimes for the ``temperature'' parameter $\beta$: \emph{high/critical temperature regime} with $\beta \le 1$ (\cref{subsec:high temperature cw}) and \emph{low temperature regime} with $\beta > 1$ (\cref{subsec:low temperature cw}). 

\vspace{2mm}
\noindent 
\subsubsection{Inference under Hidden Ising models}\label{subsec:hidden_ising_model}
The Ising model, originally proposed in statistical physics to explain ferromagnetism \citep{ising1924beitrag}, is defined as follows.

\begin{defi}[Ising model]\label{def:ising}
    Let $\A_n$ be a nonnegative and symmetric $n \times n$ coupling matrix with empty diagonals.
    For $\beta \ge 0$, the Ising model $\QQ_{0, \beta, \A_n}$ is a probability measure on $\{-1, 1\}^n$ for $n\geq 1$ with probability mass function
    \begin{align*}
        \QQ_{0, \beta, \A_n} (\Z^n = \bz) \propto e^{\frac{\beta}{2} \bz^{\top} \mathbf{A}_n \bz}, ~ \textnormal{for all } \mathbf{z} \in \{-1,1\}^n.
    \end{align*}
\end{defi}

Here, the coupling matrix $\A_n$ governs the dependence structure of $\Z^n$. When a network on the $n$ data points is given, $\A_n$ can be defined as its scaled adjacency matrix, so that vertices sharing an edge are more likely to have same labels.
Also, $\beta \ge 0$ is a parameter representing the magnitude of dependence, commonly referred to as the ``inverse temperature'' parameter in the statistical physics literature. In particular, for $\beta = 0$, the Ising model $\QQ_{0, \beta, \A_n}$ simply becomes the iid measure.

Throughout this section, $\beta$ and $\A_n$ are \emph{known and fixed}, so we simplify $\QQ_0 = \QQ_{0, \beta,\A_n}$ when the context is clear. Since we consider an asymptotic setting with a growing $n$, consider a sequence of $n \times n$ coupling matrices $\{\A_n\}_{n \ge 1}$. Additionally, assume that the coupling matrices are scaled in a manner such that the {maximum row sum is 1}, i.e. 
\begin{equation}\label{eq:row sum one scaling}
    \lim_{n \to \infty} \lVert \A_n \rVert_{\infty} =1.
\end{equation}
The exact assumptions on $\A_n$ vary across different results, and additional assumptions are imposed along the way. We provide a classical and well studied example below.

\begin{ex}[Curie-Weiss model]\label{ex:Curie-Weiss}
One important example is when $\A_n$ is the scaled adjacency matrix of a complete graph with $A_n(i,j) = \frac{1}{n}\mathbf{1}(i\neq j)$, which we denote as the Curie-Weiss model $\QQ_{0,\beta}^{\text{CW}}$. The Curie-Weiss model has been popular for modeling dependent binary data, due to its exchangeability and low-rank nature \citep{ellis1978statistics, comets1991asymptotics, mukherjee2018global}. 
For future convenience, we spell out the pmf of the Curie-Weiss model:
\begin{align}
    \QQ_{0,\beta}^{\text{CW}}(\Z^n = \bz) \propto e^{\frac{n \beta \bar{z}^2}{2}} ~ \mbox{for all } \bz \in \{-1, 1\}^n,\label{eq:CW pmf}
\end{align}
and let $P_{\ttheta_0, \beta}^{\text{CW}}$ be the distribution of $\X^n$ under Curie-Weiss labels $\QQ_{0,\beta}^{\text{CW}}$.
\end{ex}

As the Ising model $\QQ_0$ determines the true labels, it is crucial to understand its properties. One statistic of interest is the sample mean $\bar{Z}$, which determines the proportion of label $Z_i$'s equal to 1. 
Under certain assumptions on $\A_n$ (see Definitions \ref{defi:mean-field} and \ref{defi:regularity}), it is known that the limiting behavior of $\bar{Z}$ exhibits a phase transition as it concentrates around 0 when $\beta \le 1$, and around $\pm m$ when $\beta > 1$ \citep{ellis1978statistics, deb2023fluctuations}. Here, $m = m(\beta) > 0$ is defined as the unique positive root of $m = \tanh(\beta m)$. %
Thus, when $\beta \le 1$, the labels roughly have equal proportions. However, when $\beta > 1$, for each configuration, one label is more likely than the other (with probability $\frac{1+m}{2}$ and $\frac{1-m}{2}$, respectively). This motivates why we need to consider the two regimes separately.
\vspace{2mm}
\noindent \emph{Likelihood under Ising labels.} Our main ingredient for proving subsequent results under the Ising labels $\mathbf{Z}^n\sim  \QQ_{0, \beta, \A_n}$ is to understand the corresponding normalizing constant in \eqref{eq:MLE definition} as the normalizing constant of a ``random field Ising model''. Specifically, define $\QQ_{\ttheta} = \QQ_{\ttheta,\beta, \A_n,\X^n}$ as a measure on $\{-1, 1\}^n$ conditioned on the data $\X^n \sim P_{\ttheta_0, \QQ_{0,\beta,\A_n}}$ with pmf
\begin{equation}\label{eq:rfim def}
    \QQ_{\ttheta} (\bw) = \QQ_{\ttheta,\beta, \A_n,\X^n}(\bw) := \frac{e^{\frac{\beta}{2} \bw^\top \A_n \bw + \ttheta^\top \sumin \X_i w_i }}{Z_{n,\beta,\A_n}(\ttheta, \X^n)} ~ \text{for all} ~ \bw \in \{-1,1\}^n,
\end{equation}
where $$Z_{n,\beta,\A_n}(\ttheta, \X^n) := \sum_{\mathbf{w} \in \{-1,1\}^n} e^{\frac{\beta}{2} \mathbf{w}^{\top} \A_n \mathbf{w} + \ttheta^\top \sumin \X_i w_i}$$ is the normalizing constant/partition function. It is easy to see that $\QQ_{\ttheta}$ is the ``posterior'' distribution of the labels after observing $\X^n$ and assuming the knowledge of $\ttheta$. 
It is interesting that $\QQ_{\ttheta}$ can be viewed as a \emph{random field Ising model} (RFIM) from statistical physics, where the additional linear term $\sumin (\ttheta^\top \X_i) w_i$ (compared to the true label distribution $\QQ_{0,\beta,\A_n}$) correspond to the ``random fields''.
Note that we use the notation $\mathbf{w}/\W$ to denote realizations and samples under the RFIM $\W^n \sim \QQ_{\ttheta}$, and $\mathbf{z}/\Z$ for that under the true label distribution $\Z^n \sim \QQ_0$.
Also, note that the newly defined $\QQ_{\ttheta}$ is consistent with the previous notation $\QQ_0$ (see \cref{def:ising}) in the sense that $\QQ_{\ttheta} = \QQ_0$ for $\ttheta =\mathbf{0}_d$.

With these notations, %
the first order conditions of the minimization in \eqref{eq:MLE definition} can be written as
\begin{equation}\label{eq:MLE first order conditions}
    \thetamle = \frac{1}{n} \sumin \X_i \EE^{\QQ_{\thetamle}} (W_i : \X^n).
\end{equation}
Above by $\EE^{\QQ_{\ttheta}}$ corresponds to the expectation under the distribution $\QQ_{\ttheta} (\bw)$ introduced in \eqref{eq:rfim def} above.   Hence, to understand the asymptotics of the MLE, it is crucial to have a precise understanding of the RHS of \eqref{eq:MLE first order conditions}. In particular, we claim there exists a value $u_n(\beta, \X^n)$ such that for $\ttheta\approx \ttheta_0$,
\begin{equation}\label{eq:sum xi wi expansion intro}
    \frac{1}{n} \sumin \X_i \EE^{\QQ_{\ttheta}} (W_i : \X^n) = \frac{1}{n} \sumin \X_i \tanh(u_n(\beta, \X^n) + \ttheta^\top \X_i) + o_p \left(\frac{1}{\sqrt{n}} : \X^n \right).
\end{equation}
This expansion is the main tool for all of our results, such as deriving the LAN expansion, and constructing a tractable estimator $\hat{\ttheta}$ by approximating $\thetamle \approx \hat{\ttheta}$ in \eqref{eq:MLE first order conditions}:
$$\hat{\ttheta} = \frac{1}{n} \sumin \X_i \tanh(u_n(\beta, \X^n) + \hat{\ttheta}^\top \X_i).$$
We expand on this heuristics in the next to subsections. 

\subsubsection{High/critical temperature regime $\beta \le 1$}\label{subsec:high temperature cw}
Recalling the limiting variance of $\thetaiid$ from \cref{thm:upper bound iid}, we now argue its optimality under a large class of Ising model distributions $\QQ_{0, \beta, \A_n}$.
In this section, our main assumptions for the Ising model components are that $\beta \le 1$ (high-temperature) and that $\A_n$ satisfies the following \emph{mean-field condition}.
\begin{assume}[mean-field condition]\label{defi:mean-field}
    We say that the sequence of coupling matrices $\{\A_n\}_{n \ge 1}$ satisfies the mean-field condition when 
    \begin{equation}\label{eq:mean field assumption}
    \alpha_n := \max_{i=1}^n \sum_{j=1}^n A_n(i,j)^2 =o \left( \frac{1}{\sqrt{n \log n}} \right).
    \end{equation}
\end{assume}

\noindent Condition \eqref{eq:mean field assumption} implies that the variational approximation of the log-partition function $\log Z_{n,\beta,\A_n}$ is tight up to the leading order \citep[][also see eq. \eqref{eq:gibbs variational principle} below]{basak2017universality}, and was used in \cite{lee2024rfim,lee2025clt} to derive tight concentration and limiting distributions on RFIMs.
For illustration, let $\mathbf{G}_n \in \{0,1\}^{n \times n}$ be the adjacency matrix of an undirected simple graph on the vertex set $V_n = \{1, \ldots, n\}$, and let $d_i$ be the degree of vertex $i$. Then, by defining $ \A_n := \frac{\mathbf{G}_n}{\max_{i=1}^n d_i}$, \eqref{eq:mean field assumption} is 
equivalent to ${\max_{i=1}^n d_i} \gg \sqrt{n \log n}$.

We prove the optimality of $\thetaiid$ in three steps. First, in Lemma \ref{lem:sum x_i z_i}, we prove a uniform version of the identity \eqref{eq:sum xi wi expansion intro}, with the centering $u_n(\beta, \X^n) = 0$. Next, in Theorem \ref{thm:lower bound high temperature}, we compute the LAN expansion of the likelihood ratio. Then, in Corollary \ref{cor:high tmp}, we use the LAN expansion and Le Cam theory to argue that $\thetaiid$ is optimal among the class of regular estimators. The proofs are mainly based on the concentration results for linear statistics of RFIMs developed in \cite{lee2024rfim}, and deferred to \cref{sec:proof high tmp}.

\begin{lemma}\label{lem:sum x_i z_i}
    Suppose that $\beta <1$, $\A_n$ satisfies the mean-field condition, and $\X^n \sim P_{\ttheta_0,\QQ_{0, \beta, \A_n}}$. Then,     \begin{equation}\label{eq:uniform bound for sum x_i z_i}
        \sup_{\ttheta\in \Theta} \Bigg\lVert \EE^{\QQ_{\ttheta}} \Big[\sumin \X_i W_i \Big] - \sumin \X_i \tanh(\ttheta^\top \X_i) \Bigg\rVert = o_p\left({\sqrt{n}} \right).
    \end{equation}
    Additionally, \eqref{eq:uniform bound for sum x_i z_i} holds under the Curie-Weiss label distribution $\QQ_{0,\beta}^{\text{CW}}$ at the critical temperature $\beta = 1$. %
\end{lemma}

\begin{remark}%
    The careful reader would have noticed that the first set of assumptions in \cref{lem:sum x_i z_i} does not allow $\beta = 1$, which is the critical temperature for Ising models on regular graphs \citep{deb2023fluctuations}. We believe that $\thetaiid$ would still be optimal at $\beta = 1$ as well, and in fact show such a result under the Curie Weiss model $\QQ_{0,\beta}^{\text{CW}}$. The main bottleneck of our proof is that we could only prove the RFIM moment bounds for $\beta < 1$. Actually, the RFIM $\QQ_{\ttheta, 1, \A_n,\X_n}$ with $\ttheta \neq \mathbf{0}_d$ is expected to exhibit a larger critical temperature $\beta_{crit}(\ttheta) := \frac{1}{\EE_{\X \sim N_d(\ttheta_0, \I)} \sech^2(\ttheta^\top X)} > 1$ \citep{he2023hidden}, which is why we expect that the moment bounds to be still true for $\beta = 1$.

    We additionally mention that \cref{lem:sum x_i z_i} holds even without the nonnegative assumption on the entries of $\A_n$ as long as $\beta < 1$ and \eqref{eq:row sum one scaling} holds.
\end{remark}

In \cref{thm:lower bound high temperature}, we assume eq. \eqref{eq:uniform bound for sum x_i z_i} and prove the LAN expansion of the likelihood (e.g. see Section 7 in \cite{van2000asymptotic}). 
Here, we do not require any specific property for the Ising label distribution beyond \eqref{eq:uniform bound for sum x_i z_i}.

\begin{thm}\label{thm:lower bound high temperature}
    Suppose \eqref{eq:uniform bound for sum x_i z_i} holds for an Ising model $\QQ_0 = \QQ_{0,\beta,\A_n}$, and $\X^n \sim P_{\ttheta_0,\QQ_{\ttheta}}$. For $\bh\in \mathbb{R}^d$, let $\ttheta_n := \ttheta_0+\frac{\bh}{\sqrt{n}}$. Then,
    $$\log \frac{dP_{\ttheta_n,\QQ_{0}}}{dP_{\ttheta_0, \QQ_{0}}} (\X^n) = \bh^\top \Delta_{n, \ttheta_0}(\X^n)- \frac{1}{2} \bh^\top I_0(\ttheta_0) \bh + o_p(1),$$
    where $I_{0}(\ttheta_0)$ is the value defined in \cref{thm:upper bound iid} and
    \begin{equation}\label{eq:delta_n}
        \Delta_{n, \ttheta_0}(\X^n) := \sqrt{n} \left(\frac{1}{n} \sumin \X_i \tanh(\ttheta_0^\top \X_i) - \ttheta_0 \right) \xrightarrow[P_{\ttheta_0, \QQ_{0}}]{d} N_d(0, I_{0}(\ttheta_0)).
    \end{equation}
    Hence, the family $\{P_{\ttheta,\QQ_{0}} \}_{\ttheta\in \Theta_1}$ is LAN with a precision matrix $I_0(\ttheta_0)$ at any $\ttheta_0 \in \Theta_1$.
\end{thm}

In the next corollary, we combine all previous results and prove that $\thetaiid$ is a regular estimator.
Then, by the convolution theorem (e.g. see Theorem 8.8 in \cite{van2000asymptotic}), $\thetaiid$ is optimal amongst all regular estimators in the sense that for other regular estimators with limiting variance $\Sigma_n$, we must have $\Sigma_n \succeq I_0(\ttheta_0)^{-1}$. Thus, $\thetaiid$ is optimal under Ising model labels that satisfy the assumptions in \cref{lem:sum x_i z_i}. In particular, $\thetaiid$ is optimal under Curie-Weiss labels $\QQ_{0,\beta}^{\text{CW}}$ with $\beta \le 1$, as illustrated by the straight line in \cref{fig:info_wrt_beta}.

\begin{cor}\label{cor:high tmp}
    Suppose \eqref{eq:uniform bound for sum x_i z_i} holds for some Ising model $\QQ_0$, and $\X^n \sim P_{\ttheta_0,\QQ_{0}}$.Then, $\thetaiid$ is a regular estimator, i.e. for any $\bh \in \mathbb{R}^d - \{0\}$ and $\ttheta_n := \ttheta_0 + \frac{\bh}{\sqrt{n}}$, we have $$\sqrt{n}(\thetaiid - \ttheta_n) \xrightarrow[P_{\ttheta_n,\QQ_{0}}]{d} N_d(0, I_0(\ttheta_0)^{-1}).$$
\end{cor}

Even though the main focus of this paper is on estimating $\ttheta_0$, the LAN expansion in Theorem \ref{thm:lower bound high temperature} can also be applied for testing.
\begin{remark}[Testing against contiguous alternatives]\label{rmk:testing high temperature}
    For $\ttheta_0 \in \Theta_1$, consider testing $H_0: \ttheta= \ttheta_0$ v.s. $H_1: \ttheta= \ttheta_0 + \frac{\bh}{\sqrt{n}}$ for any $\bh \neq \mathbf{0}$. Using the LAN expansion in Theorem \ref{thm:lower bound high temperature}, we can construct an asymptotically optimal test by rejecting the null when $\bh^\top \Delta_{n, \ttheta_0}(\X^n)$ is large. Note that we are considering $\ttheta_0 \in \Theta_1$ and do not allow $\ttheta_0 = \mathbf{0}_d$, which corresponds to testing the number of mixture components. Similar to the iid case \citep{goffinet1992testing}, we believe that the likelihood would not be LAN at $\ttheta_0 = \mathbf{0}_d$.
\end{remark}

We conclude this subsection with a discussion on  
the mean-field assumption \eqref{eq:mean field assumption}. We believe that the universal optimality of $\thetaiid$ heavily depends on 
the mean-field assumption \eqref{eq:mean field assumption}. For non-mean-field models that do not satisfy \eqref{eq:mean field assumption}, for example when $ \A_n$ is the adjacency matrix of a lattice, one would need an alternative approximation of the log normalizing constant in order to derive a result similar to \cref{lem:sum x_i z_i}. This itself is an open research question and the current results require restrictive assumptions on the boundary conditions of the lattice \citep{chatterjee2019central}. We provide a simple counterexample below and show that the university may fail when $ \A_n$ does not satisfy \eqref{eq:mean field assumption}.

\begin{ex}[Counterexample of the mean-field condition]\label{rem:counterex}
    Consider the case when $ \A_n$ is the scaled adjacency matrix of the graph with edges $\{1 \to 2, 3\to 4, \ldots , (2k-1) \to 2k, \ldots \}$. %
    Then, we have $\an = \Theta(1)$, so \eqref{eq:mean field assumption} does not hold. For this case, the pairs $(\X_{2k-1}, \X_{2k})$ are i.i.d and it is possible to directly compute the Fisher information for estimating $\ttheta_0$. In \cref{fig:counterexample}, we display the limiting variance of the MLE $\hat{\ttheta}_n^{\text{MLE}}$ and $\thetaiid$. We see that for all $\beta>0$, the MLE has a smaller variance, and $\thetaiid$ fails to be optimal. Note that this model does not have a phase transition in terms of $\beta$, and the low temperature regime does not exist.
\end{ex}

\begin{figure}[h!]
    \centering
    \includegraphics[width = 0.9\textwidth]{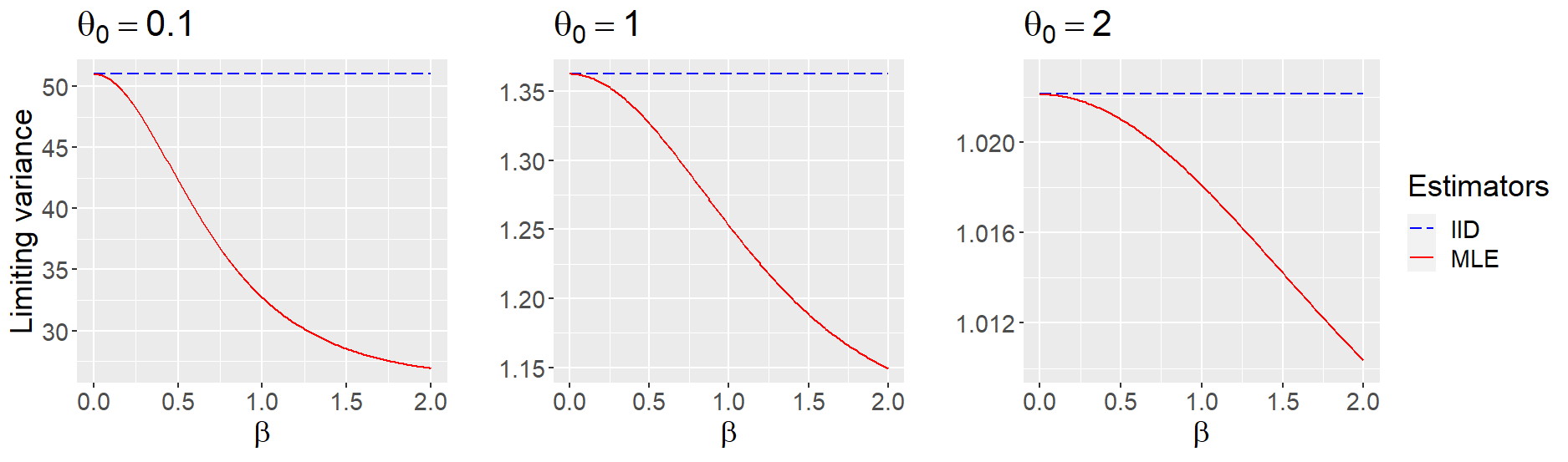}
    \caption{Scaled limiting variance of the estimators; ``IID'' denotes $\thetaiid$ and ``MLE'' denotes $\hat{\ttheta}_n^{\text{MLE}}$. For all $\beta>0$ and $\ttheta_0$, $\hat{\ttheta}_n^{\text{MLE}}$ is always more efficient compared to $\thetaiid$.}
    \label{fig:counterexample}
\end{figure}

\subsubsection{Low temperature regime: $\beta > 1$}\label{subsec:low temperature cw}
Now, we consider the low temperature regime $\beta > 1$, where the true labels are still generated from the Ising model $\QQ_{0, \beta, \A_n}$. The low-temperature regime is typically more challenging than the high-temperature case and many results depend on specific structures of the coupling matrix $\A_n$ \citep{friedli2017statistical, bhattacharya2021sharp, deb2023fluctuations}. In particular, the critical temperature (and consequently, the definition of the low temperature regime) depends on the sequence of graphs $\{\A_n\}_{n \ge 1}$ as we have seen in \cref{rem:counterex}. To make $\beta >1 $ be the bona fide low-temperature regime, we assume that $\A_n$ is an \emph{approximately regular} matrix and is \emph{well-connected} in addition to the mean-field condition \eqref{eq:mean field assumption}. These conditions are motivated by \cite{deb2023fluctuations}, which establish universal phase transitions at $\beta = 1$ for such $\A_n$.
One can immediately check that the Curie-Weiss model satisfies these conditions. Other possible choices of $\A_n$ include the Erdős–Rényi random graph and the balanced stochastic block model; see Section 1.3 in \cite{deb2023fluctuations} for additional examples.

\begin{assume}[approximately regular / well-connected]\label{defi:regularity}
    We say that a sequence of coupling matrices $\{\A_n\}_{n \ge 1}$ is \emph{approximate regular} when the row sums $R_i := \sumjn A_n(i,j)$ satisfy
    \begin{equation*}
        \sumin (R_i - 1)^2 = o(\sqrt{n}), ~\sumin (R_i - 1) = o(\sqrt{n}).
    \end{equation*}
    Also, we say that an approximate regular sequence $\{\A_n\}_{n \ge 1}$ is \emph{well-connected} when its two largest eigenvalue $\lambda_1(\A_n) \ge \lambda_2(\A_n)$ satisfies $\limsup_{n \to \infty} \frac{\lambda_2(\A_n)}{\lambda_1(\A_n)} < 1.$ Note that for approximately regular graphs that satisfy \eqref{eq:row sum one scaling}, we have $\lambda_1(\A_n) \to 1$.
\end{assume}

When $\beta >1$ and $\{\A_n\}_{n \ge 1}$ is approximately regular and well-connected, the estimator $\thetaiid$ turns out to be suboptimal. Hence, we need to find an alternative estimator with an optimal variance and also compute the LAN expansion of the likelihood. We divide this subsection into two parts, and consider the upper bound (constructing an estimator) and lower bound (LAN expansion) separately. The argument is more technical than the high/critical temperature regime due to the \emph{asymmetric proportion} of the labels, and we first introduce a common notation that will be used throughout \cref{subsec:low temperature cw}.
For the same technical reason, we present some results conditioned on the event $\bar{\X} \in \Theta_1$. 

\begin{defi}\label{defi:EE ttheta}
    Fix $\beta > 1$ and recall that $m = m(\beta)$ is the unique positive root of $m = \tanh(\beta m)$. For $\ttheta_0 \in \Theta_1$, let $\PP_{\ttheta_0}$ denote the weighted mixture of two symmetric Normals: 
    $$\PP_{\ttheta_0} := \frac{1+m}{2} N_d(\ttheta_0, \I) + \frac{1-m}{2} N_d(-\ttheta_0, \I).$$ Also, let $\EE_{\ttheta_0}$ be the expectation under the distribution $\PP_{\ttheta_0}$. 
\end{defi}

\vspace{1mm}
\paragraph{Upper bound}
We define the estimator $\thetamf$ by setting
\begin{equation}\label{eq:mean-field estimator definition}
    (\hat{U}_n, \thetamf) := \argmin_{(u, \ttheta) \in [-1,1] \times \Theta_1} M_n(u, \ttheta),
\end{equation}
where $M_n: [-1,1] \times \Theta \to \mathbb{R}$
is %
\begin{align*}%
    M_n(u, \ttheta) := \frac{\ttheta^{\top} \ttheta}{2} + \frac{\beta u^2}{2} - \frac{1}{n} \sumin \log \cosh(\beta u + \ttheta^\top \X_i).
\end{align*}
Here, $\hat{U}_n$ is a nuisance quantity that serves as a proxy for the posterior mean $\EE^{\QQCW} \bar{W}$.

\vspace{2mm}

\noindent \emph{Deriving the estimator $\thetamf$.} The function $M_n$ arises from the following \emph{mean-field approximation} of the log-likelihood. 
For simplicity, let us assume Curie-Weiss labels and recall that the true log-likelihood is proportional to 
\begin{align}\label{eq:loglik}
    l_n(\ttheta) = - \frac{\ttheta^{\top} \ttheta}{2}  + \frac{1}{n} \log Z_{n,\beta}^{\text{CW}}(\ttheta, \X^n).
\end{align}
The mean-field approximation for the log-partition function $\log{Z_{n,\beta}^{\text{CW}}(\ttheta, \X^n)}$ (see Example 5.2 in \cite{wainwright2008graphical} or eq. (2.4) in \cite{lee2024rfim}) can be written as 
\begin{equation}\label{eq:gibbs variational principle}
\begin{aligned}
    \frac{1}{n} \log{Z_{n,\beta}^{\text{CW}}(\ttheta, \X^n)} 
    &\approx \sup_{\bu \in [-1,1]^n} \left(\frac{\beta \bar{u}^2}{2} + \ttheta^\top \frac{1}{n} \sumin \X_i u_i- \frac{1}{n} \sumin H(u_i) \right),
\end{aligned}
\end{equation}
where the function $H:[-1,1]\to \mathbb{R}$ is the binary entropy, defined as
$$H(u) := \mathrm{KL}\left(\text{Rad}\big(\frac{1+u}{2}\big) \|  \text{Rad}\big(\frac{1}{2}\big)\right) = \frac{1+u}{2} \log \frac{1+u}{2} + \frac{1-u}{2} \log \frac{1-u}{2}.$$
By observing the first order conditions in \eqref{eq:gibbs variational principle}, the supremum is attained at $\hat{u}_i$s that satisfy the following fixed point equations:
\begin{align}\label{eq:fixed point}
    \hat{u}_i = \tanh(\beta \bar{\hat{u}} + \ttheta^\top \X_i), ~ \mbox{for all } i.
\end{align}
By plugging this expression of the optimizers $\hat{u}_i$ into \eqref{eq:gibbs variational principle} and \eqref{eq:loglik}, for each $\ttheta$, we have
\begin{align}\label{eq:loglik mf approx}
    l_n(\ttheta) \approx - \frac{\ttheta^\top \ttheta+ \beta \bar{\hat{u}}^2}{2} + \frac{1}{n} \sumin \log \cosh(\beta \bar{\hat{u}} + \ttheta^\top \X_i).
\end{align}
Note that the value of $\bar{\hat{u}}$ implicitly depends on the variable $\ttheta$ and it is still hard to directly maximize the RHS of \eqref{eq:loglik mf approx}. Hence, we instead view the RHS as a bivariate function of $\bar{\hat{u}}$ and $\ttheta$, which is exactly $-M_n(\bar{\hat{u}}, \ttheta)$, and maximize over both arguments. Now, the resulting M-estimator is $\thetamf$, defined in \eqref{eq:mean-field estimator definition}.

{The exact form of the optimizers $\bu$ in \eqref{eq:fixed point} requires the Curie-Weiss model. However, one can understand \eqref{eq:fixed point} as an \emph{amortization} that assumes a one-dimensional common structure for each variational parameter $u_i$. Using the language of variational inference, one can understand the RHS of \eqref{eq:gibbs variational principle} as the evidence lower bound (ELBO), and the RHS of \eqref{eq:loglik mf approx} as the amortized ELBO \citep{blei2017variational,ganguly2023amortized}. In the following paragraph, we show the robustness of amortization even when $\A_n$ deviates from the complete graph, as long as it is regular, well-connected, and mean-field.}

\vspace{2mm}
\noindent \emph{Limiting distribution of $\thetamf$.} Now, we claim that the estimator $\thetamf$ is asymptotically normal when $\A_n$ is approximately regular, well-connected, and mean-field. First, to show the consistency of $\thetamf$, we have to understand the limit of the function $M_n$. To this extent, for $|u| \le 1$ and $ \ttheta\in \Theta_1$, we define %
$$M_{\infty}(u,\ttheta) := \frac{\ttheta^{\top} \ttheta}{2} + \frac{\beta u^2}{2} - \EE_{\ttheta_0} \log \cosh (\beta u + \ttheta^{\top} \X).$$
The following Lemma shows that $M_n$ converges pointwise to $M_{\infty}$.
Recall that $\cdot:(\bar{\X} \in \Theta_1)$ denotes conditioning on the event $\bar{\X} \in \Theta_1$. Also, note that both functions $M_n$ and $M_{\infty}$ depend on the known parameter $\beta$, which we omit for notational convenience.

\begin{lemma}\label{lem:conditional limit }
Suppose $\beta > 1$, $\A_n$ satisfy Assumptions \ref{defi:mean-field}, \ref{defi:regularity}, and let $\X^n \sim P_{\ttheta_0, \QQ_{0}}$. Then, for any $|u| \le 1$ and $\ttheta\in \Theta$, $M_n(u, \ttheta) : (\bar{\X} \in \Theta_1) \xp M_\infty(u,\ttheta).$
\end{lemma}

In the next Lemma, we show that $M_{\infty}$ is minimized at $(u, \ttheta) = (m, \ttheta_0)$. This result justifies the consistency of $\thetamf$, and provides insights for computation. Due to limited of space, we postpone all low temperature regime proofs to the Supplementary Material.

\begin{lemma}\label{lem:uniqueness of minimizer} 
For any $\beta > 1$, $M_{\infty}:[-1,1]\times \Theta_1 \to \mathbb{R}$ is uniquely minimized at $(u, \ttheta) = (m, \ttheta_0)$.
\end{lemma}

Now, we derive the limiting distribution of $\thetamf$ in Theorem \ref{thm:upper bound}.
To state its variance, we need the following definitions.

\begin{defi}\label{defi:info}
    Define a $(d+1) \times (d+1)$ matrix $\Gamma = \begin{pmatrix}
    \gamma_{1,1} & \ggamma_{1,2}^{\top} \\
    \ggamma_{1,2} & \ggamma_{2,2}
\end{pmatrix}$ as the Hessian of $M_{\infty}$ at $(m, \ttheta_0)$, i.e.
\begin{align*}
\gamma_{1,1} &:= \frac{\partial^2 M_{\infty} (u, \ttheta)}{\partial u^2} \mid_{(u, \ttheta) = (m, \ttheta_0)} =  \beta - \beta^2 \EE_{\ttheta_0} \sech^2(\beta m + \ttheta_0^\top \X) \in \mathbb{R}, \\
\ggamma_{1,2} &:= \frac{\partial^2 M_{\infty} (u, \ttheta)}{\partial u \partial \ttheta} \mid_{(u, \ttheta) = (m, \ttheta_0)} = - \beta \EE_{\ttheta_0} \X \sech^2(\beta m + \ttheta_0^\top \X) \in \mathbb{R}^d, \\
\ggamma_{2,2} &:= \frac{\partial^2 M_{\infty} (u, \ttheta)}{\partial \ttheta^2} \mid_{(u, \ttheta) = (m, \ttheta_0)} = \I - \EE_{\ttheta_0} \X \X^{\top} \sech^2(\beta m + \ttheta_0^\top \X) \in \mathbb{R}^{d \times d}.
\end{align*}
For $\beta > 1$, we define a $d \times d$ matrix $I_{\beta}(\ttheta_0)$ as the Schur complement of $\gamma_{1,1}$ in $\Gamma$, i.e. $I_{\beta}(\ttheta_0) := \ggamma_{2,2} - \frac{\ggamma_{1,2} \ggamma_{1,2}^\top}{\gamma_{1,1}}$. 
\end{defi}

\begin{thm}\label{thm:upper bound}
Suppose $\beta > 1$, and that $\A_n$ satisfy Assumptions \ref{defi:mean-field} and \ref{defi:regularity}. Let $\X^n \sim P_{\ttheta_0, \QQ_{0}}$ and $\thetamf$ be the estimator defined in \eqref{eq:mean-field estimator definition}. Then,
$I_{\beta}(\ttheta_0)$ is invertible and $\thetamf$ satisfies
$$\sqrt{n}(\thetamf - \ttheta_0)  \xd N_d \left(0, {I_{\beta}(\ttheta_0)^{-1}} \right).$$
\end{thm}

The mean-field estimator requires computing the nuisance quantity $\hat{U}_n$, and it is natural to question whether there are simpler estimators with the same or better asymptotic variance. We address this in the following remark and show that a natural alternative estimator (denoted as $\thetaamle$) has a larger variance. In \cref{fig:variance}, we display the limiting variance (where $\A_n$ is mean-field, approximately regular, and well-connected) of the three estimators we consider in this paper. The figure verifies that $\thetaiid$ and $\thetaamle$ are sub-optimal compared to $\thetamf$.

\begin{remark}\label{rmk:alternative estimators low tmp}
    An alternative estimation strategy arises from approximating the \emph{true label distribution} $\QQ_{0,\beta,\A_n}$ with a product distribution. Instead of blindly assuming equally likely labels as in the construction of $\thetaiid$, we use the product distribution that is closest to $\QQ_{0,\beta,\A_n}$ in terms of the KL divergence.
    This motivates us to approximate $\QQ_{0,\beta,\A_n}$ as the $n$-fold product of $\text{Rad}(\frac{1+\tilde{m}}{2})$, where $\tilde{m} = \tilde{m}(\X^n):= \begin{cases}
        m  ~&\text{ if }~ \bar{\X} \in \Theta_1\\
        -m ~&\text{ if }~ \bar{\X} \in \Theta_2
    \end{cases}.$ We define $\hat{\ttheta}_n^{\text{aMLE}}$ as the approximate MLE computed under this approximation:
    $$\hat{\ttheta}_n^{\text{aMLE}} := \argmin_{\ttheta \in \Theta_1} \left[ \frac{\ttheta^\top\ttheta}{2} - \frac{1}{n} \sumin \log \cosh(\beta \tilde{m} + \ttheta^\top \X_i) \right].$$
    By a similar argument as in Theorem \ref{thm:upper bound iid}, we can derive the limiting distribution
    \begin{align*}
        \sqrt{n} (\hat{\ttheta}_n^{\text{aMLE}} - \ttheta_0) &\xrightarrow{d} N_d \left(0,\ggamma_{2,2}^{-1} \ssigma_{2,2}\ggamma_{2,2}^{-1} \right).
    \end{align*}
    When $\beta > 1$, this is strictly larger than $I_\beta(\ttheta_0)$ since $\ggamma_{1,2} \neq \mathbf{0}$.
\end{remark}

\begin{figure}[ht!]
    \centering
    \includegraphics[width = 0.8\textwidth]{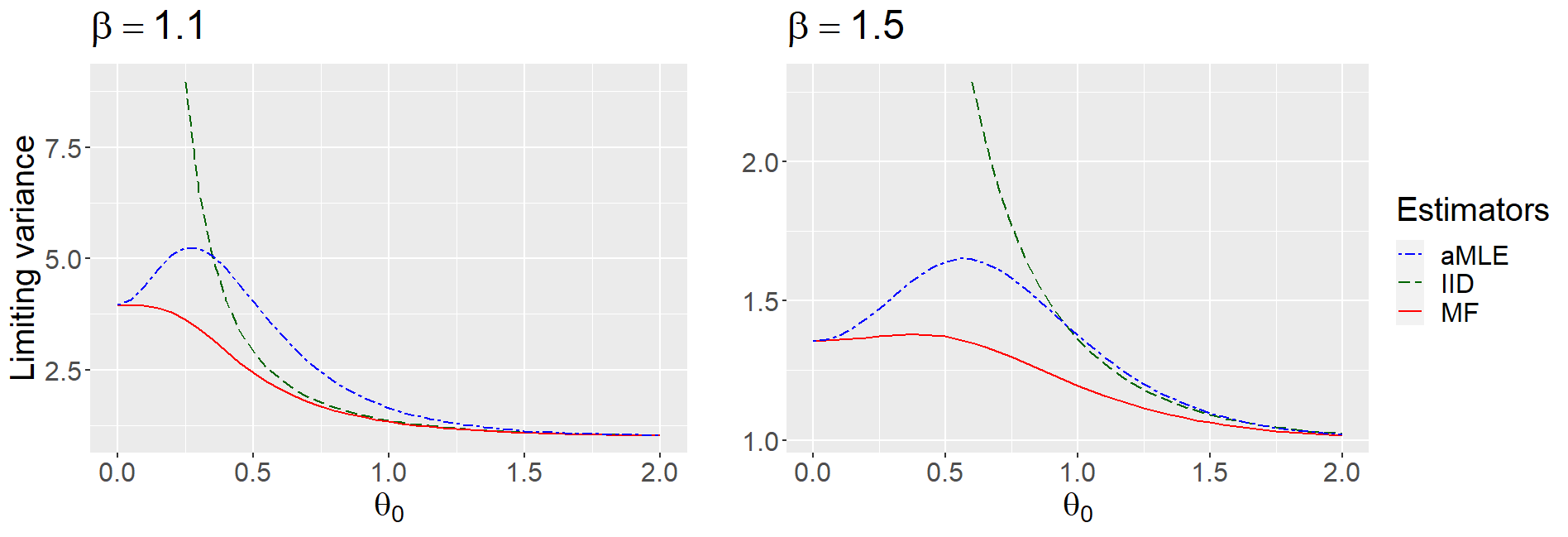}
    \caption{Scaled limiting variance of the estimators considered in this paper; ``IID'' denotes $\thetaiid$, ``MF'' denotes $\thetamf$, and ``aMLE'' denotes $\hat{\ttheta}_n^{\text{aMLE}}$. For both $\beta = 1.1$ and $1.5$, we see that $\thetamf$ has the smallest variance.}
    \label{fig:variance}
\end{figure}

Before moving on to deriving the LAN expansion with a matching precision matrix, we illustrate that $\thetamf$ can be computed by an EM-type iterative algorithm.

\begin{remark}[Computing the mean-field estimator]
    Recall from \cref{lem:uniqueness of minimizer} that the function $M_{\infty}$ is uniquely minimized at $(m, \ttheta_0)$. When $\|\ttheta_0\|$ is large enough, %
    $M_{\infty}$ turns out to be convex. This justifies using the following variational EM algorithm with a random initialization $(\hat{U}_n^{(0)}, \hat{\ttheta}^{(0)})$ to compute $\thetamf$, which iteratively computes
    \begin{align*}
        \begin{pmatrix}
            \hat{U}_n^{(t+1)} \\
            \hat{\ttheta}^{(t+1)}
        \end{pmatrix} = \frac{1}{n} \sumin \begin{pmatrix}
            \beta \\
            \X_i
        \end{pmatrix} \tanh(\beta \hat{U}_n^{(t)} + \hat{\ttheta}^{(t) \top} \X_i), ~ \mbox{for all } t \ge 0.
    \end{align*}
    For a general $\ttheta_0$ and $\beta > 1$, $M_{\infty}$ may have multiple local minimizers \citep[e.g. see Theorem 1 in][]{xu2018benefits}, and the convergence of the above algorithm would depend on the initialization. Hence, for practical purposes, we suggest using the rate-optimal initialization $(\hat{U}_n^{(0)}, \hat{\ttheta}^{(0)}) = (\tilde{m}, \thetaiid)$, which will be already close to $(\hat{U}_n, \thetamf)$. Recall the definition of $\tilde{m}$ from \cref{rmk:alternative estimators low tmp}.
\end{remark}

\paragraph{Lower bound}
Now, we compute the LAN expansion, which will give us the information theoretic lower bound for estimation. We present the LAN expansion under \emph{Curie-Weiss} labels, as we were unable to compute the LAN expansions for other Ising models with a general coupling matrix $\A_n$. Recall from \eqref{eq:CW pmf} that we write the the Curie-Weiss label distribution as $\QQ_{0,\beta}^{\text{CW}}$ and the resulting distribution of $\X^n$ as $P_{\ttheta_0,\beta}^{\text{CW}}$.

Our main ingredient for deriving the matching lower bound is the following expansion: 
\begin{equation}\label{eq:sum x_i w_i low temperature}
    \sumin \X_i \EE^{\QQ_{0,\beta}^\text{CW}} W_i = \sumin \X_i \tanh(\beta U_n + \ttheta_0^\top \X_i) + O_p(1).
\end{equation}
This is a version of \eqref{eq:sum xi wi expansion intro}, where we take the centering $u_n(\beta, \ttheta_0, \X^n) := U_n$.
Here, $U_n$ is defined as the minimizer of $M_{n}(u, \ttheta_0)$ with respect to $u$: 
\begin{equation}\label{eq:u_n defi}
    U_n := \argmin_{|u| \le 1} M_n(u, \ttheta_0) = \argmin_{|u| \le 1} \left[ \frac{\beta u^2}{2} -\frac{1}{n} \sumin \log \cosh(\beta u + \ttheta_0^{\top} \X_i) \right].
\end{equation}
We can interpret $U_n$ as an oracle quantity of $\hat{U}_n$ (defined in \eqref{eq:mean-field estimator definition}), in the sense that we are using the \emph{true value} $\ttheta_0$.
Using these notations, we state the LAN expansion below.

\begin{thm}\label{thm:lower bound}
    Suppose $\beta > 1$ and $\X^n \sim P_{\ttheta_0, \beta}^\text{CW}$. For $\bh \in \mathbb{R}^d$, let $\ttheta_n := \ttheta_0+\frac{\bh}{\sqrt{n}}$. Then, \eqref{eq:sum x_i w_i low temperature} holds, and 
    $$\log \frac{dP_{\ttheta_n, \beta}^\text{CW}}{dP_{\ttheta_0, \beta}^\text{CW}} (\X^n) = \bh^{\top} \tilde{\Delta}_{n, \ttheta_0, \beta}- \frac{1}{2} \bh^\top I_\beta(\ttheta_0) \bh + o_p(1),$$
    where $$\tilde{\Delta}_{n, \ttheta_0, \beta} := \sqrt{n} \left(\frac{1}{n} \sumin \X_i \tanh(\beta U_n + \ttheta_0^\top \X_i) - \ttheta_0 \right) \xrightarrow[P_{\ttheta_0, \beta}^\text{CW}]{d} N_d (0, I_\beta(\ttheta_0)).$$
    Hence, the family $\{P_{\ttheta,\beta}^{\text{CW}} \}_{\ttheta\in \Theta_1}$ is LAN with a precision matrix $I_\beta(\ttheta_0)$ at any $\ttheta_0 \in \Theta_1$.
\end{thm}

Now, in the next Corollary, we combine the upper bound and lower bound, and conclude that $\thetamf$ is indeed optimal. 
\begin{cor}\label{cor:low tmp}
    Suppose $\beta > 1$ and $\X^n \sim P_{\ttheta_0, \beta}^{\text{CW}}$.
    Then, $\thetamf$ is regular, i.e. for $\ttheta_n = \ttheta_0 + \frac{\bh}{\sqrt{n}}$, $$\sqrt{n}(\thetamf - \ttheta_n) \xrightarrow[P_{\ttheta_n,\beta}^\text{CW}]{d} N_d(0, I_\beta(\ttheta_0)^{-1}).$$
\end{cor}

One immediate question is whether one can generalize \cref{thm:lower bound} to Ising models beyond the Curie-Weiss model, possibly to the full extent of coupling matrices $\A_n$s that satisfy the conditions in the upper bound (see \cref{thm:upper bound}).
    The main bottleneck in terms of deriving such a lower bound is the lack of tight concentration results for RFIMs in the low temperature regime. In the high temperature lower bound, we have crucially utilized the moment bounds for RFIMs that were developed in the recent work \cite{lee2024rfim}. However, the results in \cite{lee2024rfim} do not apply to low temperatures, and we are not certain whether this is generally true. Our current proof for \cref{thm:lower bound} computes the RFIM moments by exploiting the low-rank nature of the Curie-Weiss coupling matrix, and cannot be generalized for general mean-field and approximately regular matrices $\A_n$.

We end the section with additional remarks regarding analyzing $\thetamf$ in the high/critical temperature regime, and implications of \cref{thm:lower bound} for testing.

\begin{remark}[Comparison with the high/critical-temperature regime]
    While Theorem \ref{thm:upper bound} analyzed $\thetamf$ only under $\beta > 1$, we can show that its limiting distribution under $\beta \le 1$ is the same as \cref{thm:upper bound iid}. Indeed, for $\beta \le 1$, the definition of $I_{\beta}(\ttheta_0)$ in Definition \ref{defi:info} is consistent with the definition of $I_0(\ttheta_0)$ from the previous section. This follows because $m = 0$ and $\ggamma_{1,2} = \mathbf{0}$, which allows us to simplify $\ggamma_{2,2} - \frac{\ggamma_{1,2} \ggamma_{1,2}^\top}{\gamma_{1,1}} = \ggamma_{2,2} = I_0(\ttheta_0)$. Thus, $\thetamf$ is optimal under Curie-Weiss labels for \emph{all} $\beta \ge 0$.
\end{remark}

\begin{remark}[Testing is easier than estimation in low temperatures]
    Consider testing the hypothesis in Remark \ref{rmk:testing high temperature}. While the LAN expansion in Theorem \ref{thm:lower bound} depends on $\ttheta_0$ and does not define an estimator, this can be directly applied for testing. Indeed, one can simply construct an asymptotically optimal test based on $\tilde{\Delta}_{n, \ttheta_0, \beta}$. Of course, one may also construct an optimal test using the more complicated estimator $\hat{\ttheta}_n^{\text{MF}}$.
\end{remark}

\subsection{Unknown strength of dependence}
In this paper so far, we have established the optimality of estimating the mean parameter $\ttheta$ under the assumption that the Ising model $\QQ_{0, \beta, \A_n}$ is given. In particular, we have assumed the knowledge of the inverse temperature parameter $\beta$.
One immediate question is to understand how the estimation changes when $\beta$ is unknown. Compared to the GMM with iid labels, this roughly corresponds to the setting where the label proportions are unknown.

Here, we provide a partial answer under the Curie-Weiss model with unknown $\beta$. Let $\beta_0$ be the true inverse temperature parameter. First, we test $H_0: \beta_0 \le 1$ v.s. $H_1: \beta_0 > 1$ by rejecting the null when $\lVert\bar{\X}\rVert$ is large enough\footnote{Any threshold $\tau_n$ that satisfies $n^{-1/4} \ll \tau_n \ll 1$ leads to a consistent test}. If $\beta_0 \le 1$, since the assumption that $\beta$ is unknown does not improve the information lower bound $I_0(\ttheta_0)$ \citep[e.g. pg 128 in][]{lehmann2006theory}, the universal estimator $\thetaiid$ continues being optimal. Also, noting that $\beta_0 < 1$ cannot be consistently estimated even when the labels $\Z^n$ are observed \citep{bhattacharya2018inference}, consistent estimation of $\beta$ is impossible. 

When $\beta_0 > 1$, the estimator $\thetamf$ cannot be used since it requires the knowledge of $\beta$. 
{
Indeed, we expect that the information lower bound would change under an unknown $\beta$. To understand this, one may consider the extreme case with $\beta = \infty$, which corresponds to all labels being identical. For $d=1$ dimensions, while one can attain the lower bound of $I_{\infty}(\ttheta) = 1$ given the knowledge of $\beta$ (and consequently identical labels), it is not straightforward otherwise. To rigorously understand optimality, a joint LAN expansion for $(\beta,\ttheta)$ would be required, and we leave this problem for future research.
}
However, $\sqrt{n}$-consistent estimation of $(\beta_0, \ttheta_0)$ is possible; one can still use $\thetaiid$ to estimate $\ttheta_0$, and use $\hat{\beta} := \frac{\tanh^{-1}(\hat{m})}{\hat{m}}$ to estimate $\beta_0$. Here, $\hat{m} := {\lVert \bar{\X} \rVert}/{\lVert \thetaiid \rVert}$ is a method of moment estimator for the label mean $m = m(\beta)$, which arises from noting that $\EE\lVert \bar{\X}\rVert = \lVert \ttheta_0 \rVert m + O\left( \frac{1}{\sqrt{n}}\right)$.

\section{Discussion}
\subsection{Connections with literature}
\noindent \emph{Hidden Markov Random fields.} 
As pointed out in the introduction, mixture models with dependent labels have long been studied in the context of HMRFs, but there is little work regarding inference guarantees.
HMRFs are a popular framework in spatial statistics, genetics, and image segmentation/restoration \citep{besag1986statistical, franccois2006bayesian, pyun2002robust, chatzis2010infinite} to model network dependence among latent variables. This is a generalized notion of hidden Markov models (HMM), which are a special case of HMRFs under a Markov chain dependence. For HMMs, efficiency theory has been previously established using ergodic theory \citep{bickel1996inference, bickel1998asymptotic}. However, their proof techniques are restricted to time series dependence and do not generalize to more dense network dependence that we consider. 

Recently, the model \eqref{eq:model} with HMM labels have been analyzed in the high-dimensional setting \citep{zhang2022mean, karagulyan2024adaptive}, where the authors propose rate-optimal spectral estimators based on a temporal partition of the data. However, this line of research focus on rate-optimal minimaxity, which is different from the asymptotic efficiency with sharp constants. Indeed, we do not expect such moment-based estimators to attain the information-theoretic lower bound.

In terms of HMRFs, one related theoretical work is \cite{lai2015asymptotically}, which considers time-dependent observations from spatial HMRFs and shows asymptotic efficiency of a block-likelihood-based MLE. It is also worth mentioning that after ignoring the temporal effect, the motivating example in \cite{lai2015asymptotically} is also similar to model \eqref{eq:model}. However, \cite{lai2015asymptotically} requires many implicit correlation-decay and mixing conditions regarding the latent dependence, which are extremely challenging to check for individual examples. Furthermore, the block-likelihood still suffers from the intractable normalizing constant within each block. In contrast, our work does not require any such assumptions, and we propose optimal estimators that entirely avoid computing the normalizing constant.

\vspace{2mm}

\noindent \emph{Comparison with inference on Ising models.} 
One popular research question in statistical inference on MRFs is to estimate 
the dependence/inverse temperature parameter $\beta$ \citep{comets1991asymptotics, chatterjee2007estimation, bhattacharya2018inference, ghosal2020joint, xu2023inference, mukherjee2021efficient}. The setting is that one observes the exact labels $\Z^n$ generated from $\QQ_{0, \beta, \A_n}$ with a known graph $ \A_n$, with the goal of estimating the unknown parameter $\beta$. 
Similar to dependent GMMs, the MLE is intractable due to the implicit normalizing constant. In particular, the recent paper \cite{xu2023inference} assumes that  $ \A_n$ is the scaled adjacency matrix of a dense regular graph and provides a complete picture for estimation. They show that consistent estimation of $\beta$ is impossible when $\beta < 1$, whereas the MLE and maximum pseudo-likelihood estimator (MPLE)
are $\sqrt{n}$-consistent when $\beta \ge 1$. While both estimators are optimal when $\beta > 1$, the MPLE is only rate-optimal when $\beta = 1$ and a tractable alternative to the MLE is unknown. 

Compared to this result, for our problem of estimating $\ttheta$ in GMMs, we have already proved in \cref{thm:upper bound iid} that $\sqrt{n}$-consistent estimation is possible for any distribution $\QQ_0$. %
Another comparison is at the critical temperature $\beta = 1$, at which the estimation of $\beta$ exhibits a non-Normal limiting experiment, whereas our estimation of $\ttheta$ still has a Normal limiting experiment. A final remark is that the MPLE, a popular tractable estimator in Ising models and its variants \citep{chatterjee2007estimation, daskalakis2019regression, mukherjee2022estimation, dagan2021learning}, is not applicable to us since the log-normalizing constant in \eqref{eq:MLE definition} depends on $\X_i$ and makes the psuedo-likelihood $\prod_{i=1}^n \PP(\X_i \mid \{\X_j: j\neq i\})$ intractable.

\subsection{Future research directions}
\noindent \emph{General mixture models with dependence.} 
Currently, for simplicity, we have considered the most basic GMM with two symmetric components. It would be interesting to consider GMMs with more components that may not necessarily be symmetric, by modeling the label distribution as a Potts model. %
Alternatively, one could consider other mixture models where the conditional distribution of the observed responses follows an exponential family distribution. We carefully conjecture that similar results, such as a universal $\sqrt{n}$-consistent estimator, can be derived as long as the exponential family distribution exhibits a nontrivial even partition function. 
\vspace{2mm}
\noindent \emph{High-dimensional asymptotics.}
Another interesting direction would to be explore inference guarantees under a high-dimensional setting where $d$, the dimension of the responses, grows with $n$. There has been a recent interest for understanding the estimation of $\ttheta$ under high-dimensional symmetric Gaussian mixture models \citep[][and the references therein]{wu2021randomly, dwivedi2020singularity, karagulyan2024adaptive}, but their focus has been on how the minimax rate changes with respect to the signal strength $\lVert\ttheta\rVert$. Up to our knowledge, the sharp constants for estimation as well as inference guarantees have not been established, even under the iid label setting. To this extent, it would be important to explore the limiting behavior of $\thetaiid$ in high-dimensions and understand whether our universality result (\cref{thm:upper bound iid}) can be generalized. A more challenging question would be to also extend our optimality results to high dimensions.

\vspace{2mm}
\noindent \emph{Labels with non-mean-field dependence.} Finally, an important open question is to understand optimal estimation under label distributions $\QQ_{0,\beta,\A_n}$ generated by non-mean-field matrices $\A_n$. In particular, many practical applications for HMRFs in spatial statistics and image analysis consider a lattice type of dependence, where $ \A_n$ is the adjacency matrix of $\mathbb{Z}^D$ for an integer $D \ge 2$. This choice of $ \A_n$ does not satisfy the mean-field condition \eqref{eq:mean field assumption}, and our optimality results cannot be applied. Based on preliminary simulations, we believe that the universal optimality of $\thetaiid$ in the high temperature regime would no longer hold. Thus, we require different approaches, such as using the recent developments on correlation decay \citep{mukherjee2022testing, ding2023new}. %
We plan to consider the efficiency theory under such sparse graphs in the future.

\subsection{Proof organization}
The remainder of this paper is organized as follows. {In \cref{sec:proof main paper}, we prove the theoretical claims made in Sections \ref{sec:iid estimator} and \ref{subsec:high temperature cw}. First, in \cref{sec:proof iid}, we prove \cref{lem:uniqueness iid} and \cref{thm:upper bound iid}. In \cref{proof:main lemma high tmp}, we prove \cref{lem:sum x_i z_i} by utilizing moment bounds for the RFIM. In \cref{sec:proof high tmp}, we prove \cref{thm:lower bound high temperature} and \cref{cor:high tmp} by combining the \cref{lem:sum x_i z_i} with Le Cam theory.}
We prove all low-temperature results from \cref{subsec:low temperature cw} as well as remaining Lemmas in the Supplementary Material. Hidden constants (in $\lesssim$ or $O(\cdot)$ notations) will be specified in each segment of the proof.

\section{Proof of results in Sections \ref{sec:iid estimator} and \ref{subsec:high temperature cw}}\label{sec:proof main paper}

As we work with dependent responses $\X^n$, we cannot apply the well-known limit theorems for independent random variables.
We first state a dependent variant of the uniform LLN (ULLN) and central limit theorem under model \eqref{eq:model}, which will be used multiple times throughout this section. The proofs of these Lemmas are deferred to \cref{sec:proof of lemmas}. 

Our first lemma is the following ULLN. Note that this automatically implies a non-uniform law of large number as well.

\begin{lemma}[ULLN]\label{lem:ULLN high}
For an arbitrary distribution $\QQ_0$, let $\X^n \sim P_{\ttheta_0, \QQ_0}$, and let $\Psi \subset \mathbb{R}^k$ be a compact set.
For $\bx \in \mathbb{R}^d, \psi \in \Psi$, let $f(\bx, \psi)$ be a bivariate function that is an even in $\bx$ (i.e. $f(\bx, \psi) = f(-\bx, \psi)$) and satisfies the following conditions for finite constants $C_1(\ttheta_0), C_2(\ttheta_0), C_3(\ttheta_0) < \infty$:
\begin{itemize}
    \item $\sup_{\psi \in \Psi} \Var [f(\X, \psi) \mid Z = z] \le C_1(\ttheta_0)$ for $z = \pm 1$. 
    \item $\sup_{\psi \in \Psi} \EE [ | f(\X, \psi)| \mid Z = z] \le C_2(\ttheta_0)$ for $z = \pm 1$.
    \item $\sup_{\psi \in \Psi} \lVert \frac{\partial f}{\partial \psi} (\X, \psi) \rVert \le h(\X)$, where $h$ satisfies $\EE [h(\X) \mid Z = z] \le C_3(\ttheta_0)$ for $z = \pm 1$.
\end{itemize}
Then, 
    \begin{equation*}
        \sup_{\psi \in \Psi} \left| \frac{1}{n} \sumin f(\X_i, \psi) - \EE_{\X \sim N_d(\ttheta_0, \I)} f(\X, \psi) \right| \xp 0.
    \end{equation*}
    The same conclusion holds when $f$ is vector-valued (say, $k'$-dimensional for some finite $k'$) and the absolute value is replaced by any vector norm.
\end{lemma}

Our second Lemma computes the limiting distribution of the statistic $\sumin \X_i \tanh(\ttheta_0^\top \X_i)$, and will be used in both the lower and upper bound. Note that this statistic is the gradient of $N_n$ (see \cref{sec:iid estimator}), and also appears as $\Delta_{n,\ttheta_0}(\X^n)$ in the LAN expansion (see \cref{thm:lower bound high temperature}).

\begin{lemma}[Limiting distribution of $\Delta_{n,\ttheta_0}$]\label{lem: CLT high}
Let $\QQ_0$ be an arbitrary measure on $\{-1, 1\}^n$ and let $\X^n \sim P_{\ttheta_0,\QQ_{0}}$.
Then,
\begin{align}\label{eq:N_n expansion}
    -\sqrt{n}(\nabla N_n) (\ttheta_0) = \Delta_{n,\ttheta_0}(\X^n) = \sqrt{n} \left(\frac{1}{n} \sumin \X_i \tanh(\ttheta_0^\top \X_i) - \ttheta_0 \right) \xd N_d(0, I_0(\ttheta_0)).
\end{align}
\end{lemma}

\subsection{Proof of \cref{lem:uniqueness iid} and \cref{thm:upper bound iid}}\label{sec:proof iid}
\cref{lem:uniqueness iid} follows from using the KL divergence to show the uniqueness of the minimization problem, and applying existing analysis of the first order conditions to argue convexity.

\begin{proof}[Proof of \cref{lem:uniqueness iid}]
    The differentiability is immediate. We first show that for any $\ttheta\neq \ttheta_0$ in $\Theta_1$, $N_{\infty}(\ttheta) > N_{\infty} (\ttheta_0)$.
    For any $\ttheta \in \Theta_1$, define a distribution $\bar{\PP}_{\ttheta} \equiv \frac{1}{2} N_d(\ttheta, \I) + \frac{1}{2} N_d(-\ttheta, \I)$, which has the following density function: $$\bar{p}_{\ttheta}(\x) = \frac{\exp \left[ - \frac{\x^{\top} \x}{2} - \frac{\ttheta^\top \ttheta}{2} + \log \cosh(\ttheta^{\top} \x) \right]}{\sqrt{2\pi}^d}.$$
    Then, the definition of $\Theta_1$ as the half-space makes $\{\PPB_{\ttheta}: \ttheta\in \Theta_1 \}$ an identifiable family. Since $\ttheta\neq \ttheta_0$,
    \begin{align*}
        \text{KL}(\PPB_{\ttheta_0} \| \PPB_{\ttheta}) = \EE^{\bar{\PP}_{\ttheta_0}} &\Big[ - \frac{\ttheta_0^\top \ttheta_0}{2} + \log \cosh(\ttheta_0^{\top} \X) +  \frac{\ttheta^\top \ttheta}{2} - \log \cosh(\ttheta^{\top} \X)\Big] > 0,
    \end{align*}
    and we have
    $N_{\infty}(\ttheta) > N_{\infty}(\ttheta_0)$ by rewriting the terms.

    Since $\ttheta_0$ minimizes the differentiable function $N_{\infty}$, we have $(\nabla N_{\infty})(\ttheta_0) = 0$. {This can also be shown directly by using the symmetry of $\log \cosh$ to rewrite $(\nabla N_{\infty}) (\ttheta_0)$ as
    \begin{align*}
        (\nabla N_{\infty})(\ttheta_0) &= \ttheta_0 - \EE_{\X \sim N_d(\ttheta_0, \I)} \X \tanh (\ttheta_0^\top \X) \\
        &= \ttheta_0 - \EE_{\X \sim \frac{1}{2} N_d(\ttheta_0, \I) + \frac{1}{2} N_d(-\ttheta_0, \I)} \X \tanh(\ttheta_0^\top \X)\\
        &= \ttheta_0 - \frac{1}{2 \sqrt{2 \pi}} \int_{\mathbb{R}^d} \x \tanh(\ttheta_0^\top \x)  e^{- \frac{\x^\top \x + \ttheta_0^\top \ttheta_0}{2}} (e^{\ttheta_0^\top \x} + e^{-\ttheta_0^\top \x}) d \x \\
        &= \ttheta_0 - \frac{1}{2 \sqrt{2 \pi}} \int_{\mathbb{R}^d} \x e^{- \frac{\x^\top \x + \ttheta_0^\top \ttheta_0}{2}} (e^{\ttheta_0^\top \x} - e^{-\ttheta_0^\top \x}) d\x \\
        &= \ttheta_0 - \frac{1}{2 \sqrt{2 \pi}} \int_{\mathbb{R}^d} \x e^{- \frac{(\x - \ttheta_0)^\top (\x - \ttheta_0)}{2}} d\x + \frac{1}{2 \sqrt{2 \pi}} \int_{\mathbb{R}^d} \x e^{-\frac{(\x+\ttheta_0)^\top (\x+\ttheta_0)}{2}} d\x = 0.
    \end{align*}
    }
    
    To show the uniqueness of the solution of $\nabla N_{\infty} = 0$ in $\text{int}(\Theta_1)$, we use Theorem 2 in \cite{daskalakis2017ten}. This result states that for a mapping $T: \text{int}(\Theta_1) \to \mathbb{R}^d$ defined as $T(\ttheta) := \EE_{\X \sim N_d(\ttheta_0, \I)} \X \tanh(\ttheta^\top \X)$, we have
    $$\|T(\ttheta) - T(\ttheta_0)\| \le \kappa(\ttheta) \|\ttheta- \ttheta_0\|.$$
    Here, $\kappa(\ttheta) := \exp\left[ - \frac{\min(\ttheta^\top \ttheta, \ttheta_0^\top \ttheta)^2}{2 \ttheta^\top \ttheta} \right] \le 1$ and note that $(\nabla N_{\infty})(\ttheta_0) = 0$ implies $T(\ttheta_0) = \ttheta_0$.
    Suppose that there exists $\ttheta\in \text{int}(\Theta_1) - \{\ttheta_0\}$ such that $$(\nabla N_\infty)(\ttheta) = \ttheta- T(\ttheta) =\mathbf{0}_d.$$ If $\ttheta^\top \ttheta_0 \neq 0$, $\kappa(\ttheta)$ is strictly less than 1, and we have a contradiction. When $\ttheta^\top \ttheta_0 = 0$, Theorem 2 in \cite{daskalakis2017ten} also shows that $T(\ttheta) = 0$ and we have $(\nabla N_\infty)(\ttheta) = \ttheta \neq \mathbf{0}_d$. Consequently, $(\nabla N_{\infty})(\ttheta) = 0$ has an unique root $\ttheta= \ttheta_0$ in $\text{int}(\Theta_1)$. 
\end{proof}
\begin{remark}
    The restriction to the interior is imposed so that the gradient $\nabla N_{\infty}$ is well-defined. By considering the (nonidentifiable) entire domain $\Theta = \mathbb{R}^d-\{\mathbf{0}_d\}$ of $N_{\infty}$, one can remove this restriction and show that $\nabla N_{\infty}(\ttheta) = 0$ if and only if $\ttheta = \pm \ttheta_0$.
\end{remark}

We prove Theorem \ref{thm:upper bound iid} by modifying the classical argument for the asymptotic normality of M-estimators to our dependent setting, with the help of Lemmas \ref{lem:uniqueness iid}, \ref{lem:ULLN high}, and \ref{lem: CLT high}. Along with these, our main ingredient is the conditional independence of $\X^n \mid \Z^n$ and the symmetry of $\X_1 \mid Z_1$.

\begin{proof}[Proof of Theorem \ref{thm:upper bound iid}]
We divide the proof into two steps.

\textbf{Step 1: Consistency.}
We first claim that $\thetaiid$ is consistent. Define $B_{\ttheta_0} := \{\ttheta: \|\ttheta\| \le \|\ttheta_0\| + 2\sqrt{d} \}$.
Applying Lemma \ref{lem:ULLN high} with $f(\x, \ttheta) = \log \cosh (\ttheta^\top \x)$ gives
\begin{equation}\label{eq:ulln}
    \sup_{\ttheta\in \Theta_1 \cap B_{\ttheta_0}} |N_n(\ttheta) - N_{\infty}(\ttheta)| \xp 0.
\end{equation}
Recalling that $\thetaiid = \argmin_{\ttheta\in \Theta_1} N_n(\ttheta),$ we have $N_n(\thetaiid) \le N_n(\ttheta_0) = N_{\infty}(\ttheta_0) + o_p(1).$ Also, because the first order conditions of \eqref{eq:iid estimator definition} give $\thetaiid = \frac{1}{n} \sumin \X_i \tanh(\X_i^\top \thetaiid)$, a naive bound using $\|\X_i - \ttheta_0\| \equiv \sqrt{\chi^2_d}$ implies
\begin{align}\label{eq:estimator norm bound}
    \|\thetaiid\| \le \frac{1}{n}\sumin \|\X_i\| \le \|\ttheta_0\| + \frac{1}{n}\sumin \|\X_i - \ttheta_0\| \le \|\ttheta_0\| + 2\sqrt{d}
\end{align} with high probability. Thus, $\thetaiid \in B_{\ttheta_0}$ with high probability and \eqref{eq:ulln} gives
$N_n(\thetaiid) - N_{\infty}(\thetaiid) = o_p(1).$
Combining this, we have 
$$N_{\infty}(\thetaiid) \le N_{\infty}(\ttheta_0) + o_p(1).$$
By Lemma \ref{lem:uniqueness iid}, $N_{\infty}(\ttheta)$ is a continuous function that is uniquely minimized at $\ttheta_0$. Thus, for any $\epsilon > 0$, we have $\delta := \inf_{\ttheta \in B_{\ttheta_0}, \|\ttheta- \ttheta_0\| > \epsilon} N_{\infty}(\ttheta) - N_{\infty}(\ttheta_0) > 0$. Hence, combining the two displays above,
$$\PP(\|\thetaiid - \ttheta_0\| > \epsilon) \le \PP(N_{\infty}(\thetaiid) - N_{\infty}(\ttheta_0) > \delta) + \PP(\thetaiid \not \in B_{\ttheta_0}) \to 0.$$

\textbf{Step 2: Limiting distribution.}
The definition of $\thetaiid$ gives $(\nabla N_{n}) (\thetaiid) = 0$.
By a Taylor expansion of $\nabla N_n$ around $\thetaiid \approx \ttheta_0$, we have
\begin{equation}\label{eq:iid est eq}
    \sqrt{n} (\thetaiid - \ttheta_0) = - ((\nabla^2 N_n)(\boldsymbol{\xi}_n))^{-1}  {\sqrt{n} (\nabla N_n) (\ttheta_0)},
\end{equation}
for some $\boldsymbol{\xi}_n \in (\thetaiid, \ttheta_0)$. Note that Step 1 implies $\boldsymbol{\xi}_n \xp \ttheta_0$. 

For simplicity, denote the Hessian as a function $H_n(\ttheta) := \nabla^2 N_n (\ttheta)$.
We first claim that $H_n(\boldsymbol{\xi}_n) \xp I_0(\ttheta_0)$.
We apply Lemma \ref{lem:ULLN high} with $f(\x,\psi) = \bx \bx^\top \sech^2(\psi^\top \x)$ and $\Psi = B_{\ttheta_0}$, to write 
\begin{align}
    H_n(\boldsymbol{\xi}_n) &= \I - \frac{1}{n}\sumin \X_i \X_i^\top \sech^2(\boldsymbol{\xi}_n^\top \X_i) \notag \\
    &= \I - \EE_{\X \sim N_d(\ttheta_0,\I)} \X \X^\top \sech^2(\boldsymbol{\xi}_n^\top \X) + o_p(1 ) \notag \\
    &\xp \I - \EE_{\X \sim N_d(\ttheta_0,\I)} \X \X^\top \sech^2(\ttheta_0^\top \X) = I_0(\ttheta_0). \label{eq:hessian convergence}
\end{align}
The last convergence follows from the continuous mapping theorem.

\vspace{2mm}
Now, the first conclusion in the theorem follows by plugging the above limit in \eqref{eq:iid est eq}:
\begin{equation}\label{eq:iid est eq 2}
    \sqrt{n}(\thetaiid - \ttheta_0) = I_0(\ttheta_0)^{-1} \frac{1}{\sqrt{n}} \sumin (\X_i \tanh(\ttheta_0^\top \X_i) - \ttheta_0) + o_p(1).
\end{equation}
The second conclusion immediately follows by plugging the limiting distribution in \cref{lem: CLT high} to \eqref{eq:iid est eq 2} and simplifying the variance.
\end{proof}

\subsection{Proof of Lemma \ref{lem:sum x_i z_i}}\label{proof:main lemma high tmp}
The main idea for proving \cref{lem:sum x_i z_i} is to use a Taylor expression to simplify the LHS of \eqref{lem:sum x_i z_i} in terms of the linear and quadratic forms of the ``local fields'' $m_i(\W^n) := \sum_{j \neq i} A_n(i,j) W_j$; see eq. \eqref{eq:linear term high tmp}. We use the following two Lemmas that provide moment bounds for local fields under the two different assumptions in \cref{lem:sum x_i z_i}. We state the two Lemmas separately due to technical differences within proofs.
Recall that $\EE^{\QQ_{\ttheta}}$ denotes the conditional expectation with respect to ${\QQ}_{\ttheta}(\W^n)$, and is always conditioned on $\X^n$.

    \begin{lemma}\label{lem:RFIM moment bound}
        Suppose $\W^n \sim \QQ_{\ttheta} = \QQ_{\ttheta,\beta,\A_n,\X^n},$ where $\beta<1$ and $\ttheta, \X^n$ are arbitrary deterministic values. Then, for $$C_1(\ttheta, \X^n) : = \sumin \Big[\sumjn A_n(i,j) \tanh(\ttheta^\top \X_j) \Big]^2,$$ the following holds, where the hidden constant only depends on $\beta$.
        \begin{enumerate}[(a)]
            \item $\EE^{\QQ_{\ttheta}} \Big[\sumin m_i^2(\W^n) \Big] \lesssim n \an+ C_1(\ttheta, \X^n).$
            \item For any real-valued vector $\mathbf{d} = (d_1, \ldots, d_n),$ we have
            $$|\EE^{\QQ_{\ttheta}} \sumin d_i (W_i - \tanh(\ttheta^\top \X_i))| \lesssim \lVert \mathbf{d} \rVert (1 + \sqrt{n \alpha_n^2} + \sqrt{C_1(\ttheta, \X^n)}).$$
        \end{enumerate}
        
    \end{lemma}

    \begin{lemma}\label{lem:W bar moment bound high tmp}
        Suppose $\beta = 1$, $\X^n \sim P_{\ttheta_0, \beta}^{\text{CW}}$, and fix any $\ttheta\in \Theta$. Then, $\EE^{\QQCW} [\bar{W}^2] = O_p\left(\frac{1}{n} \right),$ where the hidden constant is universal.
    \end{lemma}

{Note that the bounds in \cref{lem:RFIM moment bound} involve the quantity $C_1(\ttheta, \X^n)$, which is a complicated function of $\X^n$. However, assuming that $\X^n$ is generated from a true GMM $P_{\ttheta_0, \QQ_{0}}$, we can additionally bound $C_1$ in terms of $\an$. We generalize this claim in the following Lemma.}

\begin{lemma}\label{lem:conditional sum}
    Suppose $\beta <1$, $\Z^n \sim \QQ_{0,\beta, \A_n}$, and $\X^n \sim P_{\ttheta_0,\QQ_{0,\beta, \A_n}}$. Then, the following holds, where $\an = \max_{i=1}^n \sumjn A_n(i,j)^2$ and the hidden constants only depend on $K, C$.
    \begin{enumerate}[(a)]
    \item Let $\phi: \mathbb{R}^d \to \mathbb{R}$ be an odd function with $|\EE(\phi(\X) | Z = z)| \le K$ and $\Var(\phi(\X) | Z = z) \le C$ for $K, C < \infty$ and $z = \pm 1$. Then,
    \begin{equation*}
        \EE \sumin \left(\sumjn A_n(i,j) \phi(\X_j) \right)^2 = O(n \an).
    \end{equation*}
    \item Let $\phi_1, \phi_2: \mathbb{R}^d \to \mathbb{R}$ be odd functions with $|\EE(\phi_a(\X) | Z = z)| \le K_a$ and $\Var(\phi_a(\X) | Z = z) \le C$, $\Cov (\phi_1(\X), \phi_2(\X) | Z = z) \le C$ for $a = 1,2$, and $z = \pm 1$, where $K_a, C < \infty$. Then,
    $$\EE \left( \sumij A_n(i,j) \phi_1(\X_i) \phi_2(\X_j) \right)^2 = O(n^2 \alpha_n^2 + n\an).$$
    \end{enumerate}
    \end{lemma}

\begin{proof}[Proof of Lemma \ref{lem:sum x_i z_i}] 
We separate the proofs under the two different assumptions we have on the Ising model $\QQ_0$. Throughout this proof, all hidden constants  will depend just on $\beta, \|\ttheta_0\|, d$, and \emph{not depend on} $\ttheta$ nor $\X^n$. Also, let $c_i = c_i(\ttheta):= \ttheta^\top \X_i$ denote the random fields of $\QQ_{\ttheta}$. Then, \eqref{eq:uniform bound for sum x_i z_i} can be re-written as:
\begin{equation}\label{eq:uniform bound with c_i}
        \sup_{\ttheta\in \Theta} \Big\lVert \EE^{\QQ_{\ttheta}} \Big[\sumin \X_i W_i \Big] - \sumin \X_i \tanh(c_i) \Big\rVert = o_p\left({\sqrt{n}} \right).
    \end{equation}

\vspace{2mm}
\noindent \textbf{Proof under $\beta <1$ and the mean-field assumption \eqref{eq:mean field assumption}.}
    We first prove \eqref{eq:uniform bound with c_i} for any deterministic $\X^n$ that satisfies the following conditions:
    \begin{enumerate}\setlength\itemsep{0.5em}
        \item[C1.] $C_1(\ttheta, \X^n) = \sumin (\sumjn A_n(i,j) \tanh(c_j))^2 = O( n \an)$,
        \item[C2.] $\sumjn \|\sumin A_n(i,j) \X_{i} \sech^2(c_i)\|^2 = O( n \an)$,
        \item[C3.] $\|\sumij \X_{i} A_n(i,j) \sech^2(c_i) \tanh(c_j)\| = O(n\an + \sqrt{n\an})$,
        \item[C4.] $\max_{i=1}^n \|\X_i\| = O(\sqrt{\log n})$.
    \end{enumerate}
   We re-emphasize that the constants in $O(\cdot)$ terms do not depend on $\ttheta$ nor $\X^n$. We will show at the end of the proof that conditions C1--C4 holds with high probability under $\X^n \sim P_{\ttheta_0, \QQ_{0,\beta,\A_n}}$.

    Let $m_i(\W^n) := \sum_{j \neq i} A_n(i,j) W_j$. Throughout this proof, we abbreviate $m_i(\W^n)$ as $m_i$. Since $\EE(W_i \mid {W_{(-i)}}) = \tanh(\beta m_i + c_i)$,
    \begin{align}
        &\EE^{\QQ_{\ttheta}} \left(\sumin \X_i W_i \right) \notag \\
        =& \EE^{\QQ_{\ttheta}} \left(\sumin \X_i \tanh(\beta m_i + c_i) \right) \notag \\
        =& \EE^{\QQ_{\ttheta}} \left(\sumin \X_i \left(\tanh(c_i) + \beta m_i \sech^2(c_i) + \frac{\beta^2 m_i^2}{2} (\sech^2)'(\beta \xi_i + c_i) \right) \right) \notag \\
        =& \sumin \X_i \tanh(c_i) + \beta \EE^{\QQ_{\ttheta}} \sumin \X_i m_i \sech^2(c_i) + \frac{\beta^2}{2} \EE^{\QQ_{\ttheta}} \sumin \X_i m_i^2 (\sech^2)'(\beta \xi_i + c_i) \label{eq:linear term high tmp} \\
        =& \sumin \X_i \tanh(c_i) + \beta \EE^{\QQ_{\ttheta}} \sumin \X_i m_i \sech^2(c_i) + O\left(n\sqrt{\log n} \an \right). \notag%
    \end{align}
    The last equality uses a union bound with C4, followed \cref{lem:RFIM moment bound}(a) with C1:
    \begin{align*}
        \|\EE^{\QQ_{\ttheta}} \sumin \X_i m_i^2 (\sech^2)'(\beta \xi_i + c_i)\| &\lesssim \sqrt{\log n} \EE^{\QQ_{\ttheta}} \sumin m_i^2 
        \lesssim \sqrt{\log n} (n \an + C_1) = O(n \sqrt{\log n} \an).
    \end{align*}
    
    Now, to conclude \eqref{eq:uniform bound with c_i}, it remains to show that $\EE^{\QQ_{\ttheta}} \sumin \X_{i} m_i \sech^2(c_i)$ 
    is $o(\sqrt{n})$.
    Using the definition of $m_i$, we can write
    \begin{align*}
        \sumin \X_{i} m_i \sech^2(c_i) &= \sumin \sum_{j=1}^n \X_{i} A_n(i,j) W_j \sech^2(c_i) = \sum_{j=1}^n \mathbf{d}_j W_j,
    \end{align*}
    where $\mathbf{d}_j:= \sumin A_n(i,j) \X_{i} \sech^2(c_i).$ 
    Then, by applying Lemma \ref{lem:RFIM moment bound}(b) (second line) we have
    \begin{align*}
        \|\EE^{\QQ_{\ttheta}} &\sum_{j=1}^n \mathbf{d}_j W_j\| \le \|\EE^{\QQ_{\ttheta}} \sum_{j=1}^n \mathbf{d}_j (W_j - \tanh(c_j))\| + %
        \|\sum_{j=1}^n \mathbf{d}_j \tanh(c_j)\| \\
        &\lesssim \sqrt{\sum_{j=1}^n \|\mathbf{d}_j\|^2} (1 + \sqrt{n \alpha_n^2} + \sqrt{C_1(\ttheta, \X^n)}) + \|\sum_{i,j} \X_{i} A_n(i,j) \sech^2(c_i) \tanh(c_j)\| \\
        &= O\big(n\an + \sqrt{n\an}\big) = o(\sqrt{n}).
    \end{align*}
    The third line uses assumptions C1-C3 to get $$\sum_{j=1}^n \|\mathbf{d}_j\|^2 \lesssim n\an, \quad C_1(\ttheta, \X^n) \lesssim n\an, \quad \|\sum_{i,j} \X_{i} A_n(i,j) \sech^2(c_i) \tanh(c_j)\| \lesssim n\an +\sqrt{n\an},$$
    and the mean-field condition $\sqrt{n}\an = o(1)$ to simplify the final bound.

    \vspace{2mm}
    Finally, we prove that C1--C4 holds with high probability, for $\X^n \sim P_{\ttheta_0, \QQ_{0,\beta,\A_n}}$. C1 and C2 follows from applying Lemma \ref{lem:conditional sum}(a) with $\phi_{1, \ttheta}(\x) := \tanh(\ttheta^\top \x)$ and $\phi_{2, \ttheta}(\x) := \x \sech^2(\ttheta^\top \x)$, respectively. Note that $\phi_{2, \ttheta}$ is vector-valued, but we can just apply Lemma \ref{lem:conditional sum}(a) for each coordinate of $\phi_2$, and sum up since $d$ is finite. Here the moment assumptions in Lemma \ref{lem:conditional sum} hold as $$\EE \left[\phi_{a, \ttheta}(\X) \mid Z \right], \quad \Var \left[\phi_{a, \ttheta}(\X) \mid Z \right]$$ can be upper bounded by absolute constants when $a=1$, and by constants that only depend on $\|\ttheta_0\|$ when $a=2$. Next, C3 follows from applying Lemma \ref{lem:conditional sum}(b) with $\phi_{1, \ttheta}$ and each coordinate of $\phi_{2, \ttheta}$. Finally, C4 follows by recalling \eqref{eq:model} to write $\max_{i=1}^n \|\X_i\| \le \|\ttheta_0\| + \max_{i=1}^n \|\mathbf{Y}_i\|$ where $\mathbf{Y}_i \equiv N_d(\mathbf{0}_d, \I)$, and applying the Gaussian maximal inequality: $\max_{i=1}^n \|\mathbf{Y}_i\| = O_p(\sqrt{\log n})$.

\vspace{2mm}
\noindent \textbf{Proof under Curie-Weiss labels at $\beta = 1$.}
Now, we prove \eqref{eq:uniform bound for sum x_i z_i} under the Curie-Weiss RFIM $\W^n \sim \QQ_{\ttheta}^{\text{CW}}$ at $\beta = 1$, for deterministic $\X^n$ that satisfy C4 above and the following condition:
\begin{enumerate}
    \item[C5.] $\sup_{\ttheta \in \Theta} \sumin \X_i \sech^2(c_i) = o(n)$.
\end{enumerate}
Under the Curie-Weiss model, the $m_i = m_i(\W^n)$'s can be written explicitly as $$m_i = \frac{1}{n} \sum_{j \neq i} W_j = \bar{W} - \frac{W_i}{n}.$$
By plugging in this formula to \eqref{eq:linear term high tmp} alongside $\beta = 1$, we get
    \begin{align*}
        &\EE^{\QQCW} \left(\sumin \X_i W_i \right) - \sumin \X_i \tanh(c_i) \\
        =& \EE^{\QQCW} \sumin \X_i m_i \sech^2(c_i) + \frac{1}{2} \EE^{\QQCW} \sumin \X_i m_i^2 (\sech^2)'( \xi_i + c_i) \label{eq:linear term high tmp} \\
        =& \Big(\sumin \X_i \sech^2(c_i)\Big) \EE^{\QQCW} [\bar{W}] - \frac{1}{n}\sumin \X_i \sech^2(c_i) \EE^{\QQCW} [W_i] + O\Big(\sumin \|\X_i\|  \EE^{\QQCW} [m_i^2] \Big) \\
        =& \Big(\sumin \X_i \sech^2(c_i)\Big) \EE^{\QQCW} [\bar{W}] + O\Big( \max_{i=1}^n \| \X_i\| \Big) + O\Big(n \max_{i=1}^n \| \X_i \| \Big[\EE^{\QQCW} [\bar{W}^2] + \frac{1}{n^2} \Big] \Big) \\
        =& o(\sqrt{n}) + O(\sqrt{\log n}) =o(\sqrt{n}) .
    \end{align*}
    In the penultimate line, we used $|W_i| = 1$ and $\|\X_i\| \le \max_{i=1}^n \| \X_i \|$. The final line used moment bounds of $\bar{W}$ from \cref{lem:W bar moment bound high tmp}, alongside conditions C4 and C5.
    
    Now, it remains to show that C4 and C5 hold with high probability for $\X^n \sim P_{\theta_0,\beta}^{\text{CW}}$. C4 follows from the exact same argument in the first segment of the proof. Recalling that $c_i = \ttheta^\top \X_i$, C5 holds because
    \begin{equation*}%
    \begin{aligned}
        \sup_{\theta \in \Theta} \Big|\frac{1}{n}\sumin \left(\X_i \sech^2(c_i) - \EE[\X_i \sech^2(c_i) \mid Z_i] \right) \Big| &\xp 0, \\
        \sup_{\theta \in \Theta} \frac{1}{n}\sumin \EE[\X_i \sech^2(c_i) \mid Z_i] = \bar{Z} \sup_{\theta \in \Theta} K(\ttheta) &\xp 0
    \end{aligned}
    \end{equation*}
    for $K(\ttheta) := \EE[\X_1 \sech^2(c_1) \mid Z_1 = 1]$. 
    Here, the first convergence follows by from the ULLN (see \cref{lem:ULLN high}). The second line uses anti-symmetry of $\x \to \x \sech^2(\ttheta^\top \x)$ to simplify the expression, followed by the LLN for the Curie-Weiss model with $\beta = 1$ to get $\bar{Z} = o_p(1)$ (e.g. see \cite{ellis1978statistics}). Note that $\|K(\ttheta)\| \le \|\ttheta_0\| + 2\sqrt{d}$ for all $\ttheta$ by a similar argument as in \eqref{eq:estimator norm bound}, and is bounded.
\end{proof}

\subsection{Proof of \cref{thm:lower bound high temperature} and \cref{cor:high tmp}}\label{sec:proof high tmp}
We prove Theorem \ref{thm:lower bound high temperature} by doing a one-term Taylor expansion of the log-likelihood ratio, and applying \cref{lem:sum x_i z_i}. %

\begin{proof}[Proof of Theorem \ref{thm:lower bound high temperature}]
Recalling the definition of the normalizing constant $Z_{n, \beta, \A_n}(\ttheta_n, \X^n)$ from \eqref{eq:rfim def}, the likelihood ratio can be simplified as
\begin{align*}
    \frac{dP_{\ttheta_n,\QQ_{0}}}{dP_{\ttheta_0, \QQ_{0}}} (\X^n) &= \frac{\sum_{\bw \in \{-1,1\}^n} \exp \left[ - \frac{n\ttheta_n^\top \ttheta_n}{2} + \frac{\beta}{2} \bw^\top \A_n \bw + \ttheta_n^\top \sumin \X_i w_i \right]}{\sum_{\bw \in \{-1,1\}^n} \exp \left[ - \frac{n\ttheta_0^\top \ttheta_0}{2} + \frac{\beta}{2}\bw^\top \A_n \bw + \ttheta_0^\top \sumin \X_i w_i \right]} \\
    &= \exp \left[ {-\frac{2 \bh^\top \ttheta_0 \sqrt{n} + \bh^\top \bh}{2}} + \log Z_{n, \beta, \A_n}(\ttheta_n, \X^n) - \log Z_{n, \beta, \A_n}(\ttheta_0, \X^n) \right].
\end{align*}
By properties of exponential families, we have
$$\frac{\partial \log Z_{n, \beta, \A_n}(\ttheta, \X^n)}{\partial \ttheta} = \EE^{\QQ_{\ttheta}} \left(\sumin \X_i W_i \right) = \sumin \X_i \tanh(\ttheta^\top \X_i) + o_p\left({\sqrt{n}}\right).$$
Here, the $o_p\left({\sqrt{n}}\right)$ term is uniform in $\ttheta$ due to assumption \eqref{eq:uniform bound for sum x_i z_i}.
Now, by the chain rule, we can write
\begin{align*}
    \log Z_{n, \beta, \A_n} &(\ttheta_n, \X^n) - \log Z_{n, \beta, \A_n}(\ttheta_0, \X^n) = \int_{0}^{\frac{1}{\sqrt{n}}} \bh^\top \left[\sumin \X_i \tanh\left((\ttheta_0 + t \bh)^\top \X_i\right) + o_p(\sqrt{n}) \right] dt      \\
    &= \sumin \log \left( \frac{\cosh(\ttheta_n^\top \X_i)}{\cosh(\ttheta_0^\top \X_i)} \right) + o_p(1) \\
    &= \frac{\bh^\top}{\sqrt{n}} \sumin \X_i \tanh(\ttheta_0^\top \X_i) + \frac{1}{2n} \bh^\top \left(\sumin \X_i \X_i^\top \sech^2(\xxi_n^\top \X_i) \right) \bh + o_p(1) \\
    &= \frac{\bh^\top}{\sqrt{n}} \sumin \X_i \tanh(\ttheta_0^\top \X_i) + \frac{1}{2} \bh^\top \EE_{\X \sim N_d(\ttheta_0, \I)} \X \X^\top \sech^2(\ttheta_0^\top \X) \bh + o_p(1).
\end{align*}
Here, the third line is due to a Taylor expansion with some error term $\xxi_n \in (\ttheta_0,\ttheta_n)$, and the last line used the limit \eqref{eq:hessian convergence}.
Finally, by combining likelihood ratio expansion and the above display, we have
\begin{align*}
    \log \frac{dP_{\ttheta_n,\QQ_{0}}}{dP_{\ttheta_0, \QQ_{0}}} &(\X^n)
    =- \bh^\top \ttheta_0 \sqrt{n} - \frac{1}{2} \bh^\top \bh +\log Z_{n, \beta, \A_n} (\ttheta_n, \X^n) - \log Z_{n, \beta, \A_n}(\ttheta_0, \X^n) \\
    =& \frac{\bh^\top}{\sqrt{n}} \sumin \Big(\X_i \tanh(\ttheta_0^\top \X_i) - \ttheta_0\Big) - \frac{1}{2} \bh^\top \Big(\I - \EE_{\X \sim N_d(\ttheta_0, \I)} \X \X^\top \sech^2(\ttheta_0^\top \X) \Big) \bh + o_p(1) \\
    =& \bh^\top \Delta_{n,\ttheta_0}(\X^n) - \frac{1}{2} \bh^\top I_0(\ttheta_0) \bh + o_p(1).
\end{align*}
Recall the definition of $\Delta_{n, \ttheta_0}$ from \eqref{eq:delta_n} and $I(\ttheta_0)$ from \cref{thm:upper bound iid}. The proof is complete as the limit distribution in \eqref{eq:delta_n} follows from Lemma \ref{lem: CLT high}.
\end{proof}

Finally, we prove Corollary \ref{cor:high tmp} using previous conclusions and Le Cam theory \citep{van2000asymptotic}.
\begin{proof}[Proof of Corollary \ref{cor:high tmp}]
    First fix $\ttheta_0 \in \Theta_1$. 
    Under our assumptions, Theorem \ref{thm:lower bound high temperature} proves that $\{P_{\ttheta, \QQ_0} \}_{\ttheta\in \Theta_1}$ is LAN, where $\beta \le 1$ is fixed.
    Then, Le Cam's first lemma (see Lemma 6.4 in \cite{van2000asymptotic}) shows that $P_{\ttheta_0,\QQ_0}$ and $P_{\ttheta_n,\QQ_{0}}$ are mutually contiguous.
    Also, note that equation \eqref{eq:high tmp iid estimator expansion} in Theorem \ref{thm:upper bound iid} allows us to write
    $$\sqrt{n}(\thetaiid - \ttheta_0) = I_0(\ttheta_0)^{-1} \Delta_{n,\ttheta_0} + o_p(1).$$
    Since \cref{lem: CLT high} gives $\Delta_{n,\ttheta_0} \xrightarrow[P_{\ttheta_0,\beta}]{d} N_d(0, I_0(\ttheta_0)),$ we have 
    \begin{align}
    \begin{pmatrix}
       \vspace{2mm}
        \sqrt{n}(\thetaiid - \ttheta_0) \\ 
        \log \frac{dP_{\ttheta_n,\QQ_{0}}}{dP_{\ttheta_0, \QQ_0}}
    \end{pmatrix} &= \begin{pmatrix}
        I_0(\ttheta_0)^{-1} \Delta_{n,\ttheta_0} \\
        \mathbf{h}^\top \Delta_{n,\ttheta_0} - \frac{1}{2}\bh^\top I_0(\ttheta_0) \bh
    \end{pmatrix} + o_p(1) \\
    & \xrightarrow[P_{\ttheta_0,\QQ_0}]{d} N_{d+1} \left(\begin{pmatrix}
        \mathbf{0}_d \\ -\frac{1}{2} \bh^\top I_0(\ttheta_0) \bh
    \end{pmatrix}, \begin{pmatrix}
        I_0(\ttheta_0)^{-1} & \bh \\ 
        \bh^\top & \bh^\top I_0(\ttheta_0) \bh
    \end{pmatrix} \right). 
    \end{align}
    Then, we can apply Le Cam's third lemma (see Theorem 6.6 in \cite{van2000asymptotic}) to get $\sqrt{n}(\thetaiid - \ttheta_0) \xrightarrow[P_{\ttheta_n,\QQ_0}]{d} N_d(\bh, I_0(\ttheta_0)^{-1})$. Now, the proof is complete by plugging in $\ttheta_n = \ttheta_0 + \frac{\bh}{\sqrt{n}}$ to adjust the centering.
\end{proof}

\begin{supplement}
\stitle{Proof of remaining Theorems}
\sdescription{We prove all low-temperature results from \cref{subsec:low temperature cw} and auxiliary lemmas.}
\end{supplement}

\bibliographystyle{imsart-number.bst} %
\bibliography{reference}       %

\begin{thebibliography}{54}

\bibitem{balakrishnan2017statistical}
\begin{barticle}[author]
\bauthor{\bsnm{Balakrishnan},~\bfnm{Sivaraman}\binits{S.}},
  \bauthor{\bsnm{Wainwright},~\bfnm{Martin~J}\binits{M.~J.}} \AND
  \bauthor{\bsnm{Yu},~\bfnm{Bin}\binits{B.}}
(\byear{2017}).
\btitle{Statistical guarantees for the EM algorithm: From population to
  sample-based analysis}.
\bjournal{Annals of Statistics}
\bvolume{45}
\bpages{77-120}.
\end{barticle}
\endbibitem

\bibitem{basak2017universality}
\begin{barticle}[author]
\bauthor{\bsnm{Basak},~\bfnm{Anirban}\binits{A.}} \AND
  \bauthor{\bsnm{Mukherjee},~\bfnm{Sumit}\binits{S.}}
(\byear{2017}).
\btitle{Universality of the mean-field for the Potts model}.
\bjournal{Probability Theory and Related Fields}
\bvolume{168}
\bpages{557--600}.
\end{barticle}
\endbibitem

\bibitem{besag1974spatial}
\begin{barticle}[author]
\bauthor{\bsnm{Besag},~\bfnm{Julian}\binits{J.}}
(\byear{1974}).
\btitle{Spatial interaction and the statistical analysis of lattice systems}.
\bjournal{Journal of the Royal Statistical Society: Series B (Methodological)}
\bvolume{36}
\bpages{192--225}.
\end{barticle}
\endbibitem

\bibitem{besag1986statistical}
\begin{barticle}[author]
\bauthor{\bsnm{Besag},~\bfnm{Julian}\binits{J.}}
(\byear{1986}).
\btitle{On the statistical analysis of dirty pictures}.
\bjournal{Journal of the Royal Statistical Society Series B: Statistical
  Methodology}
\bvolume{48}
\bpages{259--279}.
\end{barticle}
\endbibitem

\bibitem{bhattacharya2018inference}
\begin{barticle}[author]
\bauthor{\bsnm{Bhattacharya},~\bfnm{Bhaswar~B.}\binits{B.~B.}} \AND
  \bauthor{\bsnm{Mukherjee},~\bfnm{Sumit}\binits{S.}}
(\byear{2018}).
\btitle{{Inference in Ising models}}.
\bjournal{Bernoulli}
\bvolume{24}
\bpages{493 -- 525}.
\bdoi{10.3150/16-BEJ886}
\end{barticle}
\endbibitem

\bibitem{bhattacharya2021sharp}
\begin{barticle}[author]
\bauthor{\bsnm{Bhattacharya},~\bfnm{Sohom}\binits{S.}},
  \bauthor{\bsnm{Mukherjee},~\bfnm{Rajarshi}\binits{R.}} \AND
  \bauthor{\bsnm{Ray},~\bfnm{Gourab}\binits{G.}}
(\byear{2025}).
\btitle{Sharp Signal Detection under Ferromagnetic Ising Models}.
\bjournal{IEEE Transactions on Information Theory}.
\end{barticle}
\endbibitem

\bibitem{bickel1996inference}
\begin{barticle}[author]
\bauthor{\bsnm{Bickel},~\bfnm{Peter~J}\binits{P.~J.}} \AND
  \bauthor{\bsnm{Ritov},~\bfnm{Ya'acov}\binits{Y.}}
(\byear{1996}).
\btitle{Inference in hidden Markov models I: Local asymptotic normality in the
  stationary case}.
\bjournal{Bernoulli}
\bvolume{2}
\bpages{199--228}.
\end{barticle}
\endbibitem

\bibitem{bickel1998asymptotic}
\begin{barticle}[author]
\bauthor{\bsnm{Bickel},~\bfnm{Peter~J}\binits{P.~J.}},
  \bauthor{\bsnm{Ritov},~\bfnm{Ya’acov}\binits{Y.}} \AND
  \bauthor{\bsnm{Ryden},~\bfnm{Tobias}\binits{T.}}
(\byear{1998}).
\btitle{Asymptotic normality of the maximum-likelihood estimator for general
  hidden Markov models}.
\bjournal{The Annals of Statistics}
\bvolume{26}
\bpages{1614--1635}.
\end{barticle}
\endbibitem

\bibitem{blei2017variational}
\begin{barticle}[author]
\bauthor{\bsnm{Blei},~\bfnm{David~M}\binits{D.~M.}},
  \bauthor{\bsnm{Kucukelbir},~\bfnm{Alp}\binits{A.}} \AND
  \bauthor{\bsnm{McAuliffe},~\bfnm{Jon~D}\binits{J.~D.}}
(\byear{2017}).
\btitle{Variational inference: A review for statisticians}.
\bjournal{Journal of the American statistical Association}
\bvolume{112}
\bpages{859--877}.
\end{barticle}
\endbibitem

\bibitem{bulinski2017conditional}
\begin{barticle}[author]
\bauthor{\bsnm{Bulinski},~\bfnm{Alexander~V}\binits{A.~V.}}
(\byear{2017}).
\btitle{Conditional central limit theorem}.
\bjournal{Theory of Probability \& Its Applications}
\bvolume{61}
\bpages{613--631}.
\end{barticle}
\endbibitem

\bibitem{chatterjee2007estimation}
\begin{barticle}[author]
\bauthor{\bsnm{Chatterjee},~\bfnm{Sourav}\binits{S.}}
(\byear{2007}).
\btitle{{Estimation in spin glasses: A first step}}.
\bjournal{The Annals of Statistics}
\bvolume{35}
\bpages{1931 -- 1946}.
\bdoi{10.1214/009053607000000109}
\end{barticle}
\endbibitem

\bibitem{chatterjee2019central}
\begin{barticle}[author]
\bauthor{\bsnm{Chatterjee},~\bfnm{Sourav}\binits{S.}}
(\byear{2019}).
\btitle{Central limit theorem for the free energy of the random field Ising
  model}.
\bjournal{Journal of Statistical Physics}
\bvolume{175}
\bpages{185--202}.
\end{barticle}
\endbibitem

\bibitem{chatzis2010infinite}
\begin{barticle}[author]
\bauthor{\bsnm{Chatzis},~\bfnm{Sotirios~P}\binits{S.~P.}} \AND
  \bauthor{\bsnm{Tsechpenakis},~\bfnm{Gabriel}\binits{G.}}
(\byear{2010}).
\btitle{The infinite hidden Markov random field model}.
\bjournal{IEEE Transactions on Neural Networks}
\bvolume{21}
\bpages{1004--1014}.
\end{barticle}
\endbibitem

\bibitem{clifford1971markov}
\begin{barticle}[author]
\bauthor{\bsnm{Clifford},~\bfnm{P}\binits{P.}} \AND
  \bauthor{\bsnm{Hammersley},~\bfnm{JM}\binits{J.}}
(\byear{1971}).
\btitle{Markov fields on finite graphs and lattices}.
\end{barticle}
\endbibitem

\bibitem{comets1991asymptotics}
\begin{barticle}[author]
\bauthor{\bsnm{Comets},~\bfnm{Francis}\binits{F.}} \AND
  \bauthor{\bsnm{Gidas},~\bfnm{Basilis}\binits{B.}}
(\byear{1991}).
\btitle{Asymptotics of maximum likelihood estimators for the Curie-Weiss
  model}.
\bjournal{The Annals of Statistics}
\bpages{557--578}.
\end{barticle}
\endbibitem

\bibitem{dagan2021learning}
\begin{binproceedings}[author]
\bauthor{\bsnm{Dagan},~\bfnm{Yuval}\binits{Y.}},
  \bauthor{\bsnm{Daskalakis},~\bfnm{Constantinos}\binits{C.}},
  \bauthor{\bsnm{Dikkala},~\bfnm{Nishanth}\binits{N.}} \AND
  \bauthor{\bsnm{Kandiros},~\bfnm{Anthimos~Vardis}\binits{A.~V.}}
(\byear{2021}).
\btitle{{Learning Ising models from one or multiple samples}}.
In \bbooktitle{Proceedings of the 53rd Annual ACM SIGACT Symposium on Theory of
  Computing}
\bpages{161--168}.
\end{binproceedings}
\endbibitem

\bibitem{daskalakis2019regression}
\begin{binproceedings}[author]
\bauthor{\bsnm{Daskalakis},~\bfnm{Constantinos}\binits{C.}},
  \bauthor{\bsnm{Dikkala},~\bfnm{Nishanth}\binits{N.}} \AND
  \bauthor{\bsnm{Panageas},~\bfnm{Ioannis}\binits{I.}}
(\byear{2019}).
\btitle{Regression from dependent observations}.
In \bbooktitle{Proceedings of the 51st Annual ACM SIGACT Symposium on Theory of
  Computing}
\bpages{881--889}.
\end{binproceedings}
\endbibitem

\bibitem{daskalakis2017ten}
\begin{binproceedings}[author]
\bauthor{\bsnm{Daskalakis},~\bfnm{Constantinos}\binits{C.}},
  \bauthor{\bsnm{Tzamos},~\bfnm{Christos}\binits{C.}} \AND
  \bauthor{\bsnm{Zampetakis},~\bfnm{Manolis}\binits{M.}}
(\byear{2017}).
\btitle{Ten steps of EM suffice for mixtures of two Gaussians}.
In \bbooktitle{Conference on Learning Theory}
\bpages{704--710}.
\bpublisher{PMLR}.
\end{binproceedings}
\endbibitem

\bibitem{deb2023fluctuations}
\begin{barticle}[author]
\bauthor{\bsnm{Deb},~\bfnm{Nabarun}\binits{N.}} \AND
  \bauthor{\bsnm{Mukherjee},~\bfnm{Sumit}\binits{S.}}
(\byear{2023}).
\btitle{Fluctuations in mean-field Ising models}.
\bjournal{The Annals of Applied Probability}
\bvolume{33}
\bpages{1961--2003}.
\end{barticle}
\endbibitem

\bibitem{dembo2010gibbs}
\begin{barticle}[author]
\bauthor{\bsnm{Dembo},~\bfnm{Amir}\binits{A.}} \AND
  \bauthor{\bsnm{Montanari},~\bfnm{Andrea}\binits{A.}}
(\byear{2010}).
\btitle{Gibbs measures and phase transitions on sparse random graphs}.
\end{barticle}
\endbibitem

\bibitem{ding2023new}
\begin{barticle}[author]
\bauthor{\bsnm{Ding},~\bfnm{Jian}\binits{J.}},
  \bauthor{\bsnm{Song},~\bfnm{Jian}\binits{J.}} \AND
  \bauthor{\bsnm{Sun},~\bfnm{Rongfeng}\binits{R.}}
(\byear{2023}).
\btitle{A new correlation inequality for Ising models with external fields}.
\bjournal{Probability Theory and Related Fields}
\bvolume{186}
\bpages{477--492}.
\end{barticle}
\endbibitem

\bibitem{dwivedi2020singularity}
\begin{barticle}[author]
\bauthor{\bsnm{Dwivedi},~\bfnm{Raaz}\binits{R.}},
  \bauthor{\bsnm{Ho},~\bfnm{Nhat}\binits{N.}},
  \bauthor{\bsnm{Khamaru},~\bfnm{Koulik}\binits{K.}},
  \bauthor{\bsnm{Wainwright},~\bfnm{Martin~J}\binits{M.~J.}},
  \bauthor{\bsnm{Jordan},~\bfnm{Michael~I}\binits{M.~I.}} \AND
  \bauthor{\bsnm{Yu},~\bfnm{Bin}\binits{B.}}
(\byear{2020}).
\btitle{Singularity, misspecification and the convergence rate of EM}.
\bjournal{The Annals of Statistics}
\bvolume{48}
\bpages{3161--3182}.
\end{barticle}
\endbibitem

\bibitem{ellis1978statistics}
\begin{barticle}[author]
\bauthor{\bsnm{Ellis},~\bfnm{Richard~S}\binits{R.~S.}} \AND
  \bauthor{\bsnm{Newman},~\bfnm{Charles~M}\binits{C.~M.}}
(\byear{1978}).
\btitle{The statistics of Curie-Weiss models}.
\bjournal{Journal of Statistical Physics}
\bvolume{19}
\bpages{149--161}.
\end{barticle}
\endbibitem

\bibitem{franccois2006bayesian}
\begin{barticle}[author]
\bauthor{\bsnm{Fran{\c{c}}ois},~\bfnm{Olivier}\binits{O.}},
  \bauthor{\bsnm{Ancelet},~\bfnm{Sophie}\binits{S.}} \AND
  \bauthor{\bsnm{Guillot},~\bfnm{Gilles}\binits{G.}}
(\byear{2006}).
\btitle{Bayesian clustering using hidden Markov random fields in spatial
  population genetics}.
\bjournal{Genetics}
\bvolume{174}
\bpages{805--816}.
\end{barticle}
\endbibitem

\bibitem{friedli2017statistical}
\begin{bbook}[author]
\bauthor{\bsnm{Friedli},~\bfnm{Sacha}\binits{S.}} \AND
  \bauthor{\bsnm{Velenik},~\bfnm{Yvan}\binits{Y.}}
(\byear{2017}).
\btitle{Statistical mechanics of lattice systems: a concrete mathematical
  introduction}.
\bpublisher{Cambridge University Press}.
\end{bbook}
\endbibitem

\bibitem{ganguly2023amortized}
\begin{barticle}[author]
\bauthor{\bsnm{Ganguly},~\bfnm{Ankush}\binits{A.}},
  \bauthor{\bsnm{Jain},~\bfnm{Sanjana}\binits{S.}} \AND
  \bauthor{\bsnm{Watchareeruetai},~\bfnm{Ukrit}\binits{U.}}
(\byear{2023}).
\btitle{Amortized variational inference: A systematic review}.
\bjournal{Journal of Artificial Intelligence Research}
\bvolume{78}
\bpages{167--215}.
\end{barticle}
\endbibitem

\bibitem{gheissari2018concentration}
\begin{barticle}[author]
\bauthor{\bsnm{Gheissari},~\bfnm{Reza}\binits{R.}},
  \bauthor{\bsnm{Lubetzky},~\bfnm{Eyal}\binits{E.}} \AND
  \bauthor{\bsnm{Peres},~\bfnm{Yuval}\binits{Y.}}
(\byear{2018}).
\btitle{{Concentration inequalities for polynomials of
  contracting Ising models}}.
\bjournal{Electronic Communications in Probability}
\bvolume{23}
\bpages{1 -- 12}.
\bdoi{10.1214/18-ECP173}
\end{barticle}
\endbibitem

\bibitem{ghosal2020joint}
\begin{barticle}[author]
\bauthor{\bsnm{Ghosal},~\bfnm{Promit}\binits{P.}} \AND
  \bauthor{\bsnm{Mukherjee},~\bfnm{Sumit}\binits{S.}}
(\byear{2020}).
\btitle{Joint estimation of parameters in Ising model}.
\bjournal{The Annals of Statistics}
\bvolume{48}
\bpages{785--810}.
\end{barticle}
\endbibitem

\bibitem{goffinet1992testing}
\begin{barticle}[author]
\bauthor{\bsnm{Goffinet},~\bfnm{Bruno}\binits{B.}},
  \bauthor{\bsnm{Loisel},~\bfnm{Patrice}\binits{P.}} \AND
  \bauthor{\bsnm{Laurent},~\bfnm{Beatrice}\binits{B.}}
(\byear{1992}).
\btitle{Testing in normal mixture models when the proportions are known}.
\bjournal{Biometrika}
\bvolume{79}
\bpages{842--846}.
\end{barticle}
\endbibitem

\bibitem{he2023hidden}
\begin{barticle}[author]
\bauthor{\bsnm{He},~\bfnm{Yihan}\binits{Y.}},
  \bauthor{\bsnm{Liu},~\bfnm{Han}\binits{H.}} \AND
  \bauthor{\bsnm{Fan},~\bfnm{Jianqing}\binits{J.}}
(\byear{2023}).
\btitle{Hidden Clique Inference in Random Ising Model I: the planted random
  field Curie-Weiss model}.
\bjournal{arXiv preprint arXiv:2310.00667}.
\end{barticle}
\endbibitem

\bibitem{ising1924beitrag}
\begin{bphdthesis}[author]
\bauthor{\bsnm{Ising},~\bfnm{Ernst}\binits{E.}}
(\byear{1924}).
\btitle{Beitrag zur theorie des ferro-und paramagnetismus},
\btype{PhD thesis},
\bpublisher{Grefe \& Tiedemann Hamburg, Germany}.
\end{bphdthesis}
\endbibitem

\bibitem{karagulyan2024adaptive}
\begin{barticle}[author]
\bauthor{\bsnm{Karagulyan},~\bfnm{Vahe}\binits{V.}} \AND
  \bauthor{\bsnm{Ndaoud},~\bfnm{Mohamed}\binits{M.}}
(\byear{2024}).
\btitle{Adaptive Mean Estimation in the Hidden {M}arkov sub-{G}aussian Mixture
  Model}.
\bjournal{arXiv preprint arXiv:2406.12446}.
\end{barticle}
\endbibitem

\bibitem{klusowski2016statistical}
\begin{barticle}[author]
\bauthor{\bsnm{Klusowski},~\bfnm{Jason~M}\binits{J.~M.}} \AND
  \bauthor{\bsnm{Brinda},~\bfnm{WD}\binits{W.}}
(\byear{2016}).
\btitle{Statistical guarantees for estimating the centers of a two-component
  Gaussian mixture by EM}.
\bjournal{arXiv preprint arXiv:1608.02280}.
\end{barticle}
\endbibitem

\bibitem{kunsch1995hidden}
\begin{barticle}[author]
\bauthor{\bsnm{Kunsch},~\bfnm{Hans}\binits{H.}},
  \bauthor{\bsnm{Geman},~\bfnm{Stuart}\binits{S.}} \AND
  \bauthor{\bsnm{Kehagias},~\bfnm{Athanasios}\binits{A.}}
(\byear{1995}).
\btitle{Hidden Markov random fields}.
\bjournal{The annals of applied probability}
\bvolume{5}
\bpages{577--602}.
\end{barticle}
\endbibitem

\bibitem{lai2015asymptotically}
\begin{barticle}[author]
\bauthor{\bsnm{Lai},~\bfnm{Tze~Leung}\binits{T.~L.}} \AND
  \bauthor{\bsnm{Lim},~\bfnm{Johan}\binits{J.}}
(\byear{2015}).
\btitle{Asymptotically efficient parameter estimation in hidden {M}arkov
  spatio-temporal random fields}.
\bjournal{Statistica Sinica}
\bpages{403--421}.
\end{barticle}
\endbibitem

\bibitem{lee2024rfim}
\begin{barticle}[author]
\bauthor{\bsnm{Lee},~\bfnm{Seunghyun}\binits{S.}},
  \bauthor{\bsnm{Deb},~\bfnm{Nabarun}\binits{N.}} \AND
  \bauthor{\bsnm{Mukherjee},~\bfnm{Sumit}\binits{S.}}
(\byear{2025}).
\btitle{Fluctuations in random field Ising models}.
\bjournal{arXiv preprint arXiv:2503.21152}.
\end{barticle}
\endbibitem

\bibitem{lee2025clt}
\begin{barticle}[author]
\bauthor{\bsnm{Lee},~\bfnm{Seunghyun}\binits{S.}},
  \bauthor{\bsnm{Deb},~\bfnm{Nabarun}\binits{N.}} \AND
  \bauthor{\bsnm{Mukherjee},~\bfnm{Sumit}\binits{S.}}
(\byear{2025}).
\btitle{{CLT in high-dimensional Bayesian linear regression with low SNR}}.
\bjournal{arXiv preprint arXiv:2507.23285}.
\end{barticle}
\endbibitem

\bibitem{lehmann2006theory}
\begin{bbook}[author]
\bauthor{\bsnm{Lehmann},~\bfnm{Erich~L}\binits{E.~L.}} \AND
  \bauthor{\bsnm{Casella},~\bfnm{George}\binits{G.}}
(\byear{2006}).
\btitle{Theory of point estimation}.
\bpublisher{Springer Science \& Business Media}.
\end{bbook}
\endbibitem

\bibitem{mukherjee2018global}
\begin{barticle}[author]
\bauthor{\bsnm{Mukherjee},~\bfnm{Rajarshi}\binits{R.}},
  \bauthor{\bsnm{Mukherjee},~\bfnm{Sumit}\binits{S.}} \AND
  \bauthor{\bsnm{Yuan},~\bfnm{Ming}\binits{M.}}
(\byear{2018}).
\btitle{Global testing against sparse alternatives under Ising models}.
\bjournal{The Annals of Statistics}
\bvolume{46}
\bpages{2062--2093}.
\end{barticle}
\endbibitem

\bibitem{mukherjee2022testing}
\begin{barticle}[author]
\bauthor{\bsnm{Mukherjee},~\bfnm{Rajarshi}\binits{R.}} \AND
  \bauthor{\bsnm{Ray},~\bfnm{Gourab}\binits{G.}}
(\byear{2022}).
\btitle{On testing for parameters in {I}sing models}.
\bjournal{Annales de l'Institut Henri Poincare (B) Probabilites et
  statistiques}
\bvolume{58}
\bpages{164--187}.
\end{barticle}
\endbibitem

\bibitem{mukherjee2022estimation}
\begin{barticle}[author]
\bauthor{\bsnm{Mukherjee},~\bfnm{Somabha}\binits{S.}},
  \bauthor{\bsnm{Son},~\bfnm{Jaesung}\binits{J.}} \AND
  \bauthor{\bsnm{Bhattacharya},~\bfnm{Bhaswar~B}\binits{B.~B.}}
(\byear{2022}).
\btitle{Estimation in tensor Ising models}.
\bjournal{Information and Inference: A Journal of the IMA}
\bvolume{11}
\bpages{1457--1500}.
\end{barticle}
\endbibitem

\bibitem{mukherjee2021efficient}
\begin{barticle}[author]
\bauthor{\bsnm{Mukherjee},~\bfnm{Somabha}\binits{S.}},
  \bauthor{\bsnm{Son},~\bfnm{Jaesung}\binits{J.}},
  \bauthor{\bsnm{Ghosh},~\bfnm{Swarnadip}\binits{S.}} \AND
  \bauthor{\bsnm{Mukherjee},~\bfnm{Sourav}\binits{S.}}
(\byear{2024}).
\btitle{Efficient estimation in tensor Curie-Weiss and Erd{\H{o}}s-R{\'e}nyi
  Ising models}.
\bjournal{Electronic Journal of Statistics}
\bvolume{18}
\bpages{2405--2449}.
\end{barticle}
\endbibitem

\bibitem{ndaoud2022sharp}
\begin{barticle}[author]
\bauthor{\bsnm{Ndaoud},~\bfnm{Mohamed}\binits{M.}}
(\byear{2022}).
\btitle{Sharp optimal recovery in the two component Gaussian mixture model}.
\bjournal{The Annals of Statistics}
\bvolume{50}
\bpages{2096--2126}.
\end{barticle}
\endbibitem

\bibitem{pyun2002robust}
\begin{binproceedings}[author]
\bauthor{\bsnm{Pyun},~\bfnm{Kyungsuk}\binits{K.}},
  \bauthor{\bsnm{Won},~\bfnm{Chee~Sun}\binits{C.~S.}},
  \bauthor{\bsnm{Lim},~\bfnm{Johan}\binits{J.}} \AND
  \bauthor{\bsnm{Gray},~\bfnm{Robert~M}\binits{R.~M.}}
(\byear{2002}).
\btitle{Robust image classification based on a non-causal hidden Markov Gauss
  mixture model}.
In \bbooktitle{Proceedings. International Conference on Image Processing}
\bvolume{3}
\bpages{785--788}.
\bpublisher{IEEE}.
\end{binproceedings}
\endbibitem

\bibitem{van2000asymptotic}
\begin{bbook}[author]
\bauthor{\bparticle{Van~der} \bsnm{Vaart},~\bfnm{Aad~W}\binits{A.~W.}}
(\byear{2000}).
\btitle{Asymptotic statistics}
\bvolume{3}.
\bpublisher{Cambridge university press}.
\end{bbook}
\endbibitem

\bibitem{vershynin2018high}
\begin{bbook}[author]
\bauthor{\bsnm{Vershynin},~\bfnm{Roman}\binits{R.}}
(\byear{2018}).
\btitle{High-dimensional probability: An introduction with applications in data
  science}
\bvolume{47}.
\bpublisher{Cambridge university press}.
\end{bbook}
\endbibitem

\bibitem{wainwright2008graphical}
\begin{barticle}[author]
\bauthor{\bsnm{Wainwright},~\bfnm{Martin~J}\binits{M.~J.}},
  \bauthor{\bsnm{Jordan},~\bfnm{Michael~I}\binits{M.~I.}} \betal{et~al.}
(\byear{2008}).
\btitle{Graphical models, exponential families, and variational inference}.
\bjournal{Foundations and Trends{\textregistered} in Machine Learning}
\bvolume{1}
\bpages{1--305}.
\end{barticle}
\endbibitem

\bibitem{wu2020optimal}
\begin{barticle}[author]
\bauthor{\bsnm{Wu},~\bfnm{Yihong}\binits{Y.}} \AND
  \bauthor{\bsnm{Yang},~\bfnm{Pengkun}\binits{P.}}
(\byear{2020}).
\btitle{Optimal estimation of {G}aussian mixtures via denoised method of
  moments}.
\bjournal{The Annals of Statistics}
\bvolume{48}
\bpages{1981--2007}.
\end{barticle}
\endbibitem

\bibitem{wu2021randomly}
\begin{barticle}[author]
\bauthor{\bsnm{Wu},~\bfnm{Yihong}\binits{Y.}} \AND
  \bauthor{\bsnm{Zhou},~\bfnm{Harrison~H}\binits{H.~H.}}
(\byear{2021}).
\btitle{Randomly initialized EM algorithm for two-component Gaussian mixture
  achieves near optimality in $O (\sqrt{n})$ iterations}.
\bjournal{Mathematical Statistics and Learning}
\bvolume{4}.
\end{barticle}
\endbibitem

\bibitem{xu2016global}
\begin{barticle}[author]
\bauthor{\bsnm{Xu},~\bfnm{Ji}\binits{J.}},
  \bauthor{\bsnm{Hsu},~\bfnm{Daniel~J}\binits{D.~J.}} \AND
  \bauthor{\bsnm{Maleki},~\bfnm{Arian}\binits{A.}}
(\byear{2016}).
\btitle{Global analysis of expectation maximization for mixtures of two
  gaussians}.
\bjournal{Advances in Neural Information Processing Systems}
\bvolume{29}.
\end{barticle}
\endbibitem

\bibitem{xu2018benefits}
\begin{barticle}[author]
\bauthor{\bsnm{Xu},~\bfnm{Ji}\binits{J.}},
  \bauthor{\bsnm{Hsu},~\bfnm{Daniel~J}\binits{D.~J.}} \AND
  \bauthor{\bsnm{Maleki},~\bfnm{Arian}\binits{A.}}
(\byear{2018}).
\btitle{Benefits of over-parameterization with EM}.
\bjournal{Advances in Neural Information Processing Systems}
\bvolume{31}.
\end{barticle}
\endbibitem

\bibitem{xu2023inference}
\begin{barticle}[author]
\bauthor{\bsnm{Xu},~\bfnm{Yuanzhe}\binits{Y.}} \AND
  \bauthor{\bsnm{Mukherjee},~\bfnm{Sumit}\binits{S.}}
(\byear{2023}).
\btitle{Inference in Ising models on dense regular graphs}.
\bjournal{The Annals of Statistics}
\bvolume{51}
\bpages{1183--1206}.
\end{barticle}
\endbibitem

\bibitem{zhang2006schur}
\begin{bbook}[author]
\bauthor{\bsnm{Zhang},~\bfnm{Fuzhen}\binits{F.}}
(\byear{2006}).
\btitle{The Schur complement and its applications}
\bvolume{4}.
\bpublisher{Springer Science \& Business Media}.
\end{bbook}
\endbibitem

\bibitem{zhang2022mean}
\begin{barticle}[author]
\bauthor{\bsnm{Zhang},~\bfnm{Yihan}\binits{Y.}} \AND
  \bauthor{\bsnm{Weinberger},~\bfnm{Nir}\binits{N.}}
(\byear{2022}).
\btitle{Mean estimation in high-dimensional binary Markov Gaussian mixture
  models}.
\bjournal{Advances in Neural Information Processing Systems}
\bvolume{35}
\bpages{19673--19686}.
\end{barticle}
\endbibitem

\end{thebibliography}

\begin{appendix}

\section*{Proof of remaining Theorems}
The Supplementary Material is organized as follows. In \cref{sec:proof of main results}, we prove all low-temperature results stated in \cref{subsec:low temperature cw}. 
{
We begin by introducing common Lemmas and notations throughout the proofs in \cref{subsec:additional notation}. Next, in \cref{subsec:uniqueness of low tmp minimizer proof}, we prove \cref{lem:uniqueness of minimizer}. We prove the main theorems for the upper and lower bound (Theorems \ref{thm:upper bound} and \ref{thm:lower bound}) in \cref{subsec:main thm proof low tmp} and \cref{subsec:proof of lower bound low}, respectively.
}

{
We prove all remaining Lemmas in \cref{sec:proof of lemmas}, where we first prove the high/low temperature ULLN and CLTs in \cref{subsec:ulln clt proof}. We prove the high temperature concentration results for Ising models on general graphs (Lemmas \ref{lem:RFIM moment bound} and \ref{lem:conditional sum}) in \cref{subsec:Ising model concentration}.
We prove the concentration results specific to the random field Curie-Weiss model (Lemmas \ref{lem: u n consistency}, \ref{lem:W bar moment bound low tmp}, and \ref{lem:W bar moment bound high tmp}) in \cref{subsec:rfim concentration}. Finally, we prove \cref{lem:variance simplification} in \cref{subsec:proof of variance matching}.
}

\subsection{Proof of results in \cref{subsec:low temperature cw}}\label{sec:proof of main results}

\subsubsection{Additional Lemmas and notations}\label{subsec:additional notation}
We first state the low temperature analogs for the conditional LLN and CLTs that we saw in Lemmas \ref{lem:ULLN high} and \ref{lem: CLT high}. These results will be used multiple times throughout \cref{sec:proof of main results}. We defer the proofs of these Lemmas to Section \ref{sec:proof of lemmas}. 

We first state the low-temperature ULLN. Note that this immediately implies the non-uniform LLN in \cref{lem:conditional limit }.

\begin{lemma}[low temperature ULLN]\label{lem:ULLN}
Suppose $\beta > 1$, and that $\A_n$ satisfy Assumptions \ref{defi:mean-field} and \ref{defi:regularity}. Let $\X^n \sim P_{\ttheta_0, \beta, \A_n}$.  For a $k$-dimensional compact set $\Psi$, let $f:\mathbb{R}^d \times \Psi \to \mathbb{R}$ be a bivariate function that satisfies all conditions given in \cref{lem:ULLN high}. Then, we have $\PP(\bar{Z} <0 : \bar{\X} \in \Theta_1) \to 0$ and 
    $$\sup_{\psi \in \Psi} \left| \frac{1}{n} \sumin f(\X_i, \psi) - \EE_{\ttheta_0} f(\X, \psi) \right| : (\bar{\X} \in \Theta_1) \xp 0.$$
    Similarly, we have $\PP(\bar{Z} >0 : \bar{\X} \in \Theta_2) \to 0$ and
    $$\sup_{\psi \in \Psi} \left| \frac{1}{n} \sumin f(\X_i, \psi) - \EE_{-\ttheta_0} f(\X, \psi) \right| : (\bar{\X} \in \Theta_2) \xp 0.$$
    The above conclusions also hold when $f$ is vector-valued (say, $k'$-dimensional for some finite $k'$) and the absolute value is replaced by any vector norm.
\end{lemma}

The following Lemma computes the limiting distribution of the statistic $\sqrt{n} (\nabla M_n) (m, \ttheta_0)$, where the function $M_n$ is introduced in \eqref{eq:mean-field estimator definition}. %

\begin{lemma}[low temperature CLT]\label{lem: CLT}
Suppose $\beta > 1$, and that $\A_n$ satisfy Assumptions \ref{defi:mean-field} and \ref{defi:regularity}. Let $\X^n \sim P_{\ttheta_0, \beta, \A_n}$.  Then, we have
\begin{align*}
    \sqrt{n} (\nabla M_n) (m, \ttheta_0) : (\bar{\X} \in \Theta_1) &\xd N_{d+1}\left(\mathbf{0}_{d+1}, \Sigma \right), \\
    \sqrt{n} (\nabla M_n) (-m, \ttheta_0) : (\bar{\X} \in \Theta_2) &\xd N_{d+1}\left(\mathbf{0}_{d+1}, \tilde{\Sigma} \right).
\end{align*}
Here, $\Sigma$ and $\tilde{\Sigma}$ are $(d+1) \times (d+1)$ matrices that will be defined below in \cref{defi:sigma}(c).
\end{lemma}

Next, we introduce additional notations, which are required to explicitly state the limiting variance $\Sigma$ as well as simplify further computations.
\begin{defi}\label{defi:sigma}
Given $\beta > 1$ and $\ttheta_0 \in \Theta_1$, we define the following the quantities.
\begin{enumerate}[(a),itemsep=2mm]
    \item 
    For $z = \pm 1$, let
\begin{align*}
    \mu_z := \EE\big[\tanh(\beta m + \ttheta_0^{\top} \X) \mid Z=z\big], \quad \boldsymbol{\nu}_z := \EE\big[\X \tanh(\beta m + \ttheta_0^{\top} \X) \mid Z=z\big].
\end{align*}

\item Define each component of the gradient $\nabla M_n$ by setting
    \begin{align*}
        F_1(u, \ttheta) :&= \beta \left(u - \frac{1}{n} \sumin \tanh(\beta u + \ttheta^{\top} \X_i) \right), \\
        F_2(u, \ttheta) :&= \ttheta- \frac{1}{n} \sumin \X_i \tanh(\beta u + \ttheta^{\top} \X_i),
    \end{align*}
    so that $\nabla M_n = \begin{pmatrix} F_1 \\ F_2 \end{pmatrix}$.

\item Define a $(d+1) \times (d+1)$ matrix $\Sigma = \begin{pmatrix}
    \sigma_{1,1} & \ssigma_{1,2}^{\top} \\
    \ssigma_{1,2} & \ssigma_{2,2}
\end{pmatrix}$ as
\begin{align*}
    \Sigma &:= \EE_{Z \sim \text{Rad}(\frac{1+m}{2})} \Big[\Var \Big( \begin{pmatrix}
    \beta \\ \X
\end{pmatrix}\tanh(\beta m + \ttheta_0^\top \X) \mid Z \Big) \Big] + \frac{C(\beta)}{4} \begin{pmatrix}
    \beta(\mu_1 - \mu_{-1}) \\ \boldsymbol{\nu}_1 - \boldsymbol{\nu}_{-1}
\end{pmatrix} \begin{pmatrix}
    \beta(\mu_1 - \mu_{-1}) \\ \boldsymbol{\nu}_1 - \boldsymbol{\nu}_{-1}
\end{pmatrix}^\top.
\end{align*}
Here, $C(\beta) := \frac{1-m^2}{1-\beta(1-m^2)}$ is the limiting variance of $\bar{Z}$ under the Curie-Weiss model (see \cref{lem:Ising CLT}). Also define $\tilde{\Sigma} := \begin{pmatrix}
        \sigma_{1,1}   & -\ssigma_{1,2}^\top \\
        - \ssigma_{1,2} & \ssigma_{2,2}
    \end{pmatrix}$.

\item Define constants $\alpha_0 \in \mathbb{R}, \alpha_1 \in \mathbb{R}^d, \alpha_2 \in \mathbb{R}^{d\times d}$ by
\begin{align*}
    \alpha_0 &:= \EE_{\ttheta_0} \sech^2(\beta m + \ttheta_0^{\top} \X), \\
    \alpha_1 &:= \EE_{\ttheta_0} \X \sech^2(\beta m + \ttheta_0^{\top} \X), \\
    \alpha_2 &:= \EE_{\ttheta_0} \X \X^\top \sech^2(\beta m + \ttheta_0^{\top} \X).
\end{align*}
\end{enumerate}
\end{defi}

\subsubsection{Proof of Lemma \ref{lem:uniqueness of minimizer}}\label{subsec:uniqueness of low tmp minimizer proof}
\begin{proof}[Proof of Lemma \ref{lem:uniqueness of minimizer}]
The proof proceeds by a KL divergence argument similar to \cref{lem:uniqueness iid}. Fix any $\beta > 1$. For any $u \in (-1, 1)$ and $ \ttheta \in \Theta_1$, define a distribution $$\PP_{u, \ttheta} \equiv \frac{e^{\beta u}}{e^{\beta u}+e^{-\beta u}} N_d(\ttheta, \I) + \frac{e^{-\beta u}}{e^{\beta u}+e^{-\beta u}} N_d(-\ttheta, \I),$$ which has density $$p_{u, \ttheta}(\x) = \frac{\exp \left[ - \frac{\x^{\top} \x}{2} - \frac{\ttheta^\top \ttheta}{2} + \log \cosh(\beta u + \ttheta^{\top} \x) \right]}{(\sqrt{2\pi})^d \cosh(\beta u)}.$$
Note that $\{\PP_{u, \ttheta}: \ttheta\in \Theta_1 \}$ is an identifiable family, which is immediate by writing out the first two moments. Hence, for any $(u,\ttheta) \neq (m, \ttheta_0)$, 
\begin{align*}
    0 < \text{KL}(\PP_{m, \ttheta_0} \| \PP_{u, \ttheta}) = \EE^{\PP_{m, \ttheta_0}} \Big[&- \frac{\ttheta_0^\top \ttheta_0}{2} + \log \cosh(\beta m + \ttheta_0^{\top} \X) - \log \cosh(\beta m) \\
    & + \frac{\ttheta^\top \ttheta}{2}- \log \cosh(\beta u + \ttheta^{\top} \X)  + \log \cosh(\beta u) \Big].
\end{align*}

Now setting a function $g(u) := - \frac{\beta u^2}{2} + \log \cosh(\beta u)$, we can write
\begin{align*}
    & M_{\infty}(u, \ttheta) - M_{\infty}(m, \ttheta_0) \\
    =& \frac{\beta (u^2 - m^2)}{2} + \frac{\ttheta^\top \ttheta - \ttheta_0^\top \ttheta_0}{2} - \EE^{\PP_{m, \ttheta_0}} \log \cosh(\beta u + \ttheta^\top \X) + \EE^{\PP_{m, \ttheta_0}} \log \cosh(\beta m + \ttheta_0^\top \X) \\
    =& \text{KL}(\PP_{m, \ttheta_0} \| \PP_{u, \ttheta}) - g(u) + g(m) > -g(u)+g(m).
\end{align*}
Hence, to show the RHS is positive, it suffices to prove $g(m) \ge g(u)$ for all $u$. Standard calculus shows that $g$ is a symmetric function with $g'(u) > 0$ for $0 < u<m$ and $g'(u) < 0$ for $u >m$, and hence maximized at $u = \pm m$ (e.g. see pg. 144-145 in \cite{dembo2010gibbs}). This completes the proof. %

\end{proof}

\begin{remark}
    One immediate consequence of \cref{lem:uniqueness of minimizer} is that $(\nabla M_{\infty}) (m, \ttheta_0) = 0$, i.e.
    \begin{equation}\label{eq:m,theta identies}
        m = \EE_{\ttheta_0} \tanh(\beta m + \ttheta_0^\top \X), ~ \ttheta_0 = \EE_{\ttheta_0} \X \tanh(\beta m + \ttheta_0^\top \X).
    \end{equation}
    These identities can also be proved directly by using the definition of $\EE_{\ttheta_0}$ and the fact that $m = \tanh(\beta m)$.
    However, unlike \cref{lem:uniqueness iid}, multiple solutions of $\nabla M_{\infty} = 0$ may exist.
\end{remark}

\subsubsection{Proof of Theorem \ref{thm:upper bound}}\label{subsec:main thm proof low tmp}
{We prove Theorem \ref{thm:upper bound} by modifying the usual argument for deriving asymptotic normality of M-estimators. One subtlety arises in terms of simplifying the limiting variance as $I_\beta(\ttheta_0)^{-1}$.} This involves nontrivial computations, which we formally state in the following Lemma. Note that part (b) also establishes the invertibility of $I_\beta(\ttheta_0)$.

\begin{lemma}\label{lem:variance simplification}
    Under the notations from \cref{defi:info} and \cref{defi:sigma}, the following holds.
    \begin{enumerate}[(a)]
        \item $1 - \beta \alpha_0 > 0$ and $\gamma_{1,1} > 0$.
        \item For $\ddelta := \frac{\ggamma_{1,2}}{\gamma_{1,1}}$,
        \begin{align}\label{eq:variance information eq}
        I_\beta(\ttheta_0) = \ddelta \sigma_{1,1} \ddelta^{\top} - \ssigma_{1,2} \ddelta^{\top} - \ddelta \ssigma_{1,2}^{\top} + \ssigma_{2,2} \succ 0.
\end{align}
\end{enumerate}
\end{lemma}

\begin{proof}[Proof of Theorem \ref{thm:upper bound}]
    The positive definiteness of $I_\beta(\ttheta_0)$ follows from \cref{lem:variance simplification}(b). To prove the desired CLT for $\thetamf$, we claim more general \textit{joint CLTs} for $(\hat{U}_n, \thetamf)$:
    \begin{align}
        \sqrt{n} \begin{pmatrix} \hat{U}_n - m \\ \thetamf - \ttheta_0 \end{pmatrix} : (\bar{\X} \in \Theta_1) \xd N_{d+1} \left(0, \Gamma^{-1} \Sigma \Gamma^{-1} \right), \label{eq:CLT_upper bound}\\
        \sqrt{n} \begin{pmatrix} \hat{U}_n + m \\ \thetamf - \ttheta_0 \end{pmatrix} : (\bar{\X} \in \Theta_2) \xd N_{d+1} \left(0, \tilde{\Gamma}^{-1} \tilde{\Sigma} \tilde{\Gamma}^{-1} \right). \label{eq:CLT_upper bound 2}
    \end{align}
    Here, $\Gamma = \begin{pmatrix}
        \gamma_{1,1} & \ggamma_{1,2}^\top \\
        \ggamma_{2,1} & \ggamma_{2,2}
    \end{pmatrix}$ is the $(d+1)\times (d+1)$ matrix in \cref{defi:info}, and we define $\tilde{\Gamma} := \begin{pmatrix}
        \gamma_{1,1} & -\ggamma_{1,2}^\top \\
        - \ggamma_{2,1} & \ggamma_{2,2}
    \end{pmatrix}$ as a modification. Also recall $(d+1)\times (d+1)$ matrices $\Sigma, \tilde{\Sigma}$ from part (c) of \cref{defi:sigma}.

    We mainly prove \eqref{eq:CLT_upper bound}, and then illustrate how the argument modifies for \eqref{eq:CLT_upper bound 2}. Recall from \eqref{eq:mean-field estimator definition} that $(\hat{U}_n, \thetamf)$ is a solution of the ($d+1$)-dimensional equation $\mathbf{0}_{d+1} = (\nabla M_n)(u, \ttheta) = \begin{pmatrix} F_1(u, \ttheta) \\ F_2(u, \ttheta) \end{pmatrix}$. 
    By a 1-term Taylor expansion, we have
    \begin{equation}\label{eq:taylor low tmp clt}
        0 =\begin{pmatrix} F_1(\hat{U}_n,\thetamf)  \\ 
    F_2(\hat{U}_n,\thetamf) \end{pmatrix} = \begin{pmatrix}  F_1(m, \ttheta_0) \\ F_2(m, \ttheta_0) \end{pmatrix} + H_n (\xxi_n) \begin{pmatrix} \hat{U}_n - m \\ \thetamf - \ttheta_0 \end{pmatrix}
    \end{equation}
    for some $\xxi_n$, which implies
    \begin{align}\label{eq:mf estimator expansion}
        \sqrt{n} \begin{pmatrix} \hat{U}_n - m \\ \thetamf - \ttheta_0 \end{pmatrix} = - (H_n(\xxi_n))^{-1} \sqrt{n} \begin{pmatrix} F_1(m, \ttheta_0) \\ F_2(m, \ttheta_0) \end{pmatrix}.
    \end{align}
    Here, $H_n$ denotes the Hessian of $M_n$, and its invertibility will be shown later in the proof (see Step 2). We derive the limiting distribution through the following three steps.
    
    \textbf{Step 1: Consistency.} We first show $(\hat{U}_n, \thetamf) : (\bar{\X} \in \Theta_1) \xp (m, \ttheta_0)$.
    Note that Lemma \ref{lem:ULLN} gives $$\sup_{|u| \le 1, \ttheta\in\Theta_1 \cap \{\ttheta: \|\ttheta\| \le \|\ttheta_0\| + 2\sqrt{d} \}} |M_n(u,\ttheta) - M_{\infty}(u,\ttheta)| : (\bar{\X} \in \Theta_1) \xp 0,$$ 
    and $\|\thetamf\| \le \|\ttheta_0\| + 2\sqrt{d}$ with high probability (this follows from \eqref{eq:estimator norm bound}).
    Viewing our estimator as a M-estimator and repeating the proof argument in Step 1 of \cref{thm:upper bound iid}, it suffices to show that ($m, \ttheta_0$) is a unique minimizer of $M_{\infty}$, which follows from Lemma \ref{lem:uniqueness of minimizer}. %

    \vspace{2mm}
    \textbf{Step 2: Limit of $H_n(\xxi_n)$.} We claim that $H_n(\xxi_n) : (\bar{\X}\in \Theta_1) \xp \Gamma$. Step 1 implies that $\xxi_n : (\bar{\X} \in \Theta_1) \xp \begin{pmatrix}
        m \\ \ttheta_0
    \end{pmatrix}$.
    By the same argument as \eqref{eq:hessian convergence} in \cref{thm:upper bound iid} (except for using Lemma \ref{lem:ULLN} (conditional ULLN) on behalf of its unconditional analog), we can write
    $$H_n(\xxi_n) = (\nabla^2 M_n) (\xxi_n) = (\nabla^2 M_\infty) (\xxi_n) + o_p(1) = (\nabla^2 M_\infty)(m,\ttheta_0) + o_p(1) = \Gamma + o_p(1).$$
    Note that the positive definiteness of $\Gamma$ is equivalent to $\gamma_{1,1} > 0$ and $I_\beta(\ttheta_0) = \ggamma_{2,2} - \ggamma_{1,2} \gamma_{1,1}^{-1} \ggamma_{1,2} \succ 0$ (e.g. see page 34 in \cite{zhang2006schur}), both of which follow from Lemma \ref{lem:variance simplification}.
    Since $\Gamma$ is positive definite, $H_n(\xxi_n)$ is also positive definite with high probability.

    \textbf{Step 3: Limit of $\sqrt{n} \begin{pmatrix} F_1(m, \ttheta_0) \\ F_2(m, \ttheta_0) \end{pmatrix}$.} The normal limit of $\sqrt{n} \begin{pmatrix} F_1(m, \ttheta_0) \\ F_2(m, \ttheta_0) \end{pmatrix}$ is given in Lemma \ref{lem: CLT}. Now, applying  Slutsky's theorem on \eqref{eq:mf estimator expansion} gives \eqref{eq:CLT_upper bound}.

    \vspace{2mm}
    Similarly, we claim the limit \eqref{eq:CLT_upper bound 2}, which is conditioned on $\bar{\X} \in \Theta_2$. We briefly sketch the main changes. First, using the ULLN conditioned on $\bar{\X} \in \Theta_2$, Lemma \ref{lem:conditional limit } can be modified as 
    $$\frac{1}{n} \sumin \cosh(\beta u + \ttheta^\top \X_i) : (\bar{\X} \in \Theta_2) \xp \EE_{-\ttheta_0} \log \cosh(\beta u + \ttheta^\top \X) = \EE_{\ttheta_0} \log \cosh(-\beta u + \ttheta^\top \X), $$
    (here $\EE_{-\ttheta_0}$ is the natural modification of that in \cref{defi:EE ttheta})
    and $M_n(u, \ttheta)$ converges pointwise to $M_{\infty} (-u, \ttheta)$.
    By Lemma \ref{lem:uniqueness of minimizer}, ${M}_{\infty}$ 
    is minimized at $(-m, \ttheta_0)$. The remaining argument follows from doing the Taylor expansion \eqref{eq:taylor low tmp clt} around $(\hat{U}_n, \thetamf) \approx (-m, \ttheta_0)$, and noting that the limit of $H_n(-m, \ttheta_0)$ and $\sqrt{n}\begin{pmatrix}
        F_1(-m, \ttheta_0) \\ F_2(-m, \ttheta_0)
    \end{pmatrix}$
    is $\tilde{\Gamma}$ and $N_{d+1}(0, \tilde{\Sigma})$, respectively.

    \vspace{2mm}
    It remains to prove the final conclusion (individual limiting distribution for $\thetamf$). Recalling from \cref{defi:info} and \cref{defi:sigma} that $\Gamma,\Sigma$ are defined as $2 \times 2$ block matrices, it suffices to show that the $(2,2)$th block in $\Gamma^{-1} \Sigma \Gamma^{-1}$ and $\tilde{\Gamma}^{-1} \tilde{\Sigma} \tilde{\Gamma}^{-1}$ are both equal to $I_{\beta}(\ttheta_0)^{-1}$. Using the formula for the inverse of a non-singular block matrix, $\Gamma^{-1}$ can be written as
    $$\Gamma^{-1} = \begin{pmatrix}
        \star &~~~ -\ddelta^{\top} I_\beta(\ttheta_0)^{-1}  \\
        - I_\beta(\ttheta_0)^{-1} \ddelta &~ I_\beta(\ttheta_0)^{-1}
    \end{pmatrix}.$$
    Here, $\ddelta = \frac{\ggamma_{1,2}}{\gamma_{1,1}}$, and $\star$ denotes some value that will not be used in further computations. 
    By expanding $\Gamma^{-1} \Sigma \Gamma^{-1}$ using the block matrix representation and applying the identity in Lemma \ref{lem:variance simplification}(b), we have
    \begin{align*}
        (\Gamma^{-1} \Sigma \Gamma^{-1})_{2,2} &= I_\beta(\ttheta_0)^{-1} \left(\ddelta \sigma_{1,1} \ddelta^{\top} - \ssigma_{1,2} \ddelta^{\top} - \ddelta \ssigma_{1,2}^{\top} + \ssigma_{2,2} \right)  I_\beta(\ttheta_0)^{-1} =  I_\beta(\ttheta_0)^{-1}.
    \end{align*}
    The $(2,2)$th block of $\tilde{\Gamma}^{-1} \tilde{\Sigma} \tilde{\Gamma}^{-1}$ can be computed similarly.
    Note that $\Gamma^{-1} \Sigma \neq \I$ in general, and this identity is a nontrivial result.
\end{proof}

\begin{remark}
    By focusing on the $\thetamf - \ttheta_0$ term of \eqref{eq:mf estimator expansion} and plugging-in the conclusions of Steps 2 and 3, we get
    \begin{align}\label{eq:upper bound expansion low tmp}
        \sqrt{n}(\thetamf - \ttheta_0) = - I_\beta(\ttheta_0)^{-1} {\sqrt{n} (- \ddelta F_1(m, \ttheta_0) + F_2(m, \ttheta_0))} + o_p(1).
    \end{align}
    This expansion will be used later to prove \cref{cor:low tmp}.
\end{remark}

\subsubsection{Proof of Theorem \ref{thm:lower bound} and Corollary \ref{cor:low tmp}}\label{subsec:proof of lower bound low}
{
We first illustrate why proving the low temperature lower bound is more challenging compared to the high temperature case (\cref{thm:lower bound high temperature}).
Recall that \cref{thm:lower bound high temperature} directly follows by applying the uniform control (over $\ttheta \in \Theta$) in \cref{lem:sum x_i z_i} to a \emph{first order} Taylor expansion of the log likelihood ratio. However, such a strong result does not hold in the low temperature regime, even for the Curie-Weiss case considered here. Indeed, \eqref{eq:sum x_i w_i low temperature} is stated only for $\ttheta = \ttheta_0$. This is because the measure $\PP_{\ttheta_0}$ (see \cref{defi:EE ttheta}) is no longer symmetric in the low-temperature regime, and influences the expectation of the RFIM $\W^n$. Consequently, we have to conduct a more careful analysis of the likelihood ratio, by conducting a \emph{second order} Taylor expansion.
}

For this purpose, it is necessary to understand the second order behavior (variance) of the statistic $\sumin \X_i W_i$, and we require the following Lemmas regarding the Curie-Weiss RFIM. We use \cref{lem: u n consistency} to understand the limit of $U_n$ (see \eqref{eq:u_n defi}). \cref{lem:W bar moment bound low tmp} provides tight moment bounds for $\W^n$ by exploiting the low-rank structure of the Curie-Weiss coupling matrix. While both Lemmas are stated conditional on $\bar{\X} \in \Theta_1$, analogous statements conditioned on $\bar{\X} \in \Theta_2$ can be derived similarly.

\begin{lemma}\label{lem: u n consistency}
Suppose $\beta > 1, \X^n \sim P_{\ttheta_0,\beta}^{\text{CW}}$, and define $U_n$ as in \eqref{eq:u_n defi}. Then, $U_n : (\bar{\X} \in \Theta_1) \xp m$. Furthermore, for a sequence $\xxi_n:= \xxi_n(\X^n) \in \mathbb{R}^d$ such that $\xxi_n : (\bar{\X} \in \Theta_1) \xp \ttheta_0$, define $$\tilde{f}_n(v) := \frac{\beta v^2}{2} - \frac{1}{n} \sumin \log \cosh(\beta v + \xxi_n^\top \X_i)$$ and $V_n:= \argmin_v \tilde{f}_n(v)$. Then, $V_n : (\bar{\X} \in \Theta_1) \xp m$.
\end{lemma}

Before stating \cref{lem:W bar moment bound low tmp}, we introduce an additional notation. For a sequence of random variables $\{Y_n\}_{n \ge 1}$ and a deterministic sequence $\{a_n\}_{n \ge 1}$, we write $Y_n \lesssim_P a_n$ when there exists an absolute constant $K>0$ such that $Y_n \le K a_n$ with high probability. Also, recall $\alpha_0 = \EE_{\ttheta_0} \sech^2(\beta m + \ttheta_0^\top \X)$ from \cref{defi:sigma}.

\begin{lemma}\label{lem:W bar moment bound low tmp}
    Suppose $\beta >1,$ $\X^n \sim P_{\ttheta_0,\beta}^{\text{CW}}$. Let $\xxi_n:= \xxi_n(\X^n)$ satisfy $\|\xxi_n - \ttheta_0\| \lesssim \frac{1}{\sqrt{n}}$ surely, and suppose $\W^n \mid \X^n \sim \QQ_{\xxi_n,\beta}^{\text{CW}}$. Also, consider an auxiliary random variable $Y_n \mid \W^n, \X^n \sim N(\bar{W}, \frac{1}{n \beta})$ and let $V_n$ be the random variable defined in Lemma \ref{lem: u n consistency}. 
    \begin{enumerate}[(a)]
    \item $W_i \mid Y_n, \X^n$'s are independent with mean $\tanh(\beta Y_n + \xxi_n^\top \X_i)$. Also, $Y_n \mid \X^n$ has a density proportional to $e^{-\ft_n(Y_n)}$.
    \item  $n \EE ( (Y_n- V_n)^2 : \X^n, (\bar{\X} \in \Theta_1) ) \xp \frac{1}{\beta(1 - \beta \alpha_0)}$ and $\EE ( |Y_n- V_n|^q : \X^n, (\bar{\X} \in \Theta_1) ) \lesssim_P \frac{1}{n^{p/2}}$ for $q > 0$.
    \item $\EE ( (\bar{W}- V_n)^2 : \X^n, (\bar{\X} \in \Theta_1) ) \lesssim_P \frac{1}{n}$. %
    \item $|\EE \left(\bar{W}- V_n : \X^n, (\bar{\X} \in \Theta_1) \right)| \lesssim_P \frac{1}{n}$ and $|\EE \left(Y_n- V_n : \X^n, (\bar{\X} \in \Theta_1) \right)| \lesssim_P \frac{1}{n}$. %
    \end{enumerate}
    Here, high probability statements are with respect to $\X^n$, and the hidden constants only depend on $\beta,\alpha_0$. 
\end{lemma}

Now, we are ready to prove Theorem \ref{thm:lower bound}.

\begin{proof}[Proof of Theorem \ref{thm:lower bound}]
Recall the normalizing constant 
$$Z_{n, \beta}(\ttheta, \X^n) = Z_{n, \beta}^{\text{CW}}(\ttheta, \X^n) = \sum_{\bw \in \{-1,1\}^n} e^{\frac{n\beta \bar{w}^2}{2}  + \ttheta^\top \sumin \X_i w_i} $$
from \eqref{eq:rfim def}. By standard computations for exponential families, we have
\begin{align*}
    \frac{\partial \log Z_{n, \beta}(\ttheta, \X^n)}{\partial \ttheta} = \EE^{\QQ_{\ttheta}} \left(\sumin \X_i W_i \mid \X^n \right), \\
    \frac{\partial^2 \log Z_{n, \beta}(\ttheta, \X^n)}{\partial \ttheta^2}  = \Var^{\QQ_{\ttheta}} \left(\sumin \X_i W_i \mid \X^n \right).
\end{align*}
Here, $\EE^{\QQ_{\ttheta}}$ and $\Var^{\QQ_{\ttheta}}$ denotes the \emph{conditional} expectation and variance with respect to ${\QQ}_{\ttheta}(\W^n \mid \X^n)$. %
Then, a two-term Taylor expansion gives
\begin{align*}
    \log \frac{dP_{\ttheta_n, \beta}}{dP_{\ttheta_0, \beta}} &(\X^n) =  -\frac{2\bh^\top \ttheta_0 \sqrt{n} + \bh^\top \bh}{2} + \log Z_{n, \beta}(\ttheta_n, \X^n) - \log Z_{n, \beta}(\ttheta_0, \X^n) \\
    &= -\frac{2 \bh^\top \ttheta_0 \sqrt{n} + 
    \bh^\top \bh}{2} + \frac{\bh^\top}{\sqrt{n}} \EE^{\QQ_{\ttheta_0}} \left(\sumin \X_i W_i \right) + \frac{1}{2n} \bh^\top \Var^{\QQ_{\xxi_n}} \left(\sumin \X_i W_i \right) \bh \\
    &= \sqrt{n}\bh^\top \left(\frac{1}{n} \sumin \X_i \EE^{\QQ_{\ttheta_0}} W_i  - \ttheta_0 \right) - \frac{1}{2} \bh^\top \left(\I - \frac{1}{n}\Var^{\QQ_{\xxi_n}} \left(\sumin \X_i W_i \right) \right) \bh.
\end{align*}
Here, $\xxi_n \in (\ttheta_0, \ttheta_n)$ and only depends on $\X^n$. To show the LAN expansion, it suffices to prove the following three claims. Note that the first claim is exactly \eqref{eq:sum x_i w_i low temperature} from the main text.
\begin{enumerate}
    \item[Claim 1.] $\frac{1}{n} \sumin \X_i \EE^{\QQ_{\ttheta_0}} W_i = \frac{1}{n} \sumin \X_i \tanh (\beta U_n + \ttheta_0^\top \X_i) + O_p\left( \frac{1}{n} \right).$
    \item[Claim 2.] $\tilde{\Delta}_n = \sqrt{n} \left(\frac{1}{n} \sumin \X_i \tanh(\beta U_n + \ttheta_0^\top \X_i) - \ttheta_0 \right) \xrightarrow[P_{\ttheta_0, \beta}]{d} N_d(\mathbf{0}_d, I_\beta(\ttheta_0)).$
    \item[Claim 3.] $\frac{1}{n}\Var^{\QQ_{\xxi_n}} (\sumin \X_i W_i) \xrightarrow[P_{\ttheta_0, \beta}]{p} \I - I_\beta(\ttheta_0).$
\end{enumerate}

\textbf{Claim 1: Expanding the linear term.}
    Using Lemma \ref{lem:W bar moment bound low tmp}(c), (d) with $\xxi_n = \ttheta_0$, we have
    $$\EE^{\QQCW} \left( \bar{W}-U_n : (\bar{\X} \in \Theta_1) \right) \lesssim_P \frac{1}{n},\quad \EE^{\QQCW} \left( (\bar{W}-U_n)^2 : (\bar{\X} \in \Theta_1) \right) \lesssim_P \frac{1}{n}.$$
    Note that the same result also holds conditioned on $\bar{\X} \in \Theta_2$. Set $\bar{W}_{(-i)} := \frac{1}{n} \sum_{j \neq i} W_j$ and note that $W_i \mid (W_j : j \neq i)$ is a Radamacher distribution with mean $\tanh(\beta \bar{W}_{(-i)} + \ttheta_0^\top \X_i)$.
    By consecutive Taylor expansions (in the 2nd and 3rd line) alongside the moment bounds, we have 
    \begin{align*}
    \frac{1}{n} &\sumin \X_i \EE^{\QQ_{\ttheta_0}} W_i = \frac{1}{n} \sumin \X_i \EE^{\QQ_{\ttheta_0}} \tanh(\beta \bar{W}_{(-i)} + \ttheta_0^\top \X_i)  \\
    &= \frac{1}{n} \sumin \X_i \EE^{\QQ_{\ttheta_0}}  \tanh(\beta \bar{W} + \ttheta_0^\top \X_i) + O_p \left( \frac{1}{n} \right) \\
    &= \frac{1}{n} \sumin \X_i \EE^{\QQ_{\ttheta_0}} \Big(\tanh(\beta U_n + \ttheta_0^\top \X_i) + \beta (\bar{W} - U_n) \sech^2(\beta U_n + \ttheta_0^\top \X_i) \\
    &\quad + \frac{\beta^2 (\bar{W} - U_n)^2}{2} (\sech^2)'(\beta \eta_n + \ttheta_0^\top \X_i) \Big) + O_p \left( \frac{1}{n} \right) \\
    &= \frac{1}{n} \sumin \X_i \tanh(\beta U_n + \ttheta_0^\top \X_i) + \frac{\beta \EE^{\QQ_{\ttheta_0}} (\bar{W} - U_n )}{n} \sumin \X_i \sech^2(\beta U_n + \ttheta_0^\top \X_i) + O_p \left( \frac{1}{n} \right) \\
    &= \frac{1}{n} \sumin \X_i \tanh(\beta U_n + \ttheta_0^\top \X_i) + O_p \left( \frac{1}{n} \right).
\end{align*}

\textbf{Claim 2: Computing the limiting distribution of $\tilde{\Delta}_n$.} 
Next, we prove a CLT for $\tilde{\Delta}_n$.
Note that the following conditional law of $\tilde{\Delta}_n$ implies the unconditional result, so it suffices to prove:
$$\tilde{\Delta}_n = \sqrt{n} \left(\frac{1}{n} \sumin \X_i \tanh(\beta U_n + \ttheta_0^\top \X_i) - \ttheta_0 \right) : (\bar{\X} \in \Theta_a) \xrightarrow[P_{\ttheta_0, \beta}]{d}N_d (0, I_\beta(\ttheta_0)),\quad a=1,2.$$
Without the loss of generality, we prove the claim conditioned on $\bar{\X} \in \Theta_1$. We begin by writing out $U_n$.
Using the first order condition for $U_n$ (recall the definition in \eqref{eq:u_n defi}), we have
\begin{align*}
    U_n &= \frac{1}{n} \sumin \tanh(\beta U_n + \ttheta_0^\top \X_i) \\
    &= \frac{1}{n} \sumin \tanh(\beta m + \ttheta_0^\top \X_i) + \frac{\beta (U_n - m)}{n} \sumin \sech^2(\beta \kappa_n + \ttheta_0^\top \X_i)
\end{align*}
for some $\kappa_n \in (m, U_n)$. By subtracting both sides by $m$ and rearranging terms, we can write
\begin{align*}
    U_n - m = \frac{ \frac{1}{n} \sumin \tanh(\beta m + \ttheta_0^\top \X_i) - m}{1 - \frac{\beta}{n} \sumin \sech^2(\beta \kappa_n + \ttheta_0^\top \X_i)}.
\end{align*}

Now, by a Taylor approximation of $U_n \approx m$, we have
\begin{align*}
    & \sumin \Big[\X_i \tanh(\beta U_n + \ttheta_0^\top \X_i) - \ttheta_0\Big] \\ 
    = & \sumin \Big[\X_i \tanh(\beta m + \ttheta_0^\top \X_i) + \beta (U_n-m) \sumin \X_i \sech^2(\beta \eta_n + \ttheta_0^\top \X_i) - \ttheta_0\Big] \\
    = & \sumin (\X_i \tanh(\beta m + \ttheta_0^\top \X_i) - \ttheta_0) \\
    &\quad \quad + \beta \left(\sumin \tanh(\beta m + \ttheta_0^\top \X_i) - m \right) \frac{\frac{1}{n} \sumin \X_i \sech^2(\beta \eta_n + \ttheta_0^\top \X_i) }{1 - \frac{\beta}{n} \sumin \sech^2(\beta \kappa_n + \ttheta_0^\top \X_i)} \\
    = & - n F_2(m, \ttheta_0) - n F_1(m, \ttheta_0) \frac{\frac{1}{n} \sumin \X_i \sech^2(\beta \eta_n + \ttheta_0^\top \X_i) }{1 - \frac{\beta}{n} \sumin \sech^2(\beta \kappa_n + \ttheta_0^\top \X_i)}.
\end{align*}
By Lemma \ref{lem: u n consistency}, $U_n : (\bar{\X} \in \Theta_1) \xp m$ so we have $\eta_n, \kappa_n : (\bar{\X}\in \Theta_1) \xp m$. Then, Lemma \ref{lem:ULLN} gives 
$$\frac{\frac{1}{n} \sumin \X_i \sech^2(\beta \eta_n + \ttheta_0^\top \X_i) }{1 - \frac{\beta}{n} \sumin \sech^2(\beta \kappa_n + \ttheta_0^\top \X_i)} : (\bar{\X} \in \Theta_1) \xp \frac{\EE_{\ttheta_0} \X \sech^2(\beta m + \ttheta_0^\top \X)}{1 - \beta \EE_{\ttheta_0} \sech^2(\beta m + \ttheta_0^\top \X)} = - \frac{\ggamma_{1,2}}{\gamma_{1,1}} = -\ddelta.$$
Hence,
\begin{align}
    \tilde{\Delta}_n &= - \sqrt{n} \left(F_2(m, \ttheta_0) + F_1(m, \ttheta_0) \frac{\frac{1}{n} \sumin \X_i \sech^2(\beta \eta_n + \ttheta_0^\top \X_i) }{1 - \frac{\beta}{n} \sumin \sech^2(\beta \kappa_n + \ttheta_0^\top \X_i)} \right) \notag \\
    &= - \sqrt{n} \left(F_2(m, \ttheta_0) - \ddelta F_1(m, \ttheta_0) \right) + o_p(1) = \sqrt{n} \begin{pmatrix}
        \ddelta & -\I
    \end{pmatrix}(\nabla M_n)(m,\ttheta_0) + o_p(1). \label{eq:Delta expansion}
\end{align}
Recalling the limiting distribution of $(\nabla M_n)(m,\ttheta_0)$ from Lemma \ref{lem: CLT}, Slutsky's theorem gives
$$\tilde{\Delta}_n : (\bar{\X} \in \Theta_1) \xd N_d \left(0, \ddelta \sigma_{1,1} \ddelta^\top - \ddelta \ssigma_{1,2}^\top - \ssigma_{1,2} \ddelta^\top + \ssigma_{2,2} \right).$$
The claim follows by simplifying the variance using Lemma \ref{lem:variance simplification}(b).

\vspace{2mm}
\textbf{Claim 3: Expanding the variance term.}
Recall from the beginning of the proof that $\xxi_n \in (\ttheta_0, \ttheta_n)$ depends on $\X^n$ but not on $\W^n$, and note that $\xxi_n \xp \ttheta_0$.
We write
$$\Var^{\QQ_{\xxi_n}}(\sumin \X_i W_i) = \underbrace{\sumin \X_i \Var^{\QQ_{\xxi_n}} (W_i) \X_i^\top}_{:= \mathbf{C}_n} + \underbrace{\sum_{i \neq j} \X_i \Cov^{\QQ_{\xxi_n}} (W_i, W_j) \X_j^\top }_{:=\mathbf{D}_n}$$
and bound the two terms separately. For simplicity, we prove this claim assuming that $\bar{\X} \in \Theta_1$ and we omit the conditioning on $(\X^n, \bar{\X} \in \Theta_1)$ in each line. Define a random variable $V_n = V_n(\X^n)$ as in \cref{lem: u n consistency}.

First, note that $\Var^{\QQ_{\xxi_n}} (W_i) = 1 - \left[\EE^{\QQ_{\xxi_n}} W_i \right]^2$. Set $\bar{W}_{(-i)} := \frac{1}{n} \sum_{j \neq i} W_j$ and note that $W_i \mid (W_j : j \neq i)$ is a Radamacher distribution with mean $\tanh(\beta \bar{W}_{(-i)} + \xxi_n^\top \X_i)$. By Taylor expansions, we have
\begin{align*}
    \EE^{\QQ_{\xxi_n}} W_i &= \EE^{\QQ_{\xxi_n}} \tanh(\beta \bar{W}_{(-i)} + \xxi_n^\top \X_i) \\
    &= \EE^{\QQ_{\xxi_n}} \tanh(\beta \bar{W} + \xxi_n^\top \X_i) + O\Big(\frac{1}{n}\Big) \\
    &= \EE^{\QQ_{\xxi_n}} \left[ \tanh(\beta m + \xxi_n^\top \X_i) + \beta(\bar{W} - m) \sech^2(\beta \rho_n + \xxi_n^\top \X_i)\right]+ O\Big(\frac{1}{n}\Big) \\
    &= \tanh(\beta m + \xxi_n^\top \X_i) + O_p\Big(\frac{1}{\sqrt{n}} + |V_n - m| \Big).
\end{align*}
Here, the last equality uses the moment bound $\EE|\bar{W}-V_n| \lesssim_P n^{-1/2}$ in Lemma \ref{lem:W bar moment bound low tmp}(b).
Consequently, we can write
$$\Var^{\QQ_{\xxi_n}} (W_i) = 1 - \tanh^2(\beta m + \xxi_n^\top \X_i) + O_p\Big(\frac{1}{\sqrt{n}} + |V_n - m| \Big).$$
Since $\xxi_n \xp \ttheta_0$ and $V_n \xp m$, we can use the LLN to conclude that
\begin{align}\label{eq:lower bound variance A_n}
    \frac{\mathbf{C}_n}{n} &= \sumin \X_i \sech^2(\beta m + \xxi_n^\top \X_i) \X_i^\top + O_p(\sqrt{n} + n|V_n-m|) \\
    &\xp \EE_{\ttheta_0} \X \X^\top \sech^2(\beta m + \ttheta_0^\top \X) = \alpha_2. \notag
\end{align}
Recall $\alpha_2$ from part (d) of \cref{defi:sigma}.

\vspace{2mm}
Next, we control $\mathbf{D}_n$. Let $Y_n$ be the auxiliary random variable defined as in \cref{lem:W bar moment bound low tmp}.
For the sake of notational simplicity, we denote the variance and covariance under the conditional law $Y_n \mid \X^n$ as $\Var_{\xxi_n}^{Y}$ and $\Cov_{\xxi_n}^Y$.
For $i \neq j$, we can decompose
\begin{align}
    \Cov^{\QQ_{\xxi_n}} (W_i, W_j) &= \EE_{\xxi_n}^{Y} [\Cov (W_i, W_j \mid Y_n)] + \Cov_{\xxi_n}^{Y} [\EE (W_i \mid Y_n), \EE (W_j \mid Y_n)] \label{eq:cov} \\
    &= \Cov_{\xxi_n}^{Y} ( \tanh(\beta Y_n + \xxi_n^\top \X_i), \tanh(\beta Y_n + \xxi_n^\top \X_j) ). \notag
\end{align}
Here, the first term is exactly zero by part (a) of \cref{lem:W bar moment bound low tmp}, and the conditional expectation also follows from the same lemma.
We expand
\begin{align*}
&\tanh(\beta Y_n + \xxi_n^\top \X_i) \\
=& \tanh(\beta V_n + \xxi_n^\top \X_i) + \beta (Y_n - V_n) \sech^2(\beta V_n + \xxi_n^\top \X_i) + \frac{\beta^2 (Y_n - V_n)^2}{2} (\sech^2)'(\beta \omega_n + \xxi_n^\top \X_i)
\end{align*}
for some $\omega_i \in (V_n, Y_n)$. Here, the first term is a function of $\X^n$, and does not contribute when computing the covariance under the law $Y_n \mid \X^n$.

By plugging the expansion of $\tanh(\beta Y_n + \xxi_n^\top \X_i)$ in \eqref{eq:cov} and recalling the definition of $\mathbf{D}_n$, we have
\begin{equation}\label{eq:B_n expansion}
\begin{aligned}
     & \mathbf{D}_n = \sum_{i \neq j} \X_i  \Cov^{\QQ_{\xxi_n}}(W_i, W_j) \X_j^\top \\
    =& \beta^2  \left( \sumin \X_i \sech^2(\beta V_n + \xxi_n^\top \X_i) \right)  \Var_{\xxi_n}^{Y}(Y_n - V_n) \left( \sum_{j \neq i} \X_j \sech^2(\beta V_n + \xxi_n^\top \X_j) \right)^\top \\
    &+ \beta^3 \left(\sumin \X_i \sech^2(\beta V_n + \xxi_n^\top \X_i)\right) \left(\sum_{j \neq i} \X_j \Cov_{\xxi_n}^{Y} (Y_n - V_n, (Y_n - V_n)^2 (\sech^2)'(\beta \omega_n + \xxi_n \X_j)) \right)^\top \\
    &+ \beta^3 \left(\sumin \X_i \Cov_{\xxi_n}^{Y} (Y_n - V_n, (Y_n - V_n)^2 (\sech^2)'(\beta \omega_n + \xxi_n \X_i)) \right) \left( \sum_{j \neq i} \X_j \sech^2(\beta V_n + \xxi_n^\top \X_j) \right)^\top \\
    &+ O_p \left(n^2 \EE_{\xxi_n}^{Y}(Y_n - V_n)^4 \right).
\end{aligned}
\end{equation}
Note that Lemma \ref{lem:W bar moment bound low tmp} gives the following bounds: %
$$n \Var_{\xxi_n}^{Y}(Y_n - V_n) \xp \frac{1}{\beta(1 - \beta \alpha_0)}, ~~\EE_{\xxi_n}^{Y}(Y_n - U_n)^4 = O_p \Big(\frac{1}{n^2}\Big),$$
and
$$\Cov_{\xxi_n}^{Y}(Y_n - U_n, (Y_n - U_n)^2 (\sech^2)'(\beta V_n + \omega_j \X_j)) 
\lesssim \EE_{\xxi_n}^{Y} |Y_n - U_n|^3 = O_p\Big( \frac{1}{n \sqrt{n}}\Big).$$
Hence, only the first term in \eqref{eq:B_n expansion} contributes for $\mathbf{D}_n/n$. Because $\xxi_n \xp \ttheta_0$ and Lemma \ref{lem: u n consistency} gives $V_n \xp m$, we can apply the LLN in Lemma \ref{lem:ULLN} to write
\begin{align*}
    \frac{\sumin \X_i \sech^2(\beta V_n + \xxi_n^\top \X_i)}{n} \xp \EE_{\theta_0} \X \sech^2(\beta m + \ttheta_0^\top \X) = \alpha_1.
\end{align*}
Thus, we have
\begin{align}
    &\frac{\mathbf{D}_n}{n} \label{eq:lower bound variance B_n} \\
    =& \left( \frac{\beta \sumin \X_i \sech^2(\beta V_n + \xxi_n^\top \X_i)}{n} \right)  n\Var_{\xxi_n}^{Y}(Y_n - V_n) \left( \frac{\beta \sumin \X_i \sech^2(\beta V_n + \xxi_n^\top \X_i)}{n} \right)^\top + O_p(\frac{1}{\sqrt{n}}) \notag \\
    \xp& \frac{\beta^2 \alpha_1 \alpha_1^\top}{\beta(1 - \beta \alpha_0)} = \frac{\ggamma_{1,2} \ggamma_{1,2}^\top}{\gamma_{1,1}}. \notag
\end{align}

To conclude Claim 3, we sum up \eqref{eq:lower bound variance A_n} and \eqref{eq:lower bound variance B_n} to get
$$\lim_{n \to \infty} \frac{\Var_{\xxi_n} (\sumin \X_i W_i)}{n} \to \alpha_2 + \frac{\ggamma_{1,2} \ggamma_{1,2}^\top}{\gamma_{1,1}} = \I - I_{\beta}(\ttheta_0).$$
For the last equality, we are using the definition of $I_{\beta}(\ttheta_0)$ and the fact that $\ggamma_{2,2} = \I - \alpha_2$.
\end{proof}

Finally, we prove Corollary \ref{cor:low tmp} via the same line of arguments as in Corollary \ref{cor:high tmp}.
\begin{proof}[Proof of Corollary \ref{cor:low tmp}]
    Recalling the expansion $\tilde{\Delta}_{n, \ttheta_0, \beta}= -\sqrt{n}(- \ddelta F_1(m, \ttheta_0) + F_2(m, \ttheta_0)) + o_p(1)$
    from \eqref{eq:Delta expansion}
    and that for $\sqrt{n} (\thetamf - \ttheta_0)$ from \eqref{eq:upper bound expansion low tmp}, we can write
    $$\sqrt{n} (\thetamf - \ttheta_0) = I_\beta(\ttheta_0)^{-1} \tilde{\Delta}_{n, \ttheta_0, \beta} +o_p(1).$$
    Using the LAN expansion and the limiting distribution of $\tilde{\Delta}_{n, \ttheta_0, \beta}$ in Theorem \ref{thm:lower bound}, we have
    $$ \begin{pmatrix}
        \vspace{2mm}
        \sqrt{n} (\thetamf - \ttheta_0) \\
        \log \frac{d P_{\ttheta_n}}{d P_{\ttheta_0}}
    \end{pmatrix} \xrightarrow[P_{\ttheta_0,\beta}^\text{CW}]{d} N_{d+1} \left(\begin{pmatrix}
        \mathbf{0}_d \\
        - \frac{1}{2}\bh^\top I_\beta(\ttheta_0) \bh
    \end{pmatrix}, \begin{pmatrix}
        I_\beta(\ttheta_0)^{-1} & \bh \\
        \bh^\top & \bh^\top I_\beta(\ttheta_0) \bh
    \end{pmatrix} \right).$$
    By Le Cam's first Lemma, $P_{\ttheta_n}$ and $P_{\ttheta_0}$ are mutually contiguous, and Le Cam's third Lemma gives
    $$\sqrt{n}(\thetamf - \ttheta_0) \xrightarrow[P_{\ttheta_n,\beta}]{d} N_d(\bh, I_\beta(\ttheta_0)^{-1}).$$
    Hence, $\thetamf$ is regular.

\end{proof}

\subsection{Proof of auxiliary lemmas}\label{sec:proof of lemmas}
\subsubsection{Proof of the conditional ULLN and CLTs}\label{subsec:ulln clt proof}

We first prove the conditional ULLNs in Lemmas \ref{lem:ULLN high} and \ref{lem:ULLN} together. For simplicity, we only prove the claims where $f$ takes values in $\mathbb{R}$. Here, the main idea is to decompose $$\sumin f(\X_i, \psi) = \sumin \left[f(\X_i, \psi) - \EE[f(\X_i, \psi) \mid Z_i]\right] + \sumin \EE[f(\X_i, \psi) \mid Z_i],$$ this decomposition will appear again for proving other lemmas as well. The first term of the RHS concentrates due to the conditional independence of $\X_i \mid \Z^n$. Under the setting of \cref{lem:ULLN high} (with an even function $f$), the second term becomes exactly zero. Under the low temperature setting of \cref{lem:ULLN}, the second term boils downs to controlling $\bar{Z}$, and we use the following CLT for $\bar{Z}$.

    \begin{lemma}[Thm 1.2 in \citep{deb2023fluctuations}]\label{lem:Ising CLT}
        Suppose $\beta > 1$, $\A_n$ is mean-field, approximately regular, and well-connected. Then, for $\Z^n \sim \QQ_{0,\beta,\A_n}$, we have
        $$\sqrt{n}(\bar{Z} - m) \mid (\bar{Z} > 0) \xd N(0, C(\beta)).$$
        Here, the constant $C(\beta)$ is defined in part (c) of \cref{defi:sigma}.
    \end{lemma}

\begin{proof}[Proof of Lemmas \ref{lem:ULLN high} and \ref{lem:ULLN}]
For notational simplicity, fix $\ttheta_0$ and omit the dependence of $\ttheta_0$ in the constants $C_a = C_a(\ttheta_0)$ that will appear throughout this proof.
Throughout this proof, let $m^\star \in [0,1]$ be any fixed constant, and let $\EE_{\ttheta_0}^\star$ be the expectation with respect to $\PP_{\ttheta_0}^\star := \frac{1+m^\star}{2} N_d(\ttheta_0, \I) + \frac{1-m^\star}{2} N_d(-\ttheta_0, \I)$.
Let $g(z, \psi) := \EE [f(\X, \psi) \mid Z = z]$, where the expectation is taken under the distribution $\X \mid (Z=z) \equiv N_d(\ttheta_0 z, \I)$. Using these notations, we can write $\EE_{\ttheta_0}^\star f(\X, \psi) = \frac{1+m^\star}{2}g(1,\psi) + \frac{1-m^\star}{2} g(-1, \psi).$ 
Now, by centering each $f(\X_i, \psi)$ by its conditional mean given $Z_i$ (i.e. $g(Z_i, \psi)$), we can decompose 
    \begin{equation}\label{eq:ULLN bound}
    \begin{aligned}
        &\sup_{\psi \in \Psi} \left| \frac{1}{n} \sumin f(\X_i, \psi) - \EE_{\ttheta_0}^\star f(\X, \psi) \right| \\
        \le & \sup_{\psi \in \Psi} \left| \frac{1}{n} \sumin \Big[ f(\X_i, \psi) - g(Z_i, \psi) \Big] \right| + \frac{\left|\bar{Z} - m^\star \right|}{2} \sup_{\psi \in \Psi} 
         |g(1,\psi) - g(-1, \psi)|.
    \end{aligned}
    \end{equation}

Note that the LHS of \eqref{eq:ULLN bound} is exactly the LHS of Lemmas \ref{lem:ULLN high} and \ref{lem:ULLN}, by taking $m^\star = 0$ and $m^\star = m(\beta)$ respectively. 
We first establish a conditional concentration inequality that holds for any distribution of $\Z^n$ and $m^\star$, under the three conditions in \cref{lem:ULLN high}:
\begin{align}\label{eq:LLN conditional concentration}
    \PP \left(\sup_{\psi \in \Psi} \left| \frac{1}{n} \sumin \Big[f(\X_i, \psi) - g(Z_i, \psi)\Big] \right| > \epsilon \mid \Z^n \right) \lesssim \frac{n^{-\frac{1}{k+1}}}{\epsilon^2}, \quad \forall \epsilon > 0.
\end{align}
Here, the constants in $\lesssim$ only depend on $C_1, C_2, C_3$ from the statement of \cref{lem:ULLN high}.
This follows a standard uniform concentration argument for independent random variables, and we postpone the formal proof to the end of the current proof. Assuming \eqref{eq:LLN conditional concentration}, we separately prove Lemmas \ref{lem:ULLN high} and \ref{lem:ULLN}.
\vspace{2mm}

    \noindent \textbf{Proof of \cref{lem:ULLN high}.} Suppose that $f$ is even in $\psi$: $f(\x, \psi) = f(-\x, \psi)$. Then, $\EE_{\ttheta_0} f(\x, \psi)$ is invariant for the choice of $m^\star$, and we can simply take $m^\star = 0$. Since
    $$g(1, \psi) = \EE f(\ttheta_0 + N_d(\mathbf{0}_d, \I), \psi) = \EE f(-\ttheta_0 - N_d(\mathbf{0}_d, \I), \psi) = g(-1, \psi),$$
    the second term in the RHS of \eqref{eq:ULLN bound} is exactly zero.
    Hence, we have
    \begin{equation}
        \PP \left(\sup_{\psi \in \Psi} \left| \frac{1}{n} \sumin \left( f(\X_i, \psi) -\EE_{\ttheta_0} f(\X, \psi) \right) \right| > \epsilon \mid \Z^n \right) \lesssim \frac{n^{-\frac{1}{k+1} }}{\epsilon^2} \to 0.
    \end{equation}

     \noindent \textbf{Proof of \cref{lem:ULLN}.}
     Here, we work under $\beta > 1$, and take $m^\star = m(\beta)$. Without the loss of generality, we only prove the results conditioned on $\bar{\X} \in \Theta_1$. For any fixed $\epsilon> 0$,  set $$A_n := \Big\{ \sup_{ \psi \in \Psi} \left| \frac{1}{n} \sumin \left( f(\X_i, \psi) -\EE_{\ttheta_0} f(\X, \psi) \right) \right| > \epsilon \Big\},$$ and prove that     
     $$\PP \left(A_n: (\bar{\X} \in \Theta_1) \right) \xp 0.$$

     To control the second term in \eqref{eq:ULLN bound}, note that $|g(z, \psi)|\le C_2$ for all $\psi \in \Psi$ and $z = \pm 1$, so 
    $\sup_{\psi \in \Psi} |g(1,\psi) - g(-1, \psi)| \le 2 C_2$. Hence, by using the deterministic inequality \eqref{eq:ULLN bound} and the bound \eqref{eq:LLN conditional concentration}, we have
    \begin{equation}\label{eq:ULLN bound updated}
        \PP \left(A_n \mid \Z^n \right) \lesssim \frac{n^{-\frac{1}{k+1} }}{\epsilon^2} + \mathbf{1} \left(|\bar{Z} - m| > \frac{\epsilon}{2C_2} \right).
    \end{equation} 
    Since
    $$\PP(A_n : (\bar{\X} \in \Theta_1)) \le \underbrace{\PP(A_n \cap (\bar{Z} > 0) : (\bar{\X} \in \Theta_1) )}_{:=(I)} + \underbrace{\PP(\bar{Z} < 0 :(\bar{\X} \in \Theta_1) ) }_{:=(II)},$$
    it suffices to show that both terms are $o_p(1)$.

    Noting that $\PP(\bar{\X} \in \Theta_1) = 1/2$, we bound $(I)$ by
    $$(I) = \frac{\PP(A_n \cap (\bar{Z} > 0) \cap (\bar{\X} \in \Theta_1))}{\PP(\bar{\X} \in \Theta_1)} \le 2\PP(A_n \cap (\bar{Z} > 0)) = \PP(A_n \mid (\bar{Z} > 0)).$$
    It suffices to show $\PP(A_n \mid (\bar{Z} > 0)) \xp 0.$
    But, this is immediate by taking a further expectation on \eqref{eq:ULLN bound updated}, which gives
    $$\PP(A_n \mid (\bar{Z} > 0)) \lesssim \frac{n^{-\frac{1}{k+1} }}{\epsilon^2} + \PP \left(|\bar{Z} - m| > \frac{\epsilon}{2 C_2} \mid (\bar{Z} > 0) \right) \xp 0.$$
    The last convergence follows the follows from \cref{lem:Ising CLT}.

    For $(II)$, it suffices to show that $\PP( \bar{Z} < 0, \bar{\X} \in \Theta_1) \to 0$. %
    Note that
    \begin{align*}
        \PP(\bar{Z} < 0, \bar{\X} \in \Theta_1) &\le \PP(\bar{Z} < - \frac{m}{2}, \bar{\X} \in \Theta_1) + \PP( - \frac{m}{2} < \bar{Z} < 0) \\
        &= \PP(\bar{Z} < - \frac{m}{2}) \PP\Big(\bar{\X} \in \Theta_1 \mid \bar{Z} < - \frac{m}{2}\Big) + \PP( - \frac{m}{2} < \bar{Z} < 0).
    \end{align*}
    Since $\bar{\X} \mid \bar{Z} \equiv \ttheta_0 \bar{Z} + N_d(\mathbf{0}_d, \frac{1}{n} \I)$, $\bar{\X} \mid \bar{Z}$ concentrates around $\ttheta_0 \bar{Z} \in \Theta_2$ and the first term goes to $0$. The second term goes to 0 again by \cref{lem:Ising CLT}. %

    \vspace{2mm}
    \emph{Proof of \eqref{eq:LLN conditional concentration}.} Since $\Psi$ is compact in $\mathbb{R}^k$, we can let $\mathcal{N} := \{\psi_1, \ldots, \psi_{|\mathcal{N}|}\}$ be a $\delta$-net of $\Psi$ with $|\mathcal{N}| \le \left(\frac{3}{\delta}\right)^k$ (see Corollary 4.2.13 in \cite{vershynin2018high} for the existence of a such net). Then, for any $\psi$ such that $\|\psi - \psi_t\| \le \delta$, we have
    $$|\sumin [f(\X_i, \psi) - f(\X_i, \psi_t)]| \le \| \psi - \psi_t \| \sumin h(\X_i) \le \delta \sumin h(\X_i).$$
    Similarly, since $\lVert  \frac{\partial g(Z_i, \psi)}{\partial \psi} \rVert = \lVert \EE [\frac{\partial f(\X, \psi)}{\partial \psi} \mid Z = Z_i]\rVert \le \EE [h(\X) \mid Z = Z_i] \le C_3$, for $\|\psi - \psi_t\| \le \delta$, we have
    $$\left| \sumin (g(Z_i, \psi_t) - g(Z_i, \psi)) \right| \le C_3 n \delta.$$
    Consequently, 
    \begin{align*}
        &\sup_{\psi \in \Psi} \left| \frac{1}{n} \sumin \left( f(\X_i, \psi) -g(Z_i, \psi) \right) \right| \\
        \le &\sup_{\psi \in \Psi} \Bigg( \frac{1}{n}\left|  \sumin \left( f(\X_i, \psi) - f(\X_i, \psi_t) \right)\right| + \frac{1}{n}\left|  \sumin \left( f(\X_i, \psi_t) - g(Z_i, \psi_t) \right) \right| \\
        & \quad \quad + \frac{1}{n}\left| \sumin \left( g(Z_i, \psi_t) - g(Z_i, \psi) \right) \right|\Bigg) \\
        \le & \max_{t \le |\mathcal{N}|} \left(\frac{\delta}{n} \sumin h(\X_i) + \frac{1}{n}\left|  \sumin \left( f(\X_i, \psi_t) - g(Z_i, \psi_t) \right) \right| + C_3 \delta \right) \\
        = & \frac{\delta}{n} \sumin h(\X_i) + \max_{t \le |\mathcal{N}|} \frac{1}{n}\left|  \sumin \left( f(\X_i, \psi_t) - g(Z_i, \psi_t) \right) \right| + C_3 \delta.
    \end{align*}
    Fix $\epsilon > 0$. Since $f(\X_i, \psi)$'s are independent conditioned on $\Z^n$, we can bound 
    \begin{align*}
    &\PP \left(\sup_{\psi \in \Psi} \left| \frac{1}{n} \sumin \left( f(\X_i, \psi) -g(Z_i, \psi) \right) \right| > \epsilon \mid \Z^n \right) \\
    \le & \PP \left( \max_{t \le |\mathcal{N}|} \left| \frac{1}{n} \sumin \left( f(\X_i, \psi_t) -g(Z_i, \psi_t) \right) \right| + \frac{\delta}{n} \sumin h(\X_i) + C_3 \delta > \epsilon  \mid \Z^n \right) \\
    \underset{(*)}{\le} & \sum_{t \le |\mathcal{N}|} \PP\left( |\frac{1}{n} \sumin (f(\X_i,\psi_t) - g(Z_i, \psi_t)|> \frac{\epsilon}{3} \mid \Z^n \right) +  \PP \left( \frac{1}{n} \sumin h(\X_i) > \frac{\epsilon}{3\delta} \mid \Z^n \right)  \\
    \lesssim & \sum_{t \le |\mathcal{N}|} \frac{\sumin \Var(f(\X_i, \psi_t) \mid Z_i)}{\epsilon^2 n^2} + \frac{\delta}{n\epsilon} {\sumin \EE (h(\X_i) \mid Z_i)} \\
    {\le} & \frac{|\mathcal{N}|}{\epsilon^2 n} + \frac{\delta}{\epsilon} \le \frac{1}{\epsilon^2 n \delta^k} + \frac{\delta}{\epsilon}.
    \end{align*}
    The inequality $(*)$ holds for $\delta$ such that $C_3 \delta \le \frac{\epsilon}{3}$. We take $\delta := n^{-\frac{1}{k+1}}$ so that $(*)$ holds for large enough $n$. Then, the bound simplifies to
    $$\PP \left(\sup_{\psi \in \Psi} \left| \frac{1}{n} \sumin \left( f(\X_i, \psi) -g(Z_i, \psi) \right) \right| > \epsilon \mid \Z^n \right) \lesssim \frac{n^{-\frac{1}{k+1} }}{\epsilon^2}.$$
\end{proof}

Now, we prove the CLTs (Lemmas \ref{lem: CLT high}, \ref{lem: CLT}). \cref{lem: CLT high} directly follows by applying a conditional CLT, as the summands are identically distributed.
\begin{proof}[Proof of Lemma \ref{lem: CLT high}]
    Note that $\X_1 \tanh(\ttheta_0^\top \X_1) \mid Z_1$ is identically distributed for $Z_1 = \pm 1$. Thus, the mean and variance of $\X_i \tanh(\ttheta_0^\top \X_i) \mid Z_i$ are deterministic. In particular, using Lemma \ref{lem:uniqueness iid}, we can compute
    $$\EE [\X_1 \tanh(\ttheta_0^\top \X_1) \mid Z_1 = z] = \ttheta_0, ~ \mbox{for all } z = \pm 1.$$ 
    Also, noting that that $\tanh^2(y) + \sech^2(y) = 1$, we have
    $$\Var [\X_1 \tanh(\ttheta_0^\top \X_1) \mid Z_1 = z] = I_0(\ttheta_0), ~ \mbox{for all } z = \pm 1.$$
    Now, the conditional CLT for sums of IID random variables (e.g. see \cite{bulinski2017conditional}) gives 
    $$\sqrt{n} \left(\frac{1}{n} \sumin \X_i \tanh(\ttheta_0^\top \X_i) - \ttheta_0 \right) \mid \Z^n \xd N_d(0, I_0(\ttheta_0)).$$
    The proof is complete, since conditional convergence implies marginal convergence.
\end{proof}

\cref{lem: CLT} is more challenging to prove, as (a) the summands are not identically distributed and (b) we are claiming a statement conditional on $\bar{\X} \in \Theta_1$. We address issue (a) by splitting the summand into two terms, similar to the strategy for proving the ULLN in \cref{lem:ULLN}.
To resolve issue (b), we use the following Lemma to condition on an easier event, which we prove at the end of this subsection.
\begin{lemma}\label{lem:X > 0, Z > 0}
    Under the setting of \cref{lem: CLT}, let $E_n$ be an event that depends on $\X^n, \Z^n$. Then, $$\lim_{n \to \infty} |\PP(E_n, \bar{\X} \in \Theta_1) - \PP(E_n, \bar{Z}>0)| \to 0.$$ 
    Furthermore, for a $\X^n$-measurable random variable $Y_n$ such that $Y_n \mid (\bar{Z} > 0) \xd W$, we have $Y_n : (\bar{\X} \in \Theta_1) \xd W$.
\end{lemma}
We also need the following lemma to sum up two limiting distributions.
\begin{lemma}\label{lem:converging together}
    Let $A_n, B_n$ be random variables, and let $\mathcal{F}_n, \mathcal{G}_n$ be $\sigma$-algebras such that $\mathcal{G}_n \subseteq \mathcal{F}_n$ for each $n$. Assuming that
    $$A_n \mid \mathcal{F}_n \xd N(0, 1), \quad B_n \mid \mathcal{G}_n \xd N(0, \tilde{\tau}),$$
    we have $A_n + B_n \mid \mathcal{G}_n \xd N(0, 1+\tilde{\tau}).$
\end{lemma}
\noindent The proof of this lemma follows from standard arguments using characteristic functions and tower property (see e.g. Lemma A.13 in \cite{lee2025clt}).

\begin{proof}[Proof of Lemma \ref{lem: CLT}]
 We first prove the result conditioned on $\bar{\X} \in \Theta_1$. Write $\mathbf{a} := (a_1, \mathbf{a}_2^\top)^\top \in \mathbb{R}^{d+1}$, where $a_1, \mathbf{a}_2$ is a scalar and $d$-dimensional vector, respectively. Recall the notation $\nabla M_n = \begin{pmatrix}
     F_1 \\ F_2
 \end{pmatrix}$ from part (b) of \cref{defi:sigma}.
 By the Cramer-Wold device, it suffices to show the one-dimensional convergence
 \begin{align}\label{eq:clt wts}
     \sqrt{n} \mathbf{a}^\top \nabla M_n(m, \ttheta_0) : (\bar{\X} \in \Theta_1) \xd N\left(0, \mathbf{a}^\top \Sigma \mathbf{a} \right)
 \end{align}
 holds for all $\mathbf{a}$. For this goal, fix any $ \mathbf{a}$ and define the function $$f(\bx) := \mathbf{a}^\top \nabla M_n(m, \ttheta_0) =  - a_1 \beta \tanh(\beta m + \ttheta_0^\top \bx) - \mathbf{a}_2^\top \bx \tanh(\beta m + \ttheta_0^\top \bx).$$
 Here, we omit the dependence on $\mathbf{a}$ for convenience.
 Using the notation $f$ and the identity \eqref{eq:m,theta identies}, the statement \eqref{eq:clt wts} simplifies to
 \begin{align}\label{eq:clt wts simplified}
     \sqrt{n} \left(\frac{1}{n} \sumin f(\X_i) - \EE_{\ttheta_0} f(\X) \right) : (\bar{\X} \in \Theta_1) \xd N \left(0, \mathbf{a}^\top \Sigma \mathbf{a} \right).
 \end{align}

To prove \eqref{eq:clt wts simplified}, we first introduce some additional notations. For $z = \pm 1$, define
\begin{align*}
    g_z &:= \EE [f(\X) \mid Z=z], \\
    \tau_z &:= \Var(f(\X) \mid Z = z), \\
    \tau &:= \EE_{Z \sim \text{Rad}(\frac{1+m}{2})} [\Var(f(\X) \mid Z = z)] = \frac{1+m}{2} \tau_1 + \frac{1-m}{2} {\tau_{-1}}.
\end{align*}
By adding and subtracting the conditional means $\EE[f(\X_i) \mid Z_i]$, the LHS of \eqref{eq:clt wts simplified} becomes
    \begin{equation}\label{eq:-f_n(theta)}
    \begin{aligned}
        \frac{1}{\sqrt{n}} \sumin (f(\X_i) - \EE_{\ttheta_0} f(\X)) &= {\frac{1}{\sqrt{n}} \sumin (f(\X_i) - \EE [f(\X_i) \mid Z_i ])}  + \sqrt{n} (\bar{Z} - m) \frac{g_1 - g_{-1}}{2}.
    \end{aligned}
    \end{equation}
    
    Now, we control each term separately.    
    Define
    \begin{align*}
        A_n &:= \frac{ \frac{1}{\sqrt{n}} \sumin [f(\X_i) - \EE (f(\X_i) \mid Z_i )] }{\sqrt{\frac{1}{n} \sumin \Var(f(\X_i) \mid Z_i)}} = \frac{ \frac{1}{\sqrt{n}} \sumin [f(\X_i) - \EE (f(\X_i) \mid Z_i )] }{\sqrt{\tau + \frac{\bar{Z}-m}{2} (\tau_1 - \tau_{-1})}}, \\
        B_n &:= \frac{\sqrt{n} (\bar{Z} - m) \frac{g_1 - g_{-1}}{2}}{\sqrt{\tau + \frac{\bar{Z}-m}{2} (\tau_1 - \tau_{-1})}},
    \end{align*}
    so that the LHS of \eqref{eq:clt wts simplified} is equal to $(A_n + B_n) \sqrt{\tau + \frac{\bar{Z}-m}{2} (\tau_1 - \tau_{-1})}$.
    Since $\X^n$ is independent given $\Z^n$, the conditional CLT gives
    $$A_n \mid \Z^n \xd N(0,1).$$
    As this statement is true for any distribution $\Z^n$, the tower property gives
    \begin{equation*}%
        A_n \mid \bar{Z}, (\bar{Z} > 0) \xd N(0,1).
    \end{equation*}
    Next, the limiting distribution of $B_n$ can be derived using \cref{lem:Ising CLT}:
    \begin{equation*}%
        B_n \mid (\bar{Z} > 0) \xd \frac{g_1 - g_{-1}}{2 \sqrt{\tau}} \times N \left(0, C(\beta) \right) \equiv N(0, \tilde{\tau}),
    \end{equation*}
    where $\tilde{\tau} := \frac{(g_1 - g_{-1})^2 C(\beta)}{4 \tau}$ denotes the limiting variance. Note that we have used Slutsky's theorem alongside the following limit for the denominator of $B_n$:
    $$\tau + \frac{\bar{Z}-m}{2} (\tau_1 - \tau_{-1}) \mid (\bar{Z} > 0) \xp \tau.$$

    Now, we combine the above limits for $A_n$ and $B_n$ via \cref{lem:converging together}, which gives %
    the CLT for $ A_n + B_n$:
    $$ A_n + B_n \mid (\bar{Z} > 0) \xd N\left(0, 1 + \tilde{\tau} \right).$$
    By again using Slutsky's theorem to simplify the denominator, we have
    $$\frac{1}{\sqrt{n}} \sumin (f(\X_i) - \EE_{\ttheta_0} f(\X)) \mid (\bar{Z} > 0) \xd N(0, \tau (1+\tilde{\tau})).$$
   Here, we can change the event being conditioned on from $\bar{Z} > 0$ to $\bar{\X} \in \Theta_1$ by applying Lemma \ref{lem:X > 0, Z > 0}.

    It finally remains to show that the variance matches with that in \eqref{eq:clt wts}, that is 
    $$\tau (1+\tilde{\tau}) = \mathbf{a}^\top \Sigma \mathbf{a}.$$
    By the definition of $\Sigma$ in \cref{defi:sigma}, we have
    \begin{align*}
    \mathbf{a}^\top \Sigma \mathbf{a}  &= \EE_{Z \sim \text{Rad}(\frac{1+m}{2})} \Big[\Var \Big( 
        (a_1 \beta + \mathbf{a}_2^\top \X)\tanh(\beta m + \ttheta_0^\top \X) \mid Z \Big) \Big] \\
    &+ \frac{C(\beta)}{4} \big( a_1 \beta(\mu_1 - \mu_{-1})+ \mathbf{a}_2^\top (\boldsymbol{\nu}_1 - \boldsymbol{\nu}_{-1})\big) \big( a_1 \beta(\mu_1 - \mu_{-1})+ \mathbf{a}_2^\top (\boldsymbol{\nu}_1 - \boldsymbol{\nu}_{-1})\big)^\top \\
    &= \tau + \tau \tilde{\tau}.
\end{align*}
The final equality follows by the definition of $f(\bx)$ and noting that $$g_z = - a_1 \beta \mu_z - \mathbf{a}_2^\top \boldsymbol{\nu}_z, ~\mbox{for all } z = \pm 1.$$

The result conditioned on $\bar{\X} \in \Theta_2$ follows from the same arguments, where we make the following modifications:
$$m \to -m, \quad \EE_{\ttheta_0} \to \EE_{-\ttheta_0}.$$
After these updates, the limiting variance changes from $\Sigma$ to $\tilde{\Sigma}$. We omit the details.
\end{proof}

\begin{proof}[Proof of Lemma \ref{lem:X > 0, Z > 0}]
The first part follows by using Lemma \ref{lem:ULLN} and noting that
\begin{align*}
    \PP(E_n, \bar{\X} \in \Theta_1) - \PP(E_n, \bar{Z}>0) &\le \PP(\bar{\X} \in \Theta_1, \bar{Z} < 0) \to 0, \\
    \PP(E_n, \bar{Z}>0) - \PP(E_n, \bar{\X} \in \Theta_1) &\le \PP(\bar{\X} \notin \Theta_1, \bar{Z} > 0) \to 0.
\end{align*}
The second result follows by taking $E_n := \{ Y_n \le t\}$ and noting that $\PP(\bar{Z} > 0) \to \frac{1}{2}, \PP(\bar{\X} \in \Theta_1) \to \frac{1}{2}$.
\end{proof}

\subsubsection{Proof of Lemmas \ref{lem:conditional sum} and \ref{lem:RFIM moment bound}}\label{subsec:Ising model concentration}

{
To prove \cref{lem:conditional sum}, we again decompose $\phi(\X_i)$ as $\left(\phi(\X_i) - \EE[\phi(\X_i) \mid Z]\right) + \EE[\phi(\X_i) \mid Z_i]$. Unlike the case for the ULLN and CLTs, the quantities we wish to control are \textit{quadratic} forms of $\Z^n$. Hence, we use the following Lemma, which is a standard second moment bound for quadratic forms of $\Z^n \sim \QQ_{0, \beta, \A_n}$. We postpone its proof to the end of this subsection.
}

\begin{lemma}\label{lem:Ising moment concentration}
    Suppose $\beta < 1$ and let $\Z^n \sim \QQ_{0, \beta, \A_n}$. Then, the following bounds hold.
    \begin{enumerate}[(a)]
        \item $\EE (\Z^\top \A_n \Z)^2 = O(n^2 \alpha_n^2 + n\an)$
        \item $\EE (\Z^\top \A_n^2 \Z)^2 %
        = O(n^2 \alpha_n^2)$.
    \end{enumerate}
    In particular, under \eqref{eq:mean field assumption}, the RHS of (a) and (b) can be replaced with $o(n)$. %
\end{lemma}
\begin{proof}[Proof of Lemma \ref{lem:conditional sum}]
    \begin{enumerate}[(a)]
    \item Define $m_i({\Z^n}) := \sumjn A_n(i,j) Z_j$
    and note that
    \begin{equation}
        |\sumjn A_n(i,j) \phi(\X_j)| \le |\sumjn A_n(i,j) (\phi(\X_j) - K Z_j)| + |K| |m_i({\Z^n})|.
    \end{equation}
    Since $\X^n \mid \Z^n$ is independent, we have
    \begin{align*}
        \EE \left[ |\sumjn A_n(i,j) (\phi(\X_j) - K Z_j)|^2 \mid \Z^n \right] & = \Var \left( \sumjn A_n(i,j) \phi(\X_j) \mid \Z^n \right) \\
        = \sumjn A_n(i,j)^2 &\Var( \phi(\X_j) \mid \Z^n ) \le C \sumjn A_n(i,j)^2
    \end{align*}
    for all $i$. Also, Lemma \ref{lem:Ising moment concentration}(b) gives $\EE (\Z^\top \A_n^2 \Z)^2 = \EE (\sumin m_i^2(\Z))^2 \lesssim n\an$.
    The proof is complete by summing the two bounds.
   
    \item By expanding the square and using the independence of $\X^n \mid \Z^n$, we have
    \begin{equation}\label{eq:2nd moment conditional expectation}
    \begin{aligned}
        &\EE \left[ \left( \sumij A_n(i,j) \phi_1(\X_i) \phi_2(\X_j) \right)^2 \mid \Z^n \right] \\
        =&  \EE \left[ \sum_{i,j,k,l} A_n(i,j) A_n(k,l) \phi_1(\X_i) \phi_2(\X_j) \phi_1(\X_k) \phi_2(\X_l) \mid \Z^n \right]  \\
        \lesssim& |\sum_{|\{i,j,k,l\}|=4} A_n(i,j) A_n(k,l) Z_i Z_j Z_k Z_l| + |\sum_{i=k \neq j \neq l} A_n(i,j) A_n(i,l) Z_i^2 Z_j Z_l| \\
        & \quad + |\sum_{i, j} A_n(i,j)^2 Z_i^2 Z_j^2|.
    \end{aligned}
    \end{equation}
    Here, we have omitted displaying the constants arising from moments of $\phi_1, \phi_2$, which only depend on $K, C$.
    Noting that $|Z_i| = 1$, it is easy to control the last two terms:
    \begin{align*}
        \sum_{i=k \neq j \neq l} A_n(i,j) A_n(i,l) Z_i^2 Z_j Z_l &= \sum_{i=k \neq j \neq l} A_n(i,j) A_n(i,l) Z_j Z_l = \Z^\top \A_n^2 \Z, \\
        \sum_{i, j} A_n(i,j)^2 Z_i^2 Z_j^2 &= \sum_{i, j} A_n(i,j)^2 \lesssim n\an.
    \end{align*}
    Hence, by taking an expectation over $\Z^n$ on \eqref{eq:2nd moment conditional expectation} and using Lemma \ref{lem:Ising moment concentration},
    \begin{align*}
        \EE \left( \sumij A_n(i,j) \phi_1(\X_i) \phi_2(\X_j) \right)^2 &\le \EE |\sum_{|\{i,j,k,l\}|=4} A_n(i,j) A_n(k,l) Z_i Z_j Z_k Z_l| + O(n\an) \\
        &\le \EE |\sum_{i,j,k,l} A_n(i,j) A_n(k,l) Z_i Z_j Z_k Z_l| + O(n\an) \\
        &= \EE \left(\Z^\top \A_n \Z \right)^2 + O(n\an) = O(n^2 \alpha_n^2 + n\an).
    \end{align*}
    \end{enumerate}
    \end{proof} %

Next, we prove \cref{lem:RFIM moment bound} using existing moment bounds for RFIMs. 
\begin{proof}[Proof of \cref{lem:RFIM moment bound}]
Recall the mean-field approximation of the log-partition function $\log Z_{n,\beta}^{\text{CW}}(\ttheta, \X^n)$ in \eqref{eq:gibbs variational principle}. We extend \eqref{eq:gibbs variational principle} to Ising models with general graphs $\A_n$ (see e.g. (2.4) and (2.5) in \cite{lee2024rfim}), and let $\bu \in [-1,1]^n$ be the $n$-dimensional mean-field optimizers:
$$\bu := \argmax_{\mathbf{w} \in [-1,1]^n} \left[\frac{\beta}{2} \bw^\top \A_n \bw + \ttheta^\top \sumin \X_i w_i - \sumin H(w_i)\right].$$
Here, $H$ is the binary entropy function from \eqref{eq:gibbs variational principle}. Also, set 
$$c_i:= \ttheta^\top \X_i, \quad m_i(\W^n) := \sum_{j \neq i} A_n(i,j) W_j, \quad s_i:= \sum_{j \neq i} A_n(i,j)u_j.$$ 
Under these notations, the following conclusions hold, where the hidden constants only depend on $\beta < 1$:
\begin{itemize}
    \item $u_i = \tanh(s_i + c_i),$
    \item $\EE^{\QQ_{\ttheta}} \Big[d_i(W_i-u_i)\Big]^2 \lesssim \lVert \mathbf{d} \rVert^2 (1 + n \an^2),$
    \item $\EE^{\QQ_{\ttheta}}\sumin (m_i(\W^n) - s_i)^2 \lesssim n \an,$
    \item $\sumin s_i^2 \lesssim C_1(\ttheta, \X^n)$.
\end{itemize}
Here, the first equation follows from re-writing the first order conditions of the optimization. The second, third, fourth equations follow from Theorem 2.3, Lemma 3.2(a), Lemma 3.3(a) in \cite{lee2024rfim}, respectively. Note that here we consider Ising models with $\pm 1$ valued spins, so the function $\psi_i'(c)$ in \cite{lee2024rfim} simplify to $\psi_i'(c) =\tanh(c)$. Also, note that the hidden constants in \cite{lee2024rfim} only depend on an upper bound of the operator norm of $\beta \A_n$ (see Assumption 2.1(a) in \cite{lee2024rfim}), and here we use \eqref{eq:row sum one scaling} to get the upper bound
$$\beta \|\A_n\| \le \beta \|\A_n\|_{\infty} \to \beta < 1.$$

\begin{enumerate}[(a)]
    \item This is immediate from the third and fourth bullet above:
    $$\EE^{\QQ_{\ttheta}}\sumin m_i(\W^n)^2 \le 2 \EE^{\QQ_{\ttheta}}\sumin (m_i(\W^n) - s_i)^2 + 2 \sumin s_i^2 \lesssim n\an + C_1(\ttheta, \X^n).$$

    \item For $v_i := \tanh(c_i)$, we have
    $$\Big|\sumin d_i(u_i - v_i)\Big|^2 \le \lVert \mathbf{d} \rVert^2 \sumin (u_i - v_i)^2 \le \lVert \mathbf{d} \rVert^2 \sumin s_i^2 \lesssim \lVert \mathbf{d} \rVert^2 C_1(\ttheta, \X^n).$$
    The second inequality holds since
    $$|u_i - v_i| = |\tanh(\beta s_i + c_i) - \tanh(c_i)| \le \beta |s_i| \le |s_i|,$$
    and the third inequality uses the fourth bullet point above.
    Hence, using the second bullet point, we have
    \begin{align*}
        \EE \Big[\sumin d_i(W_i- v_i )\Big]^2 &\le 2\EE \Big[\sumin d_i(W_i- u_i )\Big]^2 + 2\Big|\sumin d_i(u_i - v_i)\Big|^2 \\
        &\lesssim \lVert \mathbf{d} \rVert^2 (1 + n \an^2 + C_1(\ttheta, \X^n)).
    \end{align*}
\end{enumerate}
\end{proof}

Finally, we prove \cref{lem:Ising moment concentration} using standard arguments for Ising models.

\begin{proof}[Proof of Lemma \ref{lem:Ising moment concentration}]
\begin{enumerate}[(a)]
    \item Theorem 2.1 in \cite{gheissari2018concentration} shows that for $\beta < 1$, 
    $$\Var(\Z^\top \A_n \Z) \lesssim \|\A_n\|_F^2 \le n \an.$$
    Also, using the fact that $\EE(Z_i \mid Z_{(-i)}) = \tanh(\beta m_i(\Z))$ and $|\tanh(\beta m_i(\Z))| \le \beta |m_i(\Z)|$, we have
    \begin{align*}
        \left|\EE \left[\Z^\top \A_n \Z\right]\right| &= \left|\EE \left[ \sumin Z_i m_i(\Z) \right]\right| \\
        &= \left|\EE \left[\sumin \tanh(\beta m_i(\Z)) m_i(\Z) \right]\right| \\
        &\le \beta \EE \sumin m_i^2 (\Z) = O(n \an).%
    \end{align*}
    The last bound uses part (b) of this Lemma. The proof is complete since
    $$\EE \left[ \Z^\top \A_n \Z\right]^2 = \Var(\Z^\top \A_n \Z) + \left(\EE \left[\Z^\top \A_n \Z \right]\right)^2 \lesssim n\an + n^2 \alpha_n^2.$$
    
    \item This directly follows from existing results in the literature, such as Lemma 2.1(a) in \cite{deb2023fluctuations}, or Lemma 3.2(a) in \cite{lee2024rfim}.
\end{enumerate}
\end{proof}

\subsubsection{Proof of Lemmas \ref{lem: u n consistency}, \ref{lem:W bar moment bound low tmp}, and \ref{lem:W bar moment bound high tmp}}\label{subsec:rfim concentration}
We prove the results related to the random field Curie-Weiss model $\QQCW$. Recall notations $M_n, M_\infty$ from \cref{subsec:low temperature cw}. For notational simplicity, define the following objective function in \eqref{eq:u_n defi} and its limit:
\begin{align*}
    f_n(u) &:= M_{n}(u,\ttheta_0) = \frac{\beta u^2}{2} - \frac{1}{n} \sumin \log \cosh(\beta u + \ttheta_0^\top \X_i), \\
    f_{\infty}(u) &:= M_{\infty}(u,\ttheta_0) = \frac{\beta u^2}{2} - \EE_{\ttheta_0} \log \cosh(\beta u + \ttheta_0^\top \X).
\end{align*}

First, we prove Lemma \ref{lem: u n consistency}, by modifying standard arguments for M-estimators.

\begin{proof}[Proof of Lemma \ref{lem: u n consistency}]

We prove that $V_n : (\bar{\X} \in \Theta_1) \xp m$, which implies the statement for $U_n$. First, note that Lemma \ref{lem:uniqueness of minimizer} gives that $f_\infty(u) = M_{\infty}(u, \ttheta_0)$ is uniquely minimized at $u = m$.
Fix any $\epsilon > 0$ and let $\eta := \inf_{|u - m|>\epsilon} f_{\infty}(u) - f_{\infty}(m) > 0$. Then, we can bound
\begin{align}
\PP\Big(|V_n - m| &> \epsilon : (\bar{\X} \in \Theta_1)\Big) = \PP\Big(\min_{|u - m| > \epsilon} \ft_n(u) < \min_{|u-m| \le \epsilon} \ft_n(u) : (\bar{\X} \in \Theta_1)\Big) \notag \\
&\le \PP\Big(\min_{|u-m| > \epsilon} \ft_n(u) < \ft_n(m) : (\bar{\X} \in \Theta_1)\Big) \notag \\
&\le \PP \Big(\min_{|u-m| > \epsilon} \ft_n(u) < \ft_n(m), ~\sup_{|u| \le 1} |\ft_n(u) - f_{\infty}(u)| < \frac{\eta}{2} : (\bar{\X} \in \Theta_1)\Big) \label{eq:lem 6.5 proof} \\
&\quad + \PP \Big(\sup_{|u| \le 1} |\ft_n(u) - f_{\infty}(u)| \ge \frac{\eta}{2} : (\bar{\X} \in \Theta_1)\Big). \notag
\end{align}

To control the first term in \eqref{eq:lem 6.5 proof}, suppose that $\sup_{|u| \le 1} |\ft_n(u) - f_{\infty}(u)| < \frac{\eta}{2}$ and take any $u$ with $|u| > \epsilon$. Then, $\ft_n(u) > f_{\infty}(u) - \frac{\eta}{2} \ge f_{\infty}(m) + \frac{\eta}{2} > \ft_n(m)$, so $\min_{|u| > \epsilon} \ft_n(u) \ge \ft_n(m)$. Hence, the first term is exactly 0. 
For the second term in \eqref{eq:lem 6.5 proof}, we have
\begin{align*}
    & \PP \Big(\sup_{|u| \le 1} |\ft_n(u) - f_{\infty}(u)| \ge \frac{\eta}{2} : (\bar{\X} \in \Theta_1)\Big) \\
    \le & \PP \Big(\sup_{|u| \le 1} |\ft_n(u) - f_{n}(u)| \ge \frac{\eta}{4} : (\bar{\X} \in \Theta_1)\Big) + \PP \Big(\sup_{|u| \le 1} |f_n(u) - f_{\infty}(u)| \ge \frac{\eta}{4} : (\bar{\X} \in \Theta_1)\Big) \\
    \le & \PP \Big(\frac{\|\xxi_n - \ttheta_0\|}{n} \sumin \|\X_i\| > \frac{\eta}{4} : (\bar{\X} \in \Theta_1) \Big)+ o_p(1).
\end{align*}
In the last line, we have used triangle inequality and the bound $$|\log \cosh(\beta u + \xxi_n^\top \X_i) - \log \cosh(\beta u + \ttheta_0^\top \X_i)| \le |\xxi_n^\top \X_i - \ttheta_0^\top \X_i| \le \|\xxi_n - \ttheta_0\| \|\|\X_i\|$$ (first term), and Lemma \ref{lem:ULLN} (second term).
Finally, we use the assumption $\xxi_n(\X^n) : (\bar{\X} \in \Theta_1) \xp \ttheta_0$ to see
\begin{align*}
    &\PP \left(\frac{\|\xxi_n - \ttheta_0\|}{n} \sumin \|\X_i\| > \frac{\eta}{4} : (\bar{\X} \in \Theta_1) \right) \\
    \le& \PP \left(\frac{1}{n} \sumin \|\X_i\| > C : (\bar{\X} \in \Theta_1) \right) + \PP \left(\|\xxi_n - \ttheta_0\|  > \frac{\eta}{4C} : (\bar{\X} \in \Theta_1) \right) \to 0.
\end{align*}
Here, $C$ can be any large constant, e.g. we can take $C = \|\ttheta_0\| + 2\sqrt{d}$ so that $\frac{1}{n}\sumin \|\X_i\| \le C$ with high probability (see \eqref{eq:estimator norm bound}).
Consequently, the RHS in \eqref{eq:lem 6.5 proof} is $o_p(1)$, and the proof is complete. Note that this computation shows that $f_\infty$ is the pointwise limit of $\tilde{f}_n$, which will be used in the following proof.

\end{proof}

\begin{remark}
    With some additional effort, we can additionally show the following tighter concentrations: $\sqrt{n}(U_n - m) : (\bar{\X} \in \Theta_1) \xd N \left(0, \frac{\sigma_{1,1}}{1 - \beta \alpha_0} \right)$
    and $|V_n - U_n|: (\bar{\X} \in \Theta_1) = \Theta_p(\|\xxi_n - \ttheta_0\| ).$
\end{remark}

Now, we prove concentration for the random field Curie-Weiss model in low temperatures. The main idea is to utilize the auxiliary random variable $Y_n$ is crucial, which guarantees conditional independence of the $W_i$s.
We prove the $L^p$ bounds in parts (b) and (c) using the Laplace approximation, and prove the stronger $L^1$ bound in part (d) using the method of exchangeable pairs.

\begin{proof}[Proof of Lemma \ref{lem:W bar moment bound low tmp}]
\begin{enumerate}[(a)]
\item Recall $\W^n \mid \X^n \sim \QQ^{\text{CW}}_{\xxi_n,\beta}$, \eqref{eq:rfim def} and $Y_n \mid \X^n, \W^n \sim N(\bar{W}, 1/{n\beta})$. By the Bayes rule, we get
$$\PP(\W^n \mid \X^n, Y_n) \propto \PP(\W^n \mid \X^n) \PP(Y_n \mid \X^n, \W^n) \propto \exp \Big[\sumin W_i(\beta Y_n + \xxi^\top \X_i) \Big].$$
The conditional independence of $W_i \mid Y_n, \X^n$ is immediate from the above formula. The marginal distribution $\PP(Y_n \mid \X^n)$ also directly follows by marginalizing the below expression over $\W^n \in \{\pm 1\}^n$:
$$\PP(Y_n, \W^n \mid \X^n) \propto \exp \Big[- \frac{n\beta Y_n^2}{2} +\sumin W_i(\beta Y_n + \xxi_n^\top \X_i) \Big].$$

\item For notational simplicity, we prove the result for any deterministic $\X^n$ that satisfies $\bar{\X}\in \Theta_1$ and
\begin{equation}\label{eq:deterministic X conditions}
    V_n \to m, ~~ \frac{1}{n} \sumin \sech^2(\beta V_n + \xxi_n^\top \X_i) \to \alpha_0, ~~ \lim_n \sup_{|y| \le 2} |\ft_n''(y) - f_{\infty}''(y)| < \frac{f_{\infty}''(m)}{4}.
\end{equation}
We claim that \eqref{eq:deterministic X conditions} holds with high probability for $\X^n \sim P_{\ttheta_0,\beta}, \bar{\X} \in \Theta_1$. To elaborate, the first limit is immediate by \ref{lem: u n consistency}. The second limit follows from writing 
$$\frac{1}{n} \sumin \sech^2(\beta V_n + \xxi_n^\top \X_i) = \frac{1}{n} \sumin \sech^2(\beta m + \ttheta_0^\top \X_i) + O\Big(|V_n-m| + \frac{\|\xxi_n - \ttheta_0\|}{n} \sumin\|\X_i\| \Big)$$
and using the LLN (see \cref{lem:ULLN}) to argue that the RHS converges to $\alpha_0$. The third limit follows from identical computations as the bounds for the second term in \eqref{eq:lem 6.5 proof}.

    \vspace{2mm}
    We bound $\EE(Y_n - V_n)^2$ using the Laplace approximation of $Y_n \mid \X^n$. Since
    $$n\EE(Y_n - V_n)^2 = \frac{\sqrt{n} \int_{-\infty}^{\infty} n(y - V_n)^2 e^{-n(\ft_n(y) - \ft_n(V_n))} dy}{\sqrt{n} \int_{-\infty}^{\infty} e^{-n(\ft_n(y) - \ft_n(V_n))} dy} =: \frac{I_1}{I_2},$$
    it suffices to show $I_1 \to C_1$ and $I_2 \to C_2$ for positive constants $C_1, C_2$ with $\frac{C_1}{C_2} = \frac{1}{\beta(1-\beta \alpha_0)}$.

    By a 3rd order Taylor expansion and using $\ft_n'(V_n) = 0$, we can write
    \begin{align}\label{eq:limit of f'' Taylor}
        |\ft_n(y) - \ft_n(V_n) - \frac{(y-V_n)^2}{2} \ft_n''(V_n)| &= \left| \frac{(y-V_n)^3}{6} \ft_n'''(\upsilon_n(y)) \right| \le C_3 {|y-V_n|^3},
    \end{align}
    where $C_3 >0$ is some constant, and $\upsilon_n(y) \in (y, V_n)$.
    Also, note that assumption \eqref{eq:deterministic X conditions} implies 
    \begin{align}\label{eq:limit of f''(V)}
        \ft_n''(V_n) = \beta - \frac{\beta^2}{n} \sumin \sech^2(\beta V_n + \xxi_n^\top \X_i) \to f_{\infty}''(m) = \beta (1- \beta \alpha_0) := \frac{1}{\sigma^2}.
    \end{align} Here, $\sigma^2 > 0$ due to Lemma \ref{lem:variance simplification}(a).
    
    To bound $I_1$, we separate the integral region into 3 parts:
    $$(-\infty, \infty) = \underbrace{[-\frac{K}{\sqrt{n}},\frac{K}{\sqrt{n}}]}_{J_1} \cup \underbrace{[-2, -\frac{K}{\sqrt{n}}) \cup (\frac{K}{\sqrt{n}}, 2]}_{J_2} \cup \underbrace{(-\infty, 2) \cup (2, \infty)}_{J_3}.$$
    For $y \in J_1$, we use \eqref{eq:limit of f'' Taylor} to upper bound the exponent as
    \begin{align*}
        &\sqrt{n} \int_{J_1} n(y - V_n)^2 e^{-n(\ft_n(y) - \ft_n(V_n))} dy \\
        \le & \sqrt{n} e^{n C_3 (K/\sqrt{n})^3 } \int_{J_1} n(y - V_n)^2 e^{-\frac{n(y - V_n)^2}{2} \ft_n''(V_n)} dy \\
        =& e^{n C_3 (K/\sqrt{n})^3 } \int_{-K}^{K} z^2 e^{- \frac{z^2 \ft_n''(V_n)}{2}} dz \to \int_{-K}^{K} z^2 e^{- \frac{z^2 }{2 \sigma^2}} dz
    \end{align*}
    as $n \to \infty$.
    The third line follows by substituting $z = \sqrt{n}(y-V_n)$, and the last limit used \eqref{eq:limit of f''(V)}. Since bounding
    $\ft_n(y) - \ft_n(V_n) \le \frac{(y - V_n)^2}{2} - C_3 |y - V_n|^3$ gives the exact same lower bound, we have
    \begin{equation}\label{eq:integral J_1}
        \sqrt{n} \int_{J_1} n(y - V_n)^2 e^{-n(\ft_n(y) - \ft_n(V_n))} dy \to \int_{-K}^{K} z^2 e^{- \frac{z^2 }{2 \sigma^2}} dz.
    \end{equation}

    Now, we bound the integral for $y \in J_2$. Recall from \eqref{eq:deterministic X conditions} that for a large enough $n$, $\sup_{|y| \le 2} |\ft_n''(y) - f_{\infty}''(y)| < \frac{f_{\infty}''(m)}{4}$. Let $\eta > 0$ be a small constant such that $\sup_{|y - m| \le \eta} |f_{\infty}''(y) - f_{\infty}''(m)| \le \frac{f_{\infty}''(m)}{4}$. Then, we have
    \begin{align*}
        \sup_{|y - m| \le \eta} |\ft_n''(y) - f_{\infty}''(m)| \le \sup_{|y| \le 2} |\ft_n''(y) - f_{\infty}''(y)| + \sup_{|y - m| \le \eta} |f_{\infty}''(y) - f_{\infty}''(m)| \le \frac{f_{\infty}''(m)}{2}.
    \end{align*}
    Then, for $|y - m| \le \eta$, $\ft_n''(y) > \frac{f_{\infty}''(m)}{2}$ and a 2nd order Taylor expansion analogous to \eqref{eq:limit of f'' Taylor} gives
    $$\ft_n(y) - \ft_n(V_n) = \frac{(y-V_n)^2}{2} \ft''(\upsilon_n(y)) \ge \frac{(y-V_n)^2}{4} {f_{\infty}''(m)} = \frac{(y-V_n)^2 }{4 \sigma^2} $$
    with high probability.
    For $y \in J_2$ such that $|y - m| > \eta$, the uniqueness of the minimizer of $f_{\infty}$ (see \cref{lem:uniqueness of minimizer}) and the ULLN (see \cref{lem:ULLN}) guarantees existence of a positive $\psi$ such that $f_{n}(y) - f_{n}(m) > \psi$ for a large enough $n$.
    Hence, 
    \begin{align}
        &\sqrt{n} \int_{J_2} n(y - V_n)^2 e^{-n(\ft_n(y) - \ft_n(V_n))} dy \notag \\
        \le & \sqrt{n} \int_{{J_2 \cap \{|y-m| \le \eta\} }} n(y - V_n)^2 e^{-\frac{n (y - V_n)^2}{4\sigma^2}} dy + \sqrt{n} \int_{{J_2 \cap \{|y-m| > \eta\} }} n(y - V_n)^2 e^{- n \psi} dy \label{eq:integral J_2} \\ 
        \to & \int_{[-\infty, -K] \cup [K, \infty]} z^2 e^{-\frac{z^2}{4\sigma^2}} dz. \notag
    \end{align}
    
    For $y \in J_3$, note that $\ft_n(y)$ is increasing for $y \ge 1$ and decreasing for $y \le -1$. Then, as $V_n$ is the minimizer of $\ft_n$ in $[-1,1]$, $\ft_n(V_n) \le \ft_n(1) < \ft_n(1.5)$. For $y > 1.5$, $\ft_n'(y) = \beta (y - \frac{1}{n} \sumin \tanh(\beta y + \xxi_n^\top \X_i)) > \frac{\beta}{2}$, and we have
    $$\ft_n(y) - \ft_n(V_n) \ge \ft_n(y) - \ft_n(1.5)  \ge \frac{\beta}{2} (y-1.5).$$
    Hence,
    \begin{align}
        &\sqrt{n} \int_{2}^{\infty} n(y - V_n)^2 e^{-n(\ft_n(y) - \ft_n(V_n))} dy \notag \\
        \le & \sqrt{n} \int_{2}^{\infty} n(y - V_n)^2 e^{-\frac{n\beta}{2} (y-1.5)} dy \label{eq:integral J_3} \\
        \lesssim& n \sqrt{n} \int_{0.5}^{\infty} (z^2 +1) e^{-\frac{n\beta z}{2}} dz \to 0. \notag
    \end{align}
    The last line substitutes $z = y-1.5$, and the limit holds as the integral is exponentially small in $n$. The integral in $(-\infty, -2)$ can be bounded similarly.

    Now, we add up all bounds in \eqref{eq:integral J_1}, \eqref{eq:integral J_2}, and \eqref{eq:integral J_3} and take $K \to \infty$ to get
    $I_1 \to \int_{-\infty}^{\infty} z^2 e^{-\frac{z^2}{2\sigma^2}} = \sqrt{2 \pi}\sigma^3.$ Note that we are using the trivial lower bound of zero for \eqref{eq:integral J_2} and \eqref{eq:integral J_3}.

    \vspace{2mm}
    To compute the limit of $I_2$, we similarly divide the integral region into 3 parts. By removing the $n(y - V_n)^2$ term in the integrated and using the same bounds, we get
    $$I_2 \to \int_{-\infty}^{\infty} e^{-\frac{z^2}{2 \sigma^2}} dz = \sqrt{2 \pi \sigma^2}.$$
    Hence, 
    $$n \EE(Y_n - V_n)^2= \frac{I_1}{I_2} \to \sigma^2 = \frac{1}{\beta(1 - \beta \alpha_0)}.$$

    The bound $n^{q/2} \EE_{\xxi_n}  |Y_n- V_n|^q \lesssim 1$ can also be similarly derived by representing the expectation as 
    $$\frac{\sqrt{n} \int_{-\infty}^{\infty} |\sqrt{n}(y - V_n)|^q e^{-n(\ft_n(y) - \ft_n(V_n))} dy}{\sqrt{n} \int_{-\infty}^{\infty} e^{-n(\ft_n(y) - \ft_n(V_n))} dy}$$
    and upper bounding the numerator with Normal moments.

    \item  Using $Y_n$, we can bound
    $$\EE^{\QQ_{\xxi_n}} (\bar{W}- V_n)^2 \le  2\EE^{\QQ_{\xxi_n}} (\bar{W} - Y_n)^2 + 2\EE_{\xxi_n}^{Y}(Y_n - V_n)^2 .$$
    The second term is $O\Big(\frac{1}{n}\Big)$ by part (a). 
    The first term is also $O\Big(\frac{1}{n}\Big)$ by using the Gaussianity of $Y_n \mid \W^n, \X^n$ to write
    \begin{align*}
        \EE^{\QQ_{\xxi_n}} ((\bar{W} - Y_n)^2) = \EE^{\QQ_{\xxi_n}} ( \EE ((\bar{W} - Y_n)^2 \mid \W^n, \X^n)) = \frac{1}{n\beta}.
    \end{align*}

    \item %
    We prove this stronger bound using the method of exchangeable pairs. 
    Similar to part (b), it suffices to prove the result for any deterministic $\X^n$ that satisfies $\bar{\X} \in \Theta_1$ and \eqref{eq:deterministic X conditions}.
    For notational simplicity, write $c_i := \xxi_n^\top \X_i$ for this segment of the proof. Let $T_n := \sqrt{n} (\bar{W} - V_n)$ and $\bar{W}_{(-i)} := \frac{1}{n} \sum_{j \neq i} W_j$. Let $(\W^n, \W'^{n})$ be the exchangeable pair that results by moving one step forward in the Glauber dynamics (i.e. pick an index $I \in [n]$ uniformly at random, and for $I=i$, replace $W_i$ by a random variable $W_i'$ generated from the complete conditional of $W_i \mid (W_j, j \neq i)$) and set $T_n' := \sqrt{n}(\bar{W'}-V_n)$.

    By a Taylor expansion, we can write
    $$\EE[W_i \mid (W_j, j \neq i)] = \tanh(\beta \bar{W}_{(-i)} + c_i) = \tanh(\beta \bar{W} + c_i) - \frac{\beta W_i}{n} \sech^2(\beta \kappa_i + c_i)$$
    for some $\kappa_i$. Define an error term
    $$E_n := \frac{\beta}{n^2\sqrt{n}} \sumin W_i \sech^2(\beta \kappa_i + c_i),$$ which is bounded deterministically by $\frac{\beta}{n\sqrt{n}}$.

    Now, using the properties of exchangable pairs alongside the above Taylor expansion, we can write
    \begin{equation}\label{eq:stein expansion}
    \begin{aligned}
        \EE(T_n - T_n' &\mid \W^n, \X^n) = \frac{1}{n\sqrt{n}}\sumin W_i - \frac{1}{n\sqrt{n}} \sumin \tanh(\beta \bar{W}_{(-i)} + c_i) \\
        =& \frac{1}{n\sqrt{n}}\sumin W_i - \frac{1}{n\sqrt{n}} \sumin \tanh(\beta \bar{W} + c_i) + E_n \\
        =& \frac{1}{n\sqrt{n}}\sumin (W_i-V_n) - \frac{1}{n\sqrt{n}} \Big(\beta (\bar{W}-V_n) \sumin \sech^2(\beta V_n + c_i) \\
        &+ \frac{\beta^2}{2}(\bar{W}-V_n)^2 \sumin (\sech^2)'(\beta \eta_n + c_i) \Big) + E_n.
    \end{aligned}
    \end{equation}
    For the last equality in \eqref{eq:stein expansion}, we are doing another Taylor expansion of $\sumin \tanh(\beta \bar{W} + c_i)$ around $\bar{W} \approx V_n$ and using the first order condition of $V_n$ (recall $V_n$ was defined as the minimizer of $\ft_n(V_n)$):
    \begin{align}\label{eq:V_n FOC}
        V_n = \frac{1}{n} \sumin \tanh(\beta V_n + c_i).
    \end{align}

    By taking a further expectation on \eqref{eq:stein expansion} with respect to $\W^n \mid \X^n \sim {\QQ}_{\xxi_n}$, we have
    \begin{equation}\label{eq:stein expectation}
    \begin{aligned}
        0 &= n\sqrt{n} \EE(T_n - T_n' : \X^n) \\
        &= \EE^{\QQ_{\xxi_n}}(\bar{W}-V_n) \left(n -\beta \sumin \sech^2(\beta V_n + c_i) \right) \\
        &\quad - \frac{\beta^2}{2} \EE^{\QQ_{\xxi_n}} \left((\bar{W}-V_n)^2 \sumin (\sech^2)'(\beta \eta_n + c_i) \right) + n\sqrt{n} \EE^{\QQ_{\xxi_n}} E_n.
    \end{aligned}
    \end{equation}
    By rearranging terms, we can write
    $$n \EE^{\QQ_{\xxi_n}}[\bar{W}-V_n] = \frac{\frac{\beta^2}{2} \EE^{\QQ_{\xxi_n}} \left((\bar{W}-V_n)^2 \sumin (\sech^2)'(\beta \eta_n + c_i) \right) - n\sqrt{n} \EE^{\QQ_{\xxi_n}} E_n}{1 - \frac{\beta}{n} \sumin \sech^2(\beta V_n + c_i)}.$$
    Recalling the assumption \eqref{eq:deterministic X conditions} on $\X^n$ and $1-\beta \alpha_0>0$ (see part (a) of \cref{lem:variance simplification}), the denominator is bounded away from 0.
    For the numerator, the $L^2$ concentration bound in part (b) gives
    \begin{align*}
        \left|\EE^{\QQ_{\xxi_n}} \left((\bar{W}-V_n)^2 \sumin (\sech^2)'(\beta \xi_n + c_i) \right) \right| \lesssim &  n  \EE^{\QQ_{\xxi_n}} \left((\bar{W}-V_n)^2 \right) \lesssim 1.
    \end{align*}
    By the deterministic bound on $E_n$, we also have
    $n\sqrt{n} |\EE^{\QQ_{\xxi_n}} E_n| \lesssim 1.$
    The bound for $\EE^{\QQ_{\xxi_n}}[\bar{W}-V_n]$ follows by combining the result for the denominator and the numerator.
    The analogous statement for $Y_n - V_n$ directly follows since $Y_n - V_n = (Y_n - \bar{W}) + (\bar{W}-V_n)$ and $$\EE [Y_n - \bar{W} : \W^n, \X^n] = 0.$$
\end{enumerate}

\end{proof}

Finally, we prove Lemma \ref{lem:W bar moment bound high tmp}, which gives the second moment bound under $\beta=1$. The overall argument is very similar to the low temperature analog in Lemma \ref{lem:W bar moment bound low tmp}(b) with a distinction that now (under $\beta =1$) we have $m=m(\beta)=0$.

\begin{proof}[Proof of Lemma \ref{lem:W bar moment bound high tmp}]

Fix $\ttheta \in \Theta$ and define the auxiliary Gaussian variable $Y_n$ as in \cref{lem:W bar moment bound low tmp}.
We divide the proof into two parts. The first step shows that the mode of the likelihood for $Y_n$ concentrates around $0$, similar to \cref{lem: u n consistency}. The second step utilizes the Laplace approximation to derive the second moment bound.
\vspace{2mm}

    \textbf{Step 1.} Similar to the setup of \cref{lem: u n consistency}, define $$\ft_n(v) := \frac{\beta v^2}{2} - \frac{1}{n} \sumin \log \cosh(\beta v + \ttheta^\top \X_i), \quad V_n := \argmin_{v \in [-1,1]} \tilde{f}_n(v).$$ Then, by part (a) in \cref{lem:W bar moment bound low tmp}, the density of $Y_n \mid \X^n$ satisfies
    $$\PP(Y_n \mid \X^n) \propto e^{-n \ft_n(Y_n)}.$$

    We first claim that $V_n = O_p(\frac{1}{\sqrt{n}})$. Similar to Lemma \ref{lem: u n consistency}, consistency follows by noting that $\ft_{\infty}''$ is strictly convex, and uniquely minimized at $v = 0$. Indeed, for $\beta = 1$ and any $v \in [-1,1]$, $\ft_{\infty}''(v) = 1 - \EE \sech^2(\beta v + \ttheta^\top \X) > 0$. Then, by a Taylor expansion with $V_n \approx 0$ on the fixed-point equation \eqref{eq:V_n FOC}, we can write
    $$V_n = \frac{\frac{1}{n} \sumin \tanh(\ttheta^\top \X_i)}{1 - \frac{1}{n} \sumin \sech^2(\eta_n + \ttheta^\top \X_i)},$$
    with $\eta_n \in (0, V_n) \xp 0$. The denominator converges to a positive constant and the numerator is $O_p(\frac{1}{\sqrt{n}})$. Hence, $V_n = O_p(\frac{1}{\sqrt{n}})$.

    \vspace{2mm}
    \textbf{Step 2.} %
    Since $\ft'_n(V_n) = 0$ and $\ft''_n (V_n) \to f_{\infty}''(0) = 1 - \EE_{\X \sim N_d(\ttheta_0, \I)} \sech^2(\ttheta^\top \X) > 0$, applying the Laplace method (see part (b) in \cref{lem:W bar moment bound low tmp}) gives
    $$\EE (Y_n - V_n)^2 \lesssim_P \frac{1}{n}.$$
    Also, the definition of $Y_n$ gives $\EE(Y_n - \bar{W})^2 = \frac{1}{n\beta}$.
    Then, by combining all bounds,
    $$\EE^{\QQCW} \bar{W}^2 \lesssim \EE (\bar{W} - Y_n)^2 + \EE (Y_n - V_n)^2 + V_n^2 = O_p\left(\frac{1}{n} \right).$$
\end{proof}

\subsubsection{Proof of Lemma \ref{lem:variance simplification}}\label{subsec:proof of variance matching}
{
The following proof crucially utilizes Stein's lemma to simplify the components of $\Sigma$ and $\Gamma$ in terms of the quantities defined in \cref{defi:sigma}. The individual statements follow by plugging-in these expressions. Recall the following multivariate Stein's lemma: for $\mathbf{Y} \sim N(\boldsymbol{\mu}, \I)$ and a differentiable function $g$ where both expectations below exist, we have
\begin{align}\label{eq:stein}
    \EE g(\mathbf{Y}) (\mathbf{Y}-\boldsymbol{\mu}) = \EE \nabla g(\mathbf{Y}).
\end{align}
}
\begin{proof}[Proof of Lemma \ref{lem:variance simplification}]
Fix $\beta > 1$.
Recall that we have defined $$\alpha_k := \EE_{\ttheta_0} \X^k \sech^2(\beta m + \ttheta_0^{\top} \X), \quad k = 0,1,2$$ 
(where $\X^2$ means $\X \X^{\top}$) and $\mu_{\pm 1}, \boldsymbol{\nu}_{\pm 1}$ in \cref{defi:sigma}, and let $p = \frac{1+m}{2}$. Also, by the identities in \eqref{eq:m,theta identies} and using the definition of $\EE_{\ttheta_0}$, we have
\begin{equation}\label{eq:m identities}
\begin{aligned}
    m &= \EE_{\ttheta_0} \tanh(\beta m + \ttheta_0^{\top} \X) = p \mu_1 + (1-p) \mu_{-1}, \\
    \ttheta_0 &= \EE_{\ttheta_0} \X \tanh(\beta m + \ttheta_0^{\top} \X) = p \boldsymbol{\nu}_1 + (1-p) \boldsymbol{\nu}_{-1} , \\
        m (\I + \ttheta_0 \ttheta_0^{\top} ) &= \EE_{\ttheta_0} \X \X^{\top} \tanh(\beta m + \ttheta_0^{\top} \X).
\end{aligned}
\end{equation}

We first claim that we can write
\begin{equation}\label{eq:alpha_0}
    \alpha_0 = (1-m^2) (1 - \frac{\mu_1 - \mu_{-1}}{2}).
\end{equation}
This follows from using of Stein's lemma (see \eqref{eq:stein}) to write
\begin{align*}
    \ttheta_0 \alpha_0 &= p \ttheta_0 \EE [\sech^2(\beta m + \ttheta_0^{\top} \X) \mid Z=1] + (1-p) \ttheta_0 \EE [\sech^2(\beta m + \ttheta_0^{\top} \X) \mid Z=-1] \\
    &= p \EE [(\X - \ttheta_0) \tanh(\beta m + \ttheta_0^{\top} \X) \mid Z=1] + (1-p) \EE [(\X + \ttheta_0) \tanh(\beta m + \ttheta_0^{\top} \X) \mid Z=-1] \\
    &= p \boldsymbol{\nu}_1 - \ttheta_0 \mu_1 + (1-p) \boldsymbol{\nu}_{-1} + \ttheta_0 \mu_{-1} \\
    &= \ttheta_0 (1 - p \mu_1 + (1-p) \mu_{-1}) \\
    &= \ttheta_0 (1-m^2) (1 - \frac{\mu_1 - \mu_{-1}}{2}).
\end{align*}
The last equality follows by writing
\begin{align*}
    1 - p \mu_1 + (1-p) \mu_{-1} = 1 - \frac{\mu_1-\mu_{-1}}{2} - \frac{m(\mu_1 + \mu_{-1})}{2}
\end{align*}
and noting that the first identity in \eqref{eq:m identities} implies $\frac{\mu_1 + \mu_{-1}}{2} = m(1-\frac{\mu_1-\mu_{-1}}{2})$.

Now, we claim the following expression for $\alpha_1$:
\begin{align}\label{eq:alpha 1}
    \alpha_1 = -\frac{1-m^2}{2} (\boldsymbol{\nu}_1 - \boldsymbol{\nu}_{-1}).
\end{align}
Again by Stein's lemma in \eqref{eq:stein}, we have
$$\EE_{\X \sim N(\boldsymbol{\mu},\I)} [(\X-\boldsymbol{\mu}) \log \cosh(\beta m + \ttheta_0^\top \X)] = \ttheta_0 \EE_{\X \sim N(\boldsymbol{\mu},\I)} \tanh(\beta m + \ttheta_0^\top \X)$$
for any $\boldsymbol{\mu} \in \mathbb{R}^d$.
Taking the derivative with respect to $\ttheta_0$ gives
\begin{align*}
    &\EE_{\X \sim N(\boldsymbol{\mu},\I)} [(\X \X^\top-\boldsymbol{\mu} \X^\top) \tanh (\beta m + \ttheta_0^\top \X)] \\
    =& \I \EE_{\X \sim N(\boldsymbol{\mu},\I)} \tanh(\beta m + \ttheta_0^\top \X) + \ttheta_0 \EE_{\X \sim N(\boldsymbol{\mu},\I)} \X^{\top} \sech^2(\beta m + \ttheta_0^\top \X).
\end{align*}
By setting $\boldsymbol{\mu} = \pm \ttheta_0$ and rearranging terms, we get
\begin{align}\label{eq:x sech^2}
    \ttheta_0 \EE[ \X^{\top} \sech^2(\beta m + \ttheta_0^\top \X)\mid Z = \pm 1] = \EE[ \X \X^{\top} \tanh(\beta m + \ttheta_0^{\top} \X) \mid Z = \pm 1] - \mu_{\pm 1} \I \mp \ttheta_0 \boldsymbol{\nu}_{\pm 1}^{\top}.
\end{align}
Using \eqref{eq:x sech^2} (identity for $m(\I+\ttheta_0 \ttheta_0^\top)$ in the second line, and identity for $\ttheta_0$ in the fourth line) alongside \eqref{eq:m identities}, we can simplify
\begin{align*}
    \ttheta_0 \alpha_1^{\top} &= p \ttheta_0\EE[\X^{\top} \sech^2(\beta m + \ttheta_0^\top \X) | Z = 1 ] + (1-p) \ttheta_0 \EE[\X^{\top} \sech^2(\beta m + \ttheta_0^\top \X) | Z = -1 ] \\
    &= {\EE [\X \X^{\top} \tanh(\beta m + \ttheta_0^\top \X)] - m \I} - \ttheta_0 (p \boldsymbol{\nu}_1 - (1-p) \boldsymbol{\nu}_{-1})^{\top} \\
    &= m \ttheta_0 \ttheta_0^\top - \ttheta_0 (p \boldsymbol{\nu}_1 - (1-p) \boldsymbol{\nu}_{-1})^\top \\
    &= \ttheta_0 \left( m (p \boldsymbol{\nu}_1 + (1-p) \boldsymbol{\nu}_{-1}) - (p \boldsymbol{\nu}_1 - (1-p) \boldsymbol{\nu}_{-1}) \right)^{\top} \\
    &= - \ttheta_0 \frac{(1-m^2) (\boldsymbol{\nu}_1 - \boldsymbol{\nu}_{-1})^{\top}}{2},
\end{align*}
which gives the desired expression for $\alpha_1$.
We are now ready to prove the individual statements.

\begin{enumerate}[(a)]
\item Noting that $$\mu_1 = \EE \tanh(\beta m + \ttheta_0^{\top} N_d (0,\I) + \ttheta_0^\top \ttheta_0) > \EE \tanh(\beta m + \ttheta_0^{\top} N_d (0,\I)  - \ttheta_0^\top \ttheta_0) =  \mu_{-1},$$  
\eqref{eq:alpha_0} gives $$1 - \beta \alpha_0 = 1- \beta(1-m^2) \left(1 - \frac{\mu_1 - \mu_{-1}}{2} \right) > 1 - \beta(1-m^2) > 0.$$
The last inequality holds since $C(\beta) = \frac{1 - m^2}{1 - \beta (1-m^2)} < \infty$.

\vspace{5mm}
\item We first show the equality in \eqref{eq:variance information eq}. By rearranging terms, it suffices to prove
$$\mathbf{M}:= \ssigma_{2,2}-\ggamma_{2,2} - \ssigma_{1,2} \ddelta^{\top} - \ddelta \ssigma_{1,2}^{\top} + \frac{\ggamma_{1,2} \ggamma_{1,2}^{\top}}{\gamma_{1,1}} + \ddelta \ddelta^{\top} \sigma_{1,1}  = 0.$$
For this goal, we rewrite all terms above using $\mu_{\pm 1}, \boldsymbol{\nu}_{\pm 1}, \alpha_k$'s. We first set $\tilde{C}(\beta) := C(\beta)/4$ and simplify $\gamma, \sigma$'s (recall the definition of $\Gamma$ from \cref{defi:info} and $\Sigma$ from part (c) of \cref{defi:sigma}):
\begin{align*}
    \gamma_{1,1} &= \beta (1 - \beta \alpha_0), \\
    \ggamma_{1,2} &= - \beta \alpha_1, \\
    \ggamma_{2,2} &= \I - \alpha_2, \\
    \sigma_{1,1} &= \beta^2 \left(1 - \alpha_0 - (p \mu_1^2 + (1-p) \mu_{-1}^2) + \tilde{C}(\beta) (\mu_1 - \mu_{-1})^2 \right), \\
    \ssigma_{1,2} &= \beta \left(\ttheta_0 m - \alpha_1 - (p \mu_1 \boldsymbol{\nu}_1 + (1-p) \mu_{-1} \boldsymbol{\nu}_{-1}) + \tilde{C}(\beta) (\mu_1 - \mu_{-1})(\boldsymbol{\nu}_1 - \boldsymbol{\nu}_{-1}) \right), \\
    \ssigma_{2,2} &= \I + \ttheta_0 \ttheta_0^{\top} - \alpha_2 -  (p \boldsymbol{\nu}_1 \boldsymbol{\nu}_1^\top + (1-p) \boldsymbol{\nu}_{-1} \boldsymbol{\nu}_{-1} ^\top) + \tilde{C}(\beta) (\boldsymbol{\nu}_1 - \boldsymbol{\nu}_{-1})(\boldsymbol{\nu}_1 - \boldsymbol{\nu}_{-1})^\top.
\end{align*}
Also, we can write $\ddelta = \frac{ \ggamma_{1,2}}{\gamma_{1,1}} = - \frac{\alpha_1}{1-\beta \alpha_0}.$ This is well defined since $1 - \beta \alpha_0 > 0$ by part (a).

First, note that $$\ssigma_{2,2} - \ggamma_{2,2} = \ttheta_0 \ttheta_0^\top - (p \boldsymbol{\nu}_1 \boldsymbol{\nu}_1^\top + (1-p) \boldsymbol{\nu}_{-1}\boldsymbol{\nu}_{-1}^\top) + \tilde{C}(\beta) (\boldsymbol{\nu}_1 - \boldsymbol{\nu}_{-1})(\boldsymbol{\nu}_1 - \boldsymbol{\nu}_{-1})^\top.$$
Also, noting that $\beta \ddelta \alpha_1^\top = \beta \alpha_1 \ddelta^\top =  - \frac{\ggamma_{1,2} \ggamma_{1,2}^\top}{\gamma_{1,1}},$ we can write
\begin{align*}
    & - \ssigma_{1,2} \ddelta^{\top} - \ddelta \ssigma_{1,2}^{\top} + \frac{\ggamma_{1,2} \ggamma_{1,2}^{\top}}{\gamma_{1,1}} \\
    = &-\beta (\ttheta_0 m - (p \mu_1 \boldsymbol{\nu}_1 + (1-p) \mu_{-1} \boldsymbol{\nu}_{-1})) \ddelta^\top - \beta \ddelta (\ttheta_0 m - (p \mu_1 \boldsymbol{\nu}_1 + (1-p) \mu_{-1} \boldsymbol{\nu}_{-1}))^\top \\
    & + \beta \tilde{C}(\beta) (\mu_1 - \mu_{-1}) \left((\boldsymbol{\nu}_1 - \boldsymbol{\nu}_{-1}) \ddelta^\top + \ddelta (\boldsymbol{\nu}_1 - \boldsymbol{\nu}_{-1})^\top \right) + \beta \ddelta \alpha_1^\top.
\end{align*}
For notational simplicity, let $\tilde{\ddelta} := \beta \ddelta$ and let $\mathbf{w}_{\pm 1} := \boldsymbol{\nu}_{\pm 1} - \tilde{\ddelta} \mu_{\pm 1}.$ Then by \eqref{eq:m identities}, we can write 
\begin{align}
    \ttheta_0 - \tilde{\ddelta} m &= p \boldsymbol{\nu}_1+(1-p) \boldsymbol{\nu}_{-1} - \tilde{\ddelta}(p \mu_1 + (1-p) \mu_{-1}) = p \w_1 +(1-p) \w_{-1}. \label{eq: theta0 - delta m}
\end{align}

Then, by simplifying the common quadratic terms multiplied by $p, 1-p$, and $\tilde{C}(\beta)$ in $\mathbf{M}$ (in the 1st equality), using \eqref{eq: theta0 - delta m} and rearranging quadratic forms involving $\w_1,\w_{-1}$ (in the 3rd equality), plugging-in the formula for $\tilde{C}(\beta)=C(\beta)/4$ from part (c) of \cref{defi:sigma} (in the 4th equality), and plugging-in $\tilde{\ddelta} = - \frac{\beta\alpha_1}{1-\beta \alpha_0}$ in all occurrences (in the 5th equality), we have %
\begin{align*}
    \mathbf{M} &= \ttheta_0 \ttheta_0^\top - (\tilde{\ddelta} \ttheta_0^\top + \ttheta_0 \tilde{\ddelta}^\top)m + \tilde{\ddelta} \tilde{\ddelta}^\top (1-\alpha_0) - p \w_1 \w_1^\top - (1-p) \w_{-1} \w_{-1}^\top \\
    & ~~ + \tilde{C}(\beta)(\w_1 - \w_{-1}) (\w_1 - \w_{-1})^\top + \tilde{\ddelta} \alpha_1^\top \\
    &= (\ttheta_0 - \tilde{\ddelta} m)(\ttheta_0 - \tilde{\ddelta} m)^\top + \tilde{\ddelta} \tilde{\ddelta}^\top (1 - \alpha_0 - m^2) - p \w_1 \w_1^\top - (1-p) \w_{-1} \w_{-1}^\top \\
    & ~~ +\tilde{C}(\beta)(\w_1 - \w_{-1}) (\w_1 - \w_{-1})^\top + \tilde{\ddelta} \alpha_1^\top \\
    &= \tilde{\ddelta} \tilde{\ddelta}^\top (1 - \alpha_0 - m^2) + \left( \tilde{C}(\beta) -p(1-p) \right) (\w_1 - \w_{-1}) (\w_1 - \w_{-1})^\top + \tilde{\ddelta}  \alpha_1^\top \\
    &=  \tilde{\ddelta} \tilde{\ddelta}^\top (1 - \alpha_0 - m^2) + \frac{\beta (1-m^2)^2}{4(1 - \beta (1-m^2))} (\w_1 - \w_{-1}) (\w_1 - \w_{-1})^\top + \tilde{\ddelta}  \alpha_1^\top \\
    &= - \frac{\beta \alpha_1 \alpha_1^\top}{(1- \beta \alpha_0)^2}(1 - \beta (1-m^2)) + \frac{\beta (1-m^2)^2}{4(1 - \beta (1-m^2))} (\w_1 - \w_{-1}) (\w_1 - \w_{-1})^\top.
\end{align*}

Hence, setting the RHS to zero, it suffices to prove 
\begin{align}
     \w_1-\w_{-1} = -\frac{2 (1-\beta(1-m^2)) \alpha_1}{(1-m^2) (1 - \beta \alpha_0)}. \label{eq: final simplification}
\end{align}
Recall that $\w_1 - \w_{-1} = \boldsymbol{\nu}_1 - \boldsymbol{\nu}_{-1} - \tilde{\ddelta} (\mu_1 - \mu_{-1}).$
Using \eqref{eq:alpha 1}, we have
$$\tilde{\ddelta} = - \frac{\beta \alpha_1}{1 - \beta \alpha_0} = \frac{\beta (1-m^2) (\boldsymbol{\nu}_1 - \boldsymbol{\nu}_{-1})/2}{1 - \beta \alpha_0},$$
and hence
\begin{align}
    \w_1 - \w_{-1} = (\boldsymbol{\nu}_1 - \boldsymbol{\nu}_{-1}) \left(1 - \frac{\beta (1-m^2) (\mu_1 - \mu_{-1})}{2(1 - \beta \alpha_0)} \right).
\end{align}
Again using \eqref{eq:alpha 1}, the RHS of \eqref{eq: final simplification} can be written as 
\begin{align*}
    -\frac{2 (1-\beta(1-m^2)) \alpha_1}{(1-m^2) (1 - \beta \alpha_0)} = \frac{(1 - \beta (1-m^2))(\boldsymbol{\nu}_1 - \boldsymbol{\nu}_{-1})}{1 - \beta \alpha_0}.
\end{align*}
Hence, \eqref{eq: final simplification} holds when the following scalar identity is true
$$ 1 - \frac{\beta (1-m^2) (\mu_1 - \mu_{-1})}{2(1 - \beta \alpha_0)} = \frac{1 - \beta (1-m^2)}{1 - \beta \alpha_0}.$$
By multiplying each side by $1-\beta \alpha_0$ and rearranging terms, the above is equivalent to
$$\alpha_0 = (1-m^2) (1 - \frac{\mu_1 - \mu_{-1}}{2}),$$
which was already shown in \eqref{eq:alpha_0}. This concludes the proof of $A=0$.
\vspace{2mm}

Finally, we show the positive definiteness claim in \eqref{eq:variance information eq} and finish the proof. This follows as
\begin{align*}
    &\ddelta \sigma_{1,1} \ddelta^{\top} - \ssigma_{1,2} \ddelta^{\top} - \ddelta \ssigma_{1,2}^{\top} + \ssigma_{2,2} \\
    =& \EE \left[ \Var (- \beta \ddelta \tanh(\beta m + \ttheta_0^\top \X) + \X \tanh(\beta m + \ttheta_0^\top \X) \mid Z) \right] \\
    &+ C(\beta) (\w_1 - \w_{-1})(\w_1 - \w_{-1})^\top \succ 0.
\end{align*}
The strict inequality follows from noting that the $\Var$ in the first term is positive definite for both $Z = \pm 1$, and that the second term is positive semi-definite.
\end{enumerate}
\end{proof}

\end{appendix}

\end{document}